\documentclass[]{amsart}
\RequirePackage[l2tabu, orthodox]{nag}
\usepackage{fixltx2e}
\usepackage{amsthm}
\usepackage{amsmath}
\usepackage{amssymb}
\usepackage{mathtools}
\usepackage[all]{xy}
\SelectTips{cm}{}
\usepackage{tikz-cd} 
\usepackage{url}
\usepackage{comment}
\usepackage{enumitem}

\urldef\katehomepage\url{websupport1.citytech.cuny.edu/faculty/kpoirier}
\urldef\gabehomepage\url{cgp.ibs.re.kr/~gabriel}

\DeclareMathOperator{\id}{id}

\DeclareMathOperator{\inj}{inj}
\DeclareMathOperator{\Maps}{Maps}

\newcommand{\pq}{{pq}}
\newcommand{\reparam}{\mathfrak{r}}
\newcommand{\length}[1]{\|#1\|}

\newcommand{\bigstarr}[2]{\nabla^{{#1}^{#2}}}

\newcommand{\spaced}{\mathcal{SD}}
\newcommand{\spaceduo}{\spaced^{\mathrm{u}}}
\newcommand{\gkl}{(\chi,k,\ell)}
\newcommand{\Mover}[2]{\mathcal{N}(#1;#2)}
\newcommand{\Mspace}[1]{\mathcal{N}(#1)}

\newcommand{\chains}{C_*}

\DeclareMathOperator{\straighten}{str}
\newcommand{\chid}{|\chi| d}
\newcommand{\prune}{\! \setminus \!}
\DeclareMathOperator{\identity}{Id}
\DeclareMathOperator{\inputs}{In}
\newcommand{\leaves}[1]{{L({#1})}}
\newcommand{\rcx}{r}

\DeclareMathOperator{\characteristic}{Ch}
\newcommand{\at}{\alpha}

\newtheorem*{theorem*}{Theorem}
\newtheorem{theorem}{Theorem}[section]
\newtheorem*{mainthm}{Main Theorem}
\newtheorem{lemma}[theorem]{Lemma}
\newtheorem{prop}[theorem]{Proposition}
\newtheorem{cor}[theorem]{Corollary}
\theoremstyle{definition}
\newtheorem{defi}[theorem]{Definition}
\newtheorem*{conjecture}{Conjecture}
\newtheorem{letterconjecture}{Conjecture}

\theoremstyle{remark}
\newtheorem*{remark}{Remark}

\begin{document}
\begin{abstract}
We construct a space of string diagrams, which are a type of fatgraph with some additional data, and show that there are string topology operations on the chains of the free loop space of a closed Riemannian manifold which are parameterized by the chains on the space of string diagrams. These operations are shown to recover known structure on homology of the free loop space.
\end{abstract}

\title{Chain-Level String Topology Operations}

\author{Gabriel C. Drummond-Cole}
\address{Center for Geometry and Physics, Institute for Basic Science (IBS)\\
77 Cheongam-ro, Nam-gu, Pohang-si, Gyeongsangbuk-do\\
Republic of Korea 790-784}
\email{gabriel@ibs.re.kr}
\urladdr{\gabehomepage} 
\thanks{This work was supported by IBS-R003-D1}

\author{Kate Poirier}
\address{Department of Mathematics\\
New York City College of Technology\\
300 Jay Street\\
Brooklyn, NY }
\email{kpoirier@citytech.cuny.edu}
\urladdr{\katehomepage}
\thanks{This work was partially supported by NSF RTG grant DMS-0838703}

\author{Nathaniel Rounds}
\address{
Reasoning Mind\\
2000 Bering Drive\\
Suite 300\\
Houston, TX}
\email{nathaniel.rounds@reasoningmind.org}

\maketitle
\setcounter{tocdepth}{1}
\tableofcontents 
\section*{Introduction}
String topology on a manifold $M$ is the study of natural operations on the free loop space $LM$, the space of continuous maps from circle to $M$. These operations on $LM$ are defined using the algebraic and geometric topology of $M$ itself. 

Historically, string topology operations were constructed as maps between various tensor powers of the homology groups of $LM$. Usually these operations involve a shift in degree, which we suppress throughout this introduction. These homology-level operations can be parameterized by spaces of graphs; the spaces of graphs parameterizing operations discovered so far are closely related to moduli spaces of Riemann surfaces.

In this paper, we define string topology operations that are maps between the tensor powers of the singular \emph{chains} of $LM$, rather than the homology groups of $LM$. These chain-level operations induce operations at the level of homology, including many of the previously constructed operations. 

This paper constitutes one step in a program to extend the construction of string topology operations to the most general possible setting, along the way employing those spaces of operations that yield the finest possible invariants.

\subsection*{Background}
This history of string topology is necessarily incomplete and is geared toward the contents of this paper. 

Chas and Sullivan introduced string topology operations by defining a loop product on the homology groups of the loop space
\[H_*(LM) \otimes H_*(LM) \to H_{*}(LM) \]
which gives $H_*(LM)$ the structure of a  commutative associative algebra~\cite{CS}.

A chain in $LM$ determines a chain in $M$ by evaluation of loops at their basepoints.
For a pair of chains in $LM$ whose chains of basepoints intersect transversally in $M$, Chas and Sullivan defined their chain-level loop product by concatenating loops along the intersection locus of basepoints in $M$.
Since it is not true that any two chains in $LM$ determine transversal chains in $M$, this chain-level product is only partially defined on $C_*(LM) \otimes C_*(LM)$.
However, for each pair of homology classes in $H_*(LM)$, Chas and Sullivan choose pairs of representative cycles whose chains of basepoints in $M$ do intersect transversally in order to give their loop product, which is fully defined on $H_*(LM) \otimes H_*(LM)$.

Later, Cohen and Jones took a different approach, using homotopy theory to generate an alternate definition of Chas and Sullivan's loop product~\cite{CJ}. This allowed them to avoid explicitly working with transversality, which is notoriously delicate. The space $LM \times_M LM$ of pairs of loops in $M$ whose basepoints coincide is a finite codimension subspace of the infinite dimensional manifold $LM \times LM$ of pairs of loops in $M$.
Because of this Cohen and Jones were able to use a Pontryagin--Thom construction to define a wrong-way map from $H_*(LM \times LM)$ to $H_{*}(LM \times_M LM)$, which is  a composition of two maps.
The first map in the composition is induced on homology by the map of spaces which takes $LM \times LM$ to the Thom space of a tubular neighborhood of $LM \times_M LM$ in $LM \times LM$.
The second is the Thom isomorphism from the homology of this Thom space to the homology of $LM \times_M LM$.
Cohen and Jones realized the loop product by composing this wrong-way map with the map induced on homology by concatenating loops whose basepoints coincide.

Chas and Sullivan also showed that the loop product is, in fact, part of a richer structure on $H_*(LM)$, namely that of a Batalin--Vilkovisky algebra.
The Batalin--Vilkovisky operator $\Delta$, which is a square zero unary second-order differential operator with respect to the loop product, is induced by the $S^1$ action on $LM$ given by rotating loops.
Additionally, Chas and Sullivan used the loop product to define the string bracket, which gives $H_*^{S^1}(LM)$, the $S^1$-equivariant homology of $LM$, the structure of a Lie algebra.
They later used similar methods to define the string cobracket, which, together with the string bracket, gives $H_*^{S^1}(LM, M)$, the $S^1$ equivariant homology of $LM$ relative to the subspace of constant loops, the structure of an involutive Lie bialgebra \cite{ChasSullivan:CSOTLLBHSA}.

Chas and Sullivan's operations inspired Cohen and Godin to develop a coherent set of more general operations~\cite{CG}  following the method of Cohen and Jones.
Cohen and Godin defined a kind of fatgraph called a Sullivan chord diagram and added additional structure to these diagrams to define a \emph{space} of marked metric chord diagrams.

For each point in their space, they defined an operation of the form 
\[H_*(LM)^{\otimes k} \to H_*(LM)^{\otimes \ell}\]
for some $k$ and $\ell$ determined by the diagram. They also showed that this operation depends only on the path component of the point in their space. In other words, they described natural operations on the homology of the loop space parameterized by the zeroth homology of the space of marked metric chord diagrams. Chataur later showed that there are natural operations on the homology of the loop space parameterized by the \emph{higher} homology of the space of marked metric chord diagrams~\cite{Chataur:BAST}. 

There is an embedding of Cohen and Godin's space of marked metric chord diagrams into the moduli space $\mathcal{M}$ of Riemann surfaces with parameterized boundary. Godin later extended Cohen--Godin's construction to provide natural operations on the homology of the loop space parameterized by the homology of $\mathcal{M}$~\cite{godin}. 
Kupers recently used a different model of $\mathcal{M}$ originally defined by B\"odigheimer to recover Godin's structure \cite{Kupers:CHSOMURSC, Bodigheimer}.
In fact, Cohen--Godin, Chataur, Godin, and Kupers each assembled their collections of operations into an algebraic structure, in each case some kind of field theory.

Tamanoi~\cite{Tamanoi:LCSTTHGTQFTO} showed that many of Cohen and Godin's operations are trivial. He also showed that Godin's operations coming from homology classes of $\mathcal{M}$ which are in the image of the stabilization map are trivial ~\cite{Tamanoi:SSOAT}.
The full implications of these facts for Chataur and Godin--Kupers' higher homology operations are still unclear, but 
this presents an interesting question because not much is known about the unstable homology of $\mathcal{M}$.
Given a homology class of $\mathcal{M}$, is there a manifold $M$ for which the corresponding string topology operation is nontrivial? If the answer is yes, then the homology class is an unstable class. As far as the authors know, so far no unstable homology classes have been found by these means.

Other lines of research have focused on algebraic models of the free loop space rooted in Hochschild homology~\cite{TZ,Malm:STBLS}. Some of this research has also used spaces of chord diagrams. This Hochschild approach is not taken in this paper and we will not discuss this perspective in detail.

\subsection*{Compactified and chain-level operations}

A general precept of homotopical algebra is that natural operations on the homology of something usually reflect a richer structure at the chain level, and that the passage to homology usually loses information. Following this precept, a natural conjecture is that all the algebraic structure on the homology $H_*(LM)$ of the free loop space discussed above is the shadow of an algebraic structure on the chains $C_*(LM)$. A second conjecture is that this algebraic structure on the chains is strictly richer than that on homology. A realization of the first conjecture is the subject of this paper; at this point the status of the second conjecture is still open, as far as we know.

Another natural goal is to describe the universal space of string topology operations. With this goal in mind, we expect that there are natural operations on the chains of the free loop space parameterized by the chains on a {\em compactification} of $\mathcal{M}$. We expect that these operations recover those parameterized by the non-compact spaces used by~\cite{CG,Chataur:BAST,godin,Kupers:CHSOMURSC} and also incorporate operations and relations not evident in this previous work.

These considerations were discussed by Sullivan in his survey~\cite{Sullivan:STBAPS}. There he outlined natural operations on a chain complex computing the $S^1$-equivariant homology of the loop space parameterized by a compactification of a related moduli space of Riemann surfaces. Following Sullivan, there have been multiple descriptions of chain-level operations and compactified spaces of operations. The second author, in her thesis~\cite{PoirierThesis}, described a compactified space of operations and a more detailed description of the operations on this equivariant chain model. The second and third author wrote a preprint~\cite{PR1} describing a non-equivariant version; this paper is the natural outgrowth of that preprint but uses new tools to give what we believe is a more natural and comprehensive presentation.
Recently, Irie defined a model for chains on the free loop space to deal with transversality issues~\cite{Irie:TPSTDRC,Irie:CLBVSSTDC}.
Using this model, he was able to define chain-level operations parameterized by a space of graphs called decorated cacti. However, this model does not seem to be well-suited to operations with more than one output or corresponding to surfaces of higher genus. Hingston and Wahl~\cite{HingstonWahl} have also announced results similar to ours.

\subsection*{Contents of this paper}
Following Sullivan~\cite{Sullivan:STBAPS}, our chain-level construction deals with transversality issues by using short geodesic segments, and generalizations of such segments, to join points which are nearby in $M$, but which may not actually coincide. It combines aspects of both the original Chas--Sullivan construction for transversally intersecting chains and the Cohen--Jones Pontryagin--Thom construction. This geodesic segment technique is our replacement for the concatenation of loops that literally intersect; our use of this technique is a fundamental difference between our methods and those of Chas--Sullivan and Cohen--Jones.

In this paper we define a version of chord diagrams called {\em string diagrams}, which generalizes the definitions of string diagrams appearing in \cite{PR1} and \cite{PoirierThesis}.
Roughly, string diagrams are metric graphs obtained by attaching leaves of trees to circles.

We denote the space of string diagrams by $\spaced$. 
Each path component $\spaced \gkl$ of $\spaced$ has the structure of a finite cell complex so is compact. 

We conjecture that the space of string diagrams is homotopy equivalent to the moduli space of Riemann surfaces with parameterized boundary. There is a natural equivalence relation on the space of string diagrams and we further conjecture that our space modulo equivalence has the homotopy type of a known compactification. See the next subsection for details.

The main theorem of the paper is the following.
\begin{mainthm}
There are chain-level string topology operations 
\[C_*(\spaced \gkl) \otimes C_*(LM)^{\otimes k} \to C_*(LM)^{\otimes \ell}.
\]
which realize diffuse intersection string topology and induce the classical string topology operations of~\cite{CS} and~\cite{CG}.
\end{mainthm}
The operations are defined in Section~\ref{section:stringtopology} and they are shown to recover previous work in Section~\ref{section:previouswork}. Dealing with transversality issues requires care and results in some technical definitions.

The most important technical accomplishments in the paper are the following. 
\begin{enumerate}
\item We construct an appropriate {\em diffuse intersection class} representative that we can cap chains with to model intersection. This involves carefully patching together Thom classes and ensuring that the output class is well-defined and unique. This diffuse intersection class plays the role of the Thom class in our version of the Pontryagin--Thom construction.
\item We define a generalized geodesic construction $\heartsuit$, so-called because it is the heart of our string topology construction. The map $\heartsuit$ takes a string diagram and a collection of loops in $M$ satisfying a closeness condition dependent on the string diagram and outputs a map from the string diagram to $M$. This must be done coherently in families as the string diagram varies in $\spaced$.
\end{enumerate}
Carefully defining these two constructions takes up much of this paper.

We show that our operations induce Cohen--Godin's operations and Chas--Sullivan's BV algebra on homology.
We expect that our operations should also induce the structures of Chataur, Godin, and Kupers as well.

This paper is organized as follows.

In Section \ref{section: conventions and definitions} we introduce our conventions and preliminary definitions which will be used for the definition of string diagrams given in Section \ref{section:string diagrams}.
Also in Section  \ref{section:string diagrams}, we discuss the space of string diagrams and show that it has the structure of a cell complex. 

In Section \ref{section:straightening} we show how to map a particular subgraph of a string diagram into a simplex by first mapping its leaves to the vertices. This {\em straightening} map is defined in terms of a {\em straightening} map for trees in Appendix  \ref{appendix:straightening} which maps a tree into a simplex by first mapping its leaves to the vertices.
The straightening map for string diagrams is the first map in a composition of two maps defining the map $\heartsuit$. Section \ref{section:stringtopology} is devoted to the definition of the map $\heartsuit$, our generalized geodesic construction.

In Section \ref{section:push-pull map} we fix an arbitrary cochain $W$ on the domain $S$ of the map $\heartsuit$ and use it to define a {\em push-pull} map  $\mathcal{ST}_W: \chains(\spaced ) \otimes \chains(LM)^{\otimes k} \to \chains(LM)^{\otimes \ell}$.
When the cochain $W$ is a cocycle representing the {\em diffuse intersection class}, we call $\mathcal{ST}_W$ a {\em string topology construction}. 

The diffuse intersection class is the pullback under an {\em evaluation map} of a {\em global Thom class over $\spaced$}.
In Section \ref{section:thom} we prove Theorem \ref{thm:patching} to describe how to assemble cohomology classes which are indexed by cells of $\spaced$ to get {\em global classes over $\spaced$}.
In Section \ref{section:diffuse intersection class} we assemble the global Thom class over $\spaced$ and define the evaluation map in order to define the diffuse intersection class.

In Section \ref{section:previouswork} we show that our string topology operations recover Cohen--Godin's operations and Chas and Sullivan's BV algebra structure on $H_*(LM)$.

\subsection*{Future directions}
There are several points we have left unaddressed in this paper.

First of all, there is the question of composition. Cohen and Godin described a sort of composition map on marked metric chord diagrams. It would perhaps be better to call it a composition relation, as it is not defined for all pairs and is not always unique when it is defined. They showed that this composition map induces a well-defined associative map on zeroth homology groups and that their string topology construction respected composition. 

In our context too there is a cognate composition map. Our composition map is fully defined and uni-valued. However, it is not strictly associative but only associative up to homotopy. Similarly, the chain-level string topology construction does not respect composition on the nose, but rather only up to homotopy. We could rectify this algebraic structure or use colored or $\infty$-properads to deal with this problem, but as this paper is already long and technical we decided instead to defer these questions.

Another question is that of the homotopy type of the space of string diagrams. As discussed previously, we have the following conjecture:
\begin{letterconjecture}\label{conj: homotopy type is moduli space}
The space of unoriented string diagrams is homotopy equivalent to the moduli space of Riemann surfaces with parameterized boundary. 
\end{letterconjecture}
There is an equivalence relation on the space of string diagrams, not discussed in this paper. Our constructions pass to the quotient by the equivalence relation. We also formulate the following conjecture:
\begin{letterconjecture}\label{conj: homotopy type is Bod}
The space of unoriented string diagrams modulo equivalence is homotopy equivalent to B\"odigheimer's harmonic compactification of the moduli space of Riemann surfaces with parameterized boundary~\cite{Bodigheimer}. 
\end{letterconjecture}
This compactification has appeared before in string topology, in the work of the second author as well as Kupers~\cite{PoirierThesis,Kupers:CHSOMURSC}.

There are a few aother avenues of research that should at this point be straightforward, if technical, generalizations of the techniques used in the current paper.

First, Godin's field theory in fact gives operations for both loops and paths in $M$; it is an {\em open-closed} theory.
While in this paper we focus purely on closed loops, we expect the methods that produce our string topology operations to generalize to an open-closed chain-level theory.

Second, we expect that our operations induce chain-level operations which recover Chas and Sullivan's involutive Lie bialgebra on reduced equivariant homology, which would enable us to realize Sullivan's outline from~\cite{Sullivan:STBAPS} in more detail.

Finally, we expect that we should be able to use cellular chains on our space of string diagrams to define operations on the Hochschild-homological models of the free loop space. 

\subsection*{Acknowledgements}
The authors gratefully acknowledge many helpful discussions with Dennis Sullivan. The second author would like to thank Thomas Tradler for helpful conversations and thank the IBS Center for Geometry and Physics in Korea for a lovely and conducive working environment.

\section{Conventions and definitions}\label{section: conventions and definitions}

\subsection{Conventions}
A {\em graph} is a quadruple $(V,H,s,\iota)$, where $V$ is a finite set of {\em vertices}, $H$ is a finite set of {\em half-edges}, $s$ is a surjective map from $H$ to $V$ (we call $s(h)$ the {\em source} of $h$), and $\iota$ is a fixed-point free involution of $H$. 
We call an orbit of $\iota$ an {\em edge}.
The {\em half-edge set} of a vertex $v$ is the set of half-edges with source $v$.
The {\em valence} of a vertex is the cardinality of its half-edge set.
A univalent vertex is called a {\em leaf}; a vertex which is not a leaf is called an {\em internal vertex}.
A {\em leaf half-edge} (respectively an {\em internal half-edge}) is a half-edge whose source is a leaf (respectively an internal vertex).
An {\em external edge} is an edge containing a leaf half-edge; an {\em internal edge} is an edge not containing a leaf half-edge.

A graph is {\em connected} if it admits no partition into two disjoint nonempty subgraphs. A {\em component} of a graph is a maximal connected subgraph.  
A {\em cycle graph} is a nonempty connected bivalent graph. A {\em tree} is a nonempty connected graph with no cycle subgraph. A {\em segment} is a tree with one edge. The {\em Euler characteristic} of a graph is the difference between its number of vertices and its number of edges.

A {\em cyclic order} on a finite set is a permutation of that set which is a single cycle.  We shall call a set equipped with a cyclic order a {\em cycle}.
A {\em fatgraph} $\Gamma$ is a graph $G$ together with a cyclic order of the half-edge set of each vertex of $G$.

 A {\em pseudometric (fat)graph} is a (fat)graph together with a non-negative length function $\ell$ on the set of edges.
Let $\Gamma$ be a pseudometric (fat)graph.  
The {\em length} of $\Gamma$ is the sum of the lengths of all the edges of $\Gamma$.

Let $e$ be an edge of a graph $G=(V,H, s, \iota)$ which is not a segment component.
{\em Contracting the edge} $e$ produces a graph $G/e$ which has half-edge set $H-e$ and vertex set the quotient of $V$ where the sources of $e$ are identified. The source map and involution for $G/e$ are induced by those for $G$.
If $G$ has a fat or pseudometric structure then $G/e$ inherits a fat or pseudometric structure from $G$\footnote{Many definitions do not allow contraction of edges whose sources coincide or edges of positive length. In practice we will never contract such edges.}. 

An { \em oriented edge} $\vec{e_h}$  of a graph consists of a choice of order $(h, \iota(h))$ for the two  half-edges that make up the edge $e_h = \{h, \iota(h)\}$.
The source of the first half-edge is called the {\em source}  of  the oriented edge and the source of the second half-edge is called the  {\em target} of the oriented edge.
The set $\widetilde{E}$ of oriented edges of a graph is in bijective correspondence with its set of half-edges.
Hence, $\widetilde{E}$ inherits the involution $\iota$ from the set of half-edges and $\iota(\vec{e}_h)$ is equal to $\vec{e}_{\iota(h)}$.

The set $\widetilde{E}$ also forms a double cover of the set $E$ of edges of a graph,
where the map forgets the ordering of the constituent half-edges.
If $G$ is a pseudometric graph, then we define the length of an oriented edge to be equal to the length of its underlying unoriented edge.

A {\em boundary cycle} of a fatgraph is a cycle of half-edges $h_1,\ldots, h_n$ so that $h_{i+1}$ follows $\iota(h_i)$ in the cyclic order of the half-edge set of $s(h_{i+1})$. See Figure \ref{fig:boundary cycle}.

\begin{remark}
There is a standard construction which produces an oriented surface with boundary from a fatgraph using the fatgraph structure to thicken it \cite{Strebel:QD}.
The boundary cycles of a fatgraph correspond precisely to the boundary components of this surface.
In what follows we do not use the thickening explicitly, though it is sometimes included in the figures.
\end{remark}

A {\em marking} of a boundary cycle $C$ of a fatgraph $\Gamma$ is a leaf of $\Gamma$ whose leaf half-edge appears in $C$. 
If $G$ is pseudometric, we require that the length of the external edge is zero.
A {\em partially marked fatgraph} is a fatgraph such that each boundary cycle contains at most one marking.
A {\em marked fatgraph} is a fatgraph  such that each boundary cycle contains exactly one marking.

\subsection{Preliminary definitions}\label{section:preliminary defs}

Next, we collect several less standard definitions for graphs that we will use in our main definition of Section \ref{section:string diagrams}.

Informally, a {\em vertex explosion} of a graph  pulls apart a graph at a specified vertex, making an $n$-valent vertex into $n$ distinct leaves.
\begin{defi}
Let $G=(V,H,s,\iota)$ be a graph and let $v$ a vertex of $G$. The {\em vertex explosion} of $G$ at $v$ is a new graph whose vertices are
\[
V-\{v\} \sqcup\left( \{v\}\times s^{-1}(v)\right)
.\]
The half-edge set and involution are not changed; the source map takes $h$ in $s^{-1}(v)$ to $v\times h$ and is otherwise unchanged.
\end{defi}
\begin{remark}
Up to canonical isomorphism, vertex explosion at different vertices commutes so we can unambiguously take a vertex explosion at a set of vertices. 
\end{remark}
We can also do a kind of partial vertex explosion.
\begin{defi}
Let $G=(V,H,s,\iota)$ be a graph. Let $v$ a vertex of $G$ and let $S$ be a proper subset of $s^{-1}(v)$. The {\em pruning} of $G$ at $v$ along $S$ is a new graph whose vertices are
\[
V \sqcup\left( \{v\}\times S\right)
.\]
The half-edge set and involution are not changed; the source map takes $h$ in $S$ to $v\times h$ and is otherwise unchanged.
\end{defi}
Figure \ref{fig:pruning} shows an example of pruning a tree at an internal vertex along a single half-edge.
\begin{figure}[ht]
\includegraphics{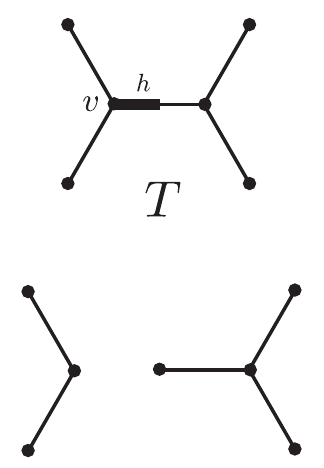}
\caption{A tree $T$ and the pruning of $T$ at $v$ along the half-edge $h$.
}
\label{fig:pruning}
\end{figure}

\begin{remark}
There is a canonical quotient map from any vertex explosion or pruning of $G$ to $G$ itself. If $G$ is fat and/or pseudometric, then the same structure is induced on a vertex explosion or pruning of $G$. 
\end{remark}

Let $U$-graphs stand for either graphs, fatgraphs, pseudometric graphs, or pseudometric fatgraphs. We will tend to abuse notation by specifying only the data of the underlying graph of a $U$-graph, eliding the fat or pseudometric structure until it is needed.

There is a {\em pseudometric realization}  $|\ \cdot\ |$ functor from pseudometric graphs to metric spaces as follows.
Consider the metric space
\[
V \sqcup \bigsqcup_{\vec{e} \in \widetilde{E}}( [0, \ell(\vec{e})] \times \{\vec{e}\})
\]
(where the distance between two points in different components is defined to be the length of $G$).
The space $|G|$ is defined as the quotient metric space 
 by the following relations.
For the oriented edge $\vec{e}$, $(0 , \vec{e})$ is identified with the source of $\vec{e}$ and $(\ell(\vec{e}) , \vec{e})$ is identified with the target of $\vec{e}$. 
Further, $( t , \vec{e})$ is identified with $((\ell(\vec{e}) -t ), \iota(\vec{e}))$. See Figure \ref{fig:boundary cycle}.

\begin{figure}[ht]
\includegraphics{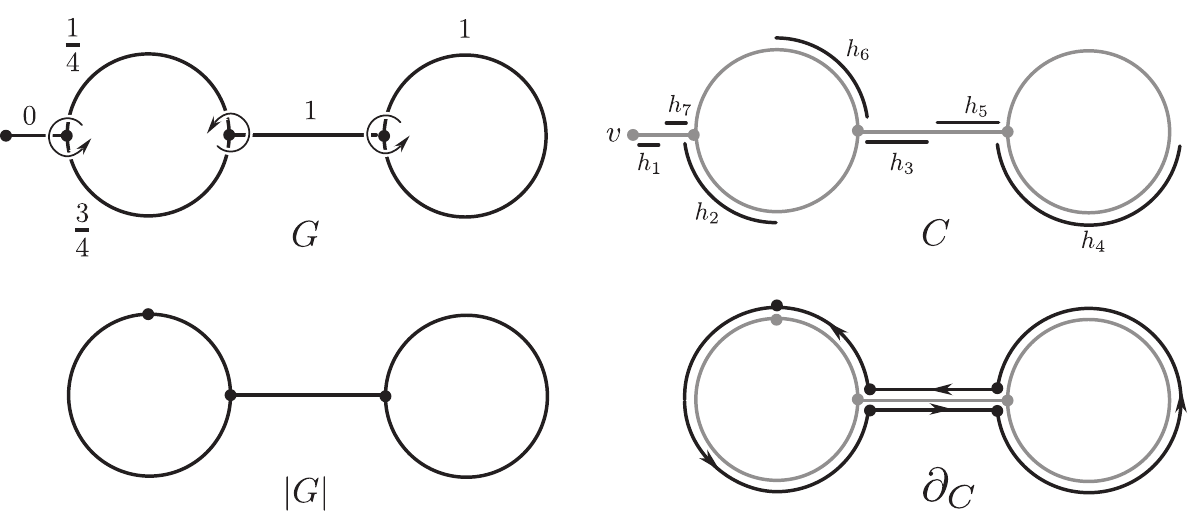}
\caption{A pseudometric fatgraph $G$, its pseudometric realization $|G|$, a marked boundary cycle $C$, and the map $\partial_C$ from a standard circle into $|G|$.
}
\label{fig:boundary cycle}
\end{figure}

Note that in the case that the pseudometric graph has no length zero edges, the pseudometric realization is homeomorphic to the usual geometric realization of a graph.
In general, the pseudometric realization is homeomorphic to a quotient of  the usual geometric realization obtained by contracting all length zero edges. 
This means that if $e$ is an edge of length zero of the pseudometric graph $G$, there is a canonical identification of $|G|$ with $|G/e|$.

We will have need for the following graph model of the circle.
A {\em fat lollipop} is a connected  fatgraph which has an external edge between a leaf and a trivalent vertex and all other vertices bivalent. We call the boundary cycle containing the leaf half-edge the {\em marked boundary cycle}.
See Figure \ref{fig:fat lollipop}.

\begin{figure}[ht]
\includegraphics{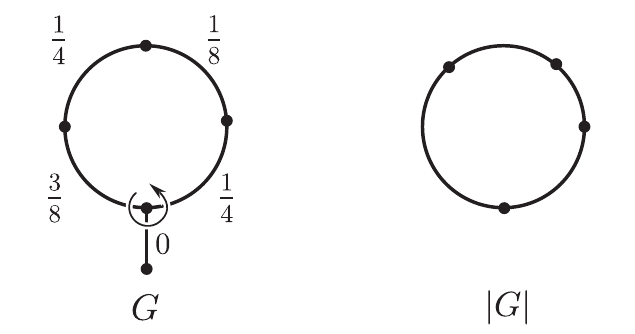}
\caption{A pseudometric fat lollipop of total length one $G$ and its pseudometric realization $|G|$.}
\label{fig:fat lollipop}
\end{figure}

We would like to use a boundary cycle to define a map from the circle into the realization of a pseudometric fatgraph.
To do this, we require exactly one marking of the boundary cycle.
Our model for the {\em length $L$ standard circle} is the quotient of the  interval $[0, L]$ which identifies $0$ with $L$.
Linear scaling provides identifications of standard circles of different  lengths.
Let $C$ be a boundary cycle of a pseudometric fatgraph $G$ with exactly one marking $v$.
There is a canonical map $\partial_C$ from a standard circle  into the pseudometric realization $|G|$ of $G$, defined as follows.
Let $C$ be the marked boundary cycle $(h_1, \dots, h_n)$ where $h_1$ is the leaf half-edge of the leaf $v$ and $h_n$ is equal to $\iota(h_1)$.
Let $\vec{e_i}$ be the oriented edge $(h_i, \iota(h_i))$ with length $\ell_i$. 
Let $L$ be the sum of the lengths $\ell_i$.
If $L=0$, $\partial_C$ is the constant map to the image of the leaf $v$ in the pseudometric realization $|G|$ of $G$.
If $L>0$, $\partial_C$ is induced by a map from $[0,L]$ to $|G|$ which maps
 the subinterval $[\ell_1 + \dots \ell_{i-1}, \ell_1 + \dots \ell_i]$ to the image of $[0, \ell_i] \times \vec{e_i}$ in $|G|$ by translation on the first factor.
This map sends $L$ to the image of $v$ in $|G|$, so it determines the map $\partial_C$ from the length $L$ standard circle to $|G|$.
See Figure \ref{fig:boundary cycle}. We will also use the map $\bar\partial_C$ from the length $L$ standard circle to $|G|$ which uses the same algorithm but follows the boundary cycle in the opposite order. Alternatively, this can be written $\bar\partial_C(t)=\partial_C(L-t)$.
	
We will be considering pseudometric fatgraph trees in service of our main definition, Definition~\ref{def:string diagram}. Later these trees will parameterize our diffuse intersections. For technical reasons it will be useful to have the distance between distinct leaves in such a tree be both bounded above and bounded away from zero (see Definition~\ref{def:short-branched}). We will achieve the bound above by imposing a length restriction on the trees we consider. We will achieve the bound away from zero by imposing a further length restriction on certain subtrees. This will require some setup.

\begin{defi}\label{def:leaf length}
The {\em leaf length} of a pseudometric fatgraph tree $T$ is one less than the number of leaves of $T$.
\end{defi}
The definition of leaf length may seem unmotivated, but see the remark following the next definition.

\begin{defi} \label{def:branch}
Let $h$ be a half-edge of a tree $T$ whose source is at least trivalent. The {\em branch} $T_h$ is the unique maximal subtree of $T$ containing $h$ so that the vertex of $h$ is a leaf in $T_h$. The {\em pollard} $T^h$ is the unique maximal subtree of $T$ which contains no edge of $T_h$.
\end{defi}
\begin{remark}
The leaf length of a branch of $T$ is the number of its leaves which are leaves in $T$.
\end{remark}
\begin{remark}
We can canonically identify $T_h$ with the component of the pruning of $T$ at $s(h)$ along $\{h\}$ which contains $h$ and $T^h$ with the component of the same pruning which does not contain $h$. See Figure \ref{fig:prunepollard}.
\end{remark}

\begin{figure}[ht]
\includegraphics{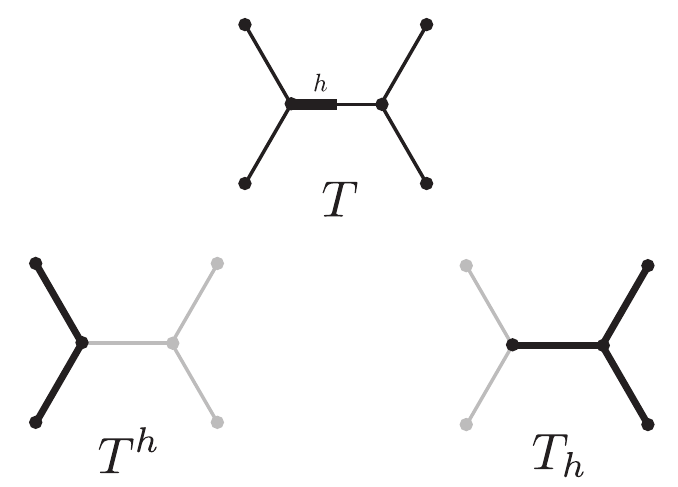}
\caption{The branch $T_h$ and pollard $T^h$ of a tree $T$ and half-edge $h$.}
\label{fig:prunepollard}
\end{figure}

Recall that the \emph{length} of a pseudometric graph is the sum of the lengths of its edges.
\begin{defi}
We call a branch of a pseudometric tree {\em prunable} if its length is equal to its leaf length.
\end{defi}

\begin{defi}\label{def:short-branched}
A pseudometric tree $T$ is {\em short branched} if:
\begin{enumerate}
\item The length of $T$ is equal to its leaf length.
\item For every half-edge $h$ of $T$ whose source is at least trivalent, the length of the branch $T_h$ is less than or equal to its leaf length.
\end{enumerate}
We call a short-branched tree {\em non-degenerate} if it has no edges of length zero and no prunable branches.
We call a short-branched tree {\em degenerate} otherwise.

\end{defi}

Figure \ref{fig:short-branched trees} shows examples of short-branched pseudometric trees and pseudometric trees that are not short branched.

\begin{figure}[ht]

\includegraphics{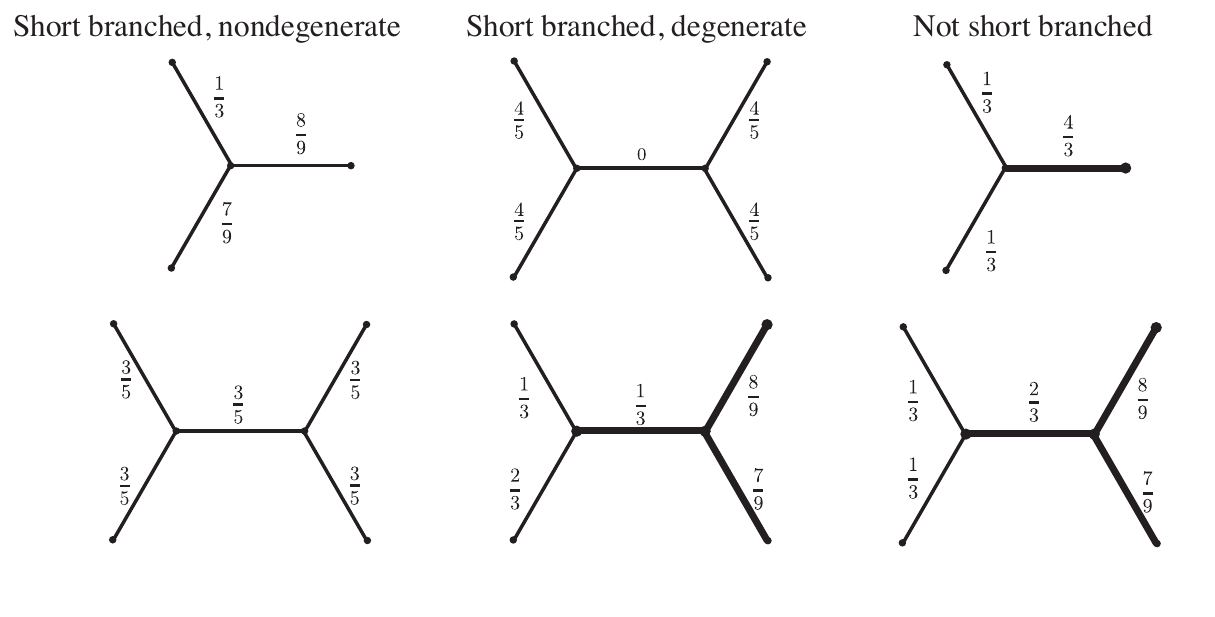}
\caption{Examples of short-branched trees and trees which are not short branched. Edge lengths are given. The subtrees made up of of bold edges are examples of branches whose lengths are greater than or equal to their leaf lengths.
}
\label{fig:short-branched trees}
\end{figure}

\section{String diagrams}\label{section:string diagrams}
\subsection{Definitions of string diagrams}
The main two definitions of this paper are Definitions~\ref{def:combinatorial string diagram} (a combinatorial definition) and~\ref{def:string diagram} (which adds continuous pseudometric data). In these definitions, we describe the type of graphs we will work with. Intuitively, we start with a collection of fat lollipops, $\{Q_i\}$. Then we glue on a collection of fatgraph trees $\{T_j\}$ along their leaves. The leaves of the first fatgraph trees are glued onto the fat lollipops, but successive trees can have some or all of their leaves glued onto trees that had been glued on previously as well. Finally, we glue on markings $\{L_i\}$ onto any boundary cycle which is not already marked. The rigorous presentation of Definition~\ref{def:combinatorial string diagram} does not follow this intutitive outline exactly, but captures the ideas behind it.

\begin{defi}\label{def:combinatorial string diagram}
A {\em combinatorial string diagram} $\Gamma$ is a connected marked fatgraph (also called $\Gamma$) with no bivalent vertices equipped with the following data: 
\begin{enumerate}
\item collections of subgraphs $\{Q_1,\ldots, Q_k\}$, $\{L_1,\ldots, L_\ell\}$, and $\{\widetilde{T}_j\}$ for $j$ in some finite set so that each edge of $\Gamma$ is contained in precisely one subgraph, and
\item a choice of a subset of internal vertices of each $\widetilde{T}_j$ as {\em fundamental}.
\end{enumerate}
This data satisfies the following conditions:
\begin{enumerate}
\item Each $Q_i$ is a fat lollipop whose inclusion into $\Gamma$ preserves its marked boundary cycle,
\item each $L_i$ is a segment precisely one of whose leaves is a leaf of $\Gamma$,
\item no $\widetilde{T}_j$ contains a leaf of $\Gamma$,
\item each internal vertex of $\Gamma$ is either contained in precisely one $Q_i$ or is disjoint from all $Q_i$ and fundamental in precisely one $\widetilde{T}_j$, 
\item the vertex explosion of each $\widetilde{T}_j$ at its non-fundamental vertices is a fatgraph tree, and
\item there is no cycle $\widetilde{T}_{j_1},\ldots, \widetilde{T}_{j_m}=\widetilde{T}_{j_1}$ such that each $\widetilde{T}_{j_i}$ has a fundamental vertex which coincides with a non-fundamental vertex of $\widetilde{T}_{j_{i+1}}$.
\end{enumerate}
We denote the vertex explosion of $\widetilde{T}_j$ at its non-fundamental vertices by $T_j$. 
\end{defi}

See Figure \ref{fig:string diagram} for a picture of a combinatorial string diagram.

\begin{remark}
We can recover $\widetilde{T}_j$ and its subset of fundamental vertices from the isomorphism class of $T_j$ and the map 
\[T_j\to \widetilde{T}_j\to \Gamma.\]
That is, a vertex is fundamental in $\widetilde{T}_j$ if and only if it is internal in $T_j$. We will often specify the additional data necessary to turn a  fatgraph into a combinatorial string diagram by specifying $T_j\to \Gamma$ rather than $\widetilde{T}_j\subset \Gamma$. The final condition captures the idea from the intuitive outline that the $T_j$ should be glued on according to a partial order on the set of trees $\{T_j\}$.
\end{remark}

Notice that the leaves of the $Q_i$ and the leaves of the $L_i$ which are also leaves of $\Gamma$ provide the markings of the boundary cycles of a combinatorial string diagram $\Gamma$.

\begin{defi}
We call the $k$ boundary cycles $C_1, \dots C_k$ of a combinatorial string diagram $\Gamma$ induced by the marked boundary cycles of $Q_i$ \emph{input boundary cycles}. 
Similarly, 
we call the other $\ell$ boundary cycles of $\Gamma$ \emph{output boundary cycles}.
\end{defi}

Later, for the string topology construction, we will use combinatorial string diagrams to describe how a $k$-fold loop in a manifold may intersect--or nearly intersect--itself.
In order to describe the combinatorics of these intersection configurations we will use the components of the following  fatgraph associated to a combinatorial string diagram.

\begin{defi}\label{def:intersectiongraph}
Let $\Gamma$ be a  combinatorial string diagram.
The {\em intersection graph} of $\Gamma$ is the  fatgraph  $\widehat{\Gamma}$ obtained as follows. Begin with the subgraph of $\Gamma$ which is the union of all the $\widetilde{T}_j$. Then perform vertex explosion on this subgraph at all vertices which are vertices of the $Q_i$ subgraphs of $\Gamma$. See Figure \ref{fig:intersection graph}.
\end{defi}

\begin{remark}
The canonical map from $T_j$ to $\Gamma$  naturally factors through $\widehat{\Gamma}$.
\end{remark}

\begin{defi}\label{def:string diagram}
A {\em string diagram} is a 
combinatorial string diagram $\Gamma$ equipped with a pseudometric structure such that
\begin{enumerate}
\item each $Q_i$ has total length one, and
\item each $T_j$ is short branched.
\end{enumerate}
By abuse of notation, we also use $\Gamma$ to denote the string diagram with underlying combinatorial string diagram $\Gamma$.
\end{defi}

\begin{figure}[ht]
\includegraphics[scale=0.75]{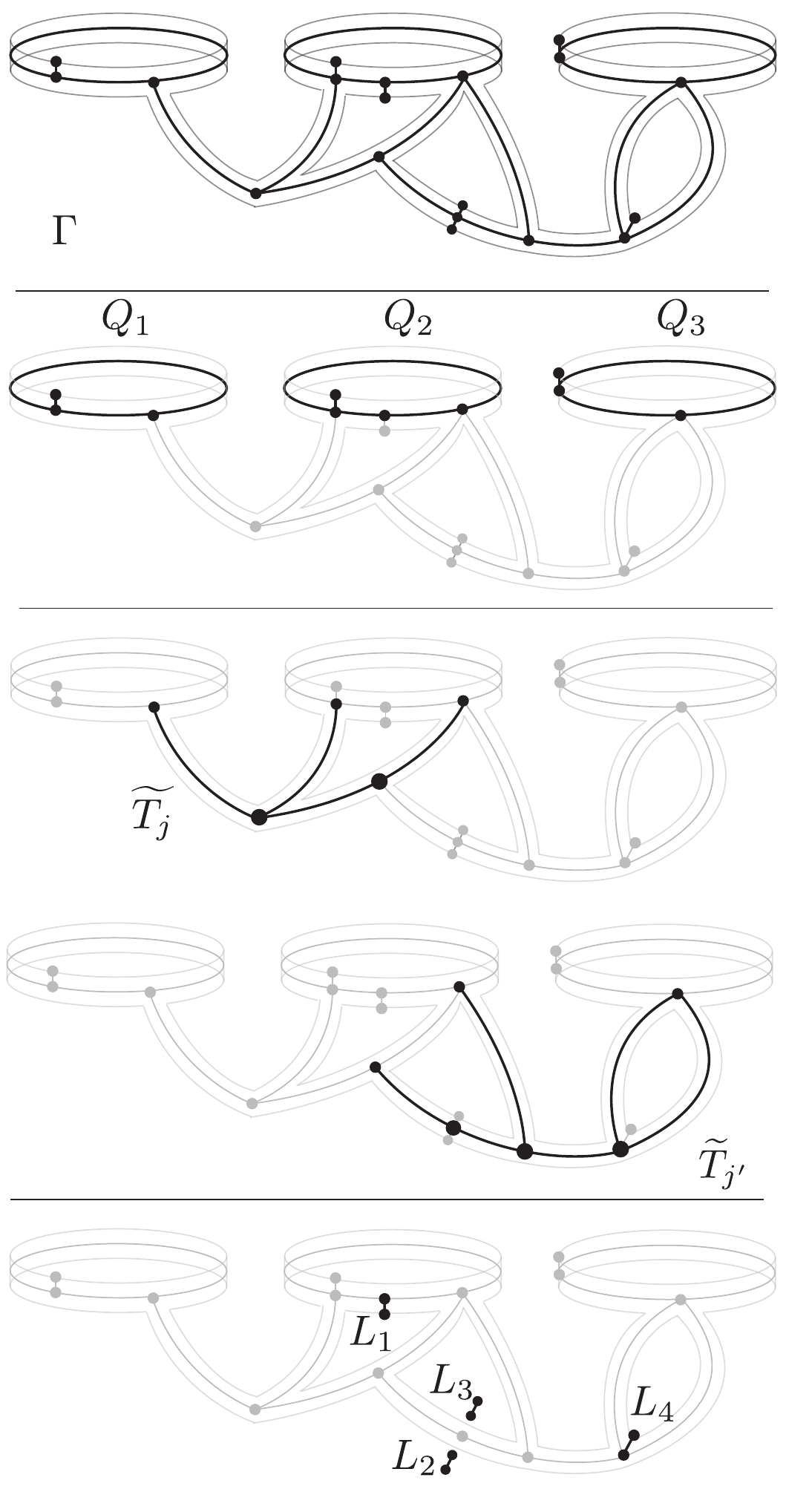}
\caption{
The data of a combinatorial string diagram $\Gamma$. Fundamental vertices are depicted as larger than non-fundamental vertices.}
\label{fig:string diagram}
\end{figure}

\begin{defi}

We call the disjoint union of $k$ copies of the standard circle, equipped with the map given by the disjoint union of the canonical maps $ \partial_{C_i}$ to $|\Gamma|$, the {\em inputs} of $\Gamma$ and denote the map by $\inputs_\Gamma$.
Similarly, 
we call the disjoint union of $\ell$ copies of the standard circle, equipped with the analogous map for output boundary cycles of $\Gamma$, the {\em outputs} of $\Gamma$.
\end{defi}

\begin{remark}
The map from the inputs to $|\Gamma|$ is always a homeomorphism onto its image. 
The same need not be true for the outputs.
\end{remark}

\begin{figure}[ht]
\includegraphics{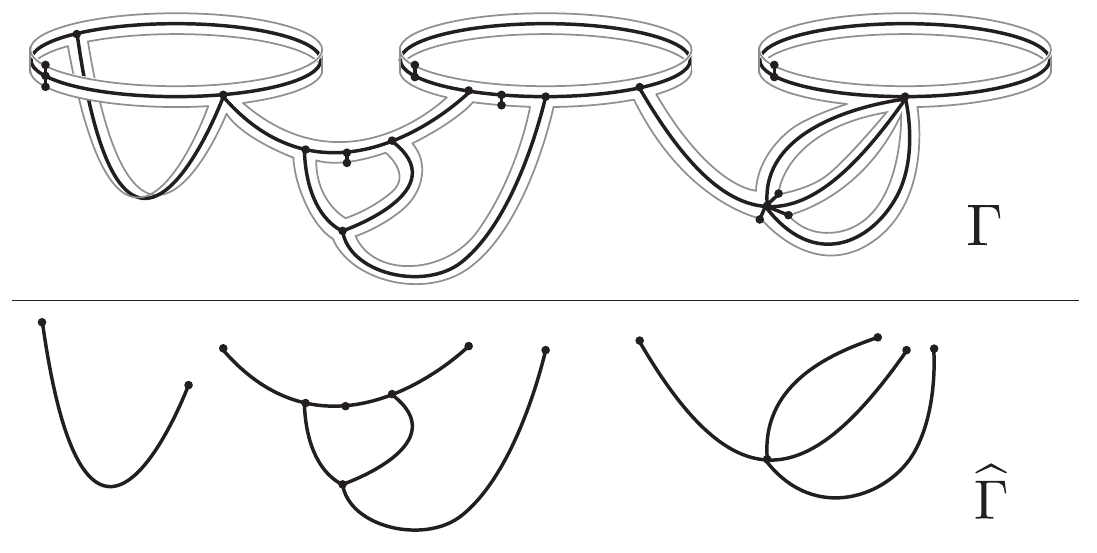}
\caption{A combinatorial string diagram $\Gamma$ and its intersection graph $\widehat{\Gamma}$.}
\label{fig:intersection graph}
\end{figure}

\begin{defi}
An {\em ordering} on a combinatorial string diagram is 
\begin{enumerate}
\item an ordering on the set $\mathcal{T} = \{T_j\}$ and
\item an ordering on the disjoint union  $\mathcal{L}$ of the sets of leaves $\leaves{T_j}$.
\end{enumerate}
\end{defi}

There is an evident action of the product of the symmetric groups of $\mathcal{T}$ and $\mathcal{L}$ on the set of orderings of a string diagram. There is a homomorphism from this group to $\mathbb{Z}/2\mathbb{Z}$ given by the product of the sign homomorphisms. An element in this product group is a pair of permutations; the element is in the kernel of the homomorphism to $\mathbb{Z}/2\mathbb{Z}$ if either both are even permutations or both are odd permutations.

\begin{defi}
An {\em orientation} on a combinatorial string diagram is a choice of ordering up to the action of the kernel of the homomorphism described above.
\end{defi}
For a string diagram, an ordering or orientation means the respective structure on its underlying combinatorial string diagram.

Orientations will allow us to build Thom classes of products of diagonal maps of our manifold $M$ in Section \ref{section:diffuse intersection class}; they will not be used in a fundamental way until then.

\subsection{The space of string diagrams}

The rest of the section is devoted to describing the topology and CW structure on the set of (isomorphism classes of) oriented string diagrams.

\begin{lemma}\label{lem:cell degenerations}
Let $\Gamma$ be a combinatorial string diagram.
Let $K$ be the set of string diagrams with underlying combinatorial string diagram $\Gamma$.

Taking internal edge lengths defines an inclusion map from $K$ to $\mathbb{R}^E$ where $E$ is the set of internal edges of $\Gamma$. The image of this inclusion is a bounded convex polytope (intersection of half-spaces) in $\mathbb{R}^E$.

Furthermore, passing to the boundary of this polytope corresponds to the following types of degenerations:
\begin{enumerate}
\item
allowing edge lengths to shrink to zero, preserving the overall length of each $Q_i$ and $T_j$ and  
\item
allowing edge lengths to change so that $T_j$ has a prunable branch.
\end{enumerate}
\end{lemma}

\begin{figure}[ht]
\includegraphics[scale=1]{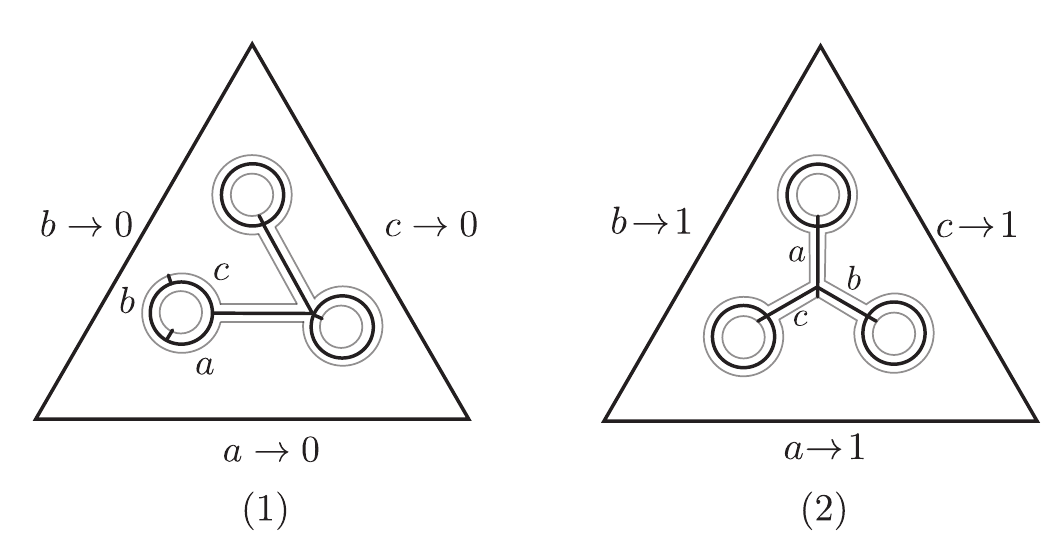}
\caption{
 Degenerations of the types (1)--edge lengths shrink to zero--and (2)--branches grow--listed in Lemma \ref{lem:cell degenerations}. Compare to Figure~\ref{fig:cell attachings}, where degenerate string diagrams at the boundary of the cell are identified with nondegenerate string diagrams in a different cell.}
\label{fig:cell degenerations}
\end{figure}

\begin{proof}
The set $K$ has parameters corresponding to the lengths of edges in $\Gamma$. Since external edges always have length zero, the set of possible edge-length parameters form a subspace of $\mathbb{R}_{\geq0}^{E}$.
The lengths are subject to the conditions:
\begin{enumerate}
\item\label{condition:lollipop length}
each $Q_i$ has total length one,
\item\label{condition:leaf length}
each $T_j$ has total length equal to its leaf length, and 
\item\label{condition:prunable}
each branch of a $T_j$ has total length at most its leaf length.
\end{enumerate}

Each condition in (\ref{condition:lollipop length}) and (\ref{condition:leaf length}) corresponds to a hyperplane in $\mathbb{R}^{E}$.
The intersections of all these hyperplanes with $\mathbb{R}_{\geq0}^{E}$ is a bounded convex polytope, in fact a product of simplices.
Each condition (\ref{condition:prunable}) corresponds to a half-space in $\mathbb{R}^{E}$.
The intersection of these  half-spaces with the convex polytope is still a convex polytope.
The faces of this polytope correspond to edges of length zero and prunable branches.

\end{proof}
By abuse of notation, we use $K$ to denote the set of string diagrams with underlying combinatorial string diagram $\Gamma$, given the subspace topology and polytope structure from the inclusion of the lemma.

The space of oriented string diagrams with a fixed underlying combinatorial diagram is the disjoint union of two copies of such a polytope.
We will define a cell complex built by assembling the oriented version of these polytopes below. 
In preparation, we discuss two operations on string diagrams that we will use to define the attaching maps.
\begin{defi}\label{def: contract edge}
Let $\Gamma$ be a string diagram with $e$ an internal edge of length zero. We define a string diagram whose underlying pseudometric fatgraph is $\Gamma/e$, the fatgraph $\Gamma$ with the edge $e$ contracted as follows. We will specify subgraphs $Q_i$ and $L_i$ along with pseudometric fatgraph trees $T_j$ equipped with a map to $\Gamma/e$ as in the remark following Definition~\ref{def:string diagram}.
\begin{enumerate}
\item 
\label{case: e in Q_i}
 If $e$ is an edge of $Q_i$, then contract $e$ in $Q_i$ and let all data be induced.
\item 
\label{case: e internal of T_j}
Similarly, if $e$ is an internal edge of $T_j$ or an external edge of $T_j$ whose internal vertex is bivalent in $T_j$, then contract $e$ in $T_j$ and let all other data be induced.
\item 
\label{case: e external of T_j}
If $e$ is an external edge of $T_j$ whose internal vertex $v$ is at least trivalent in $T_j$, then contract $e$ in $T_j$ and then perform a vertex explosion at the identified vertex. This procedure yields a graph $(T_j/e)^{[v]}$ which has components $\{T_a\}$, each of which comes equipped with a map $T_a\to (T_j/e)^{[v]}\to T_j/e\to \Gamma/e$. This specifies a subgraph $\widetilde{T}_a$ for each $T_a$. Let all other data be induced.  See Figure \ref{fig:contract external edge}.

\begin{figure}[ht]
\includegraphics[scale=.7]{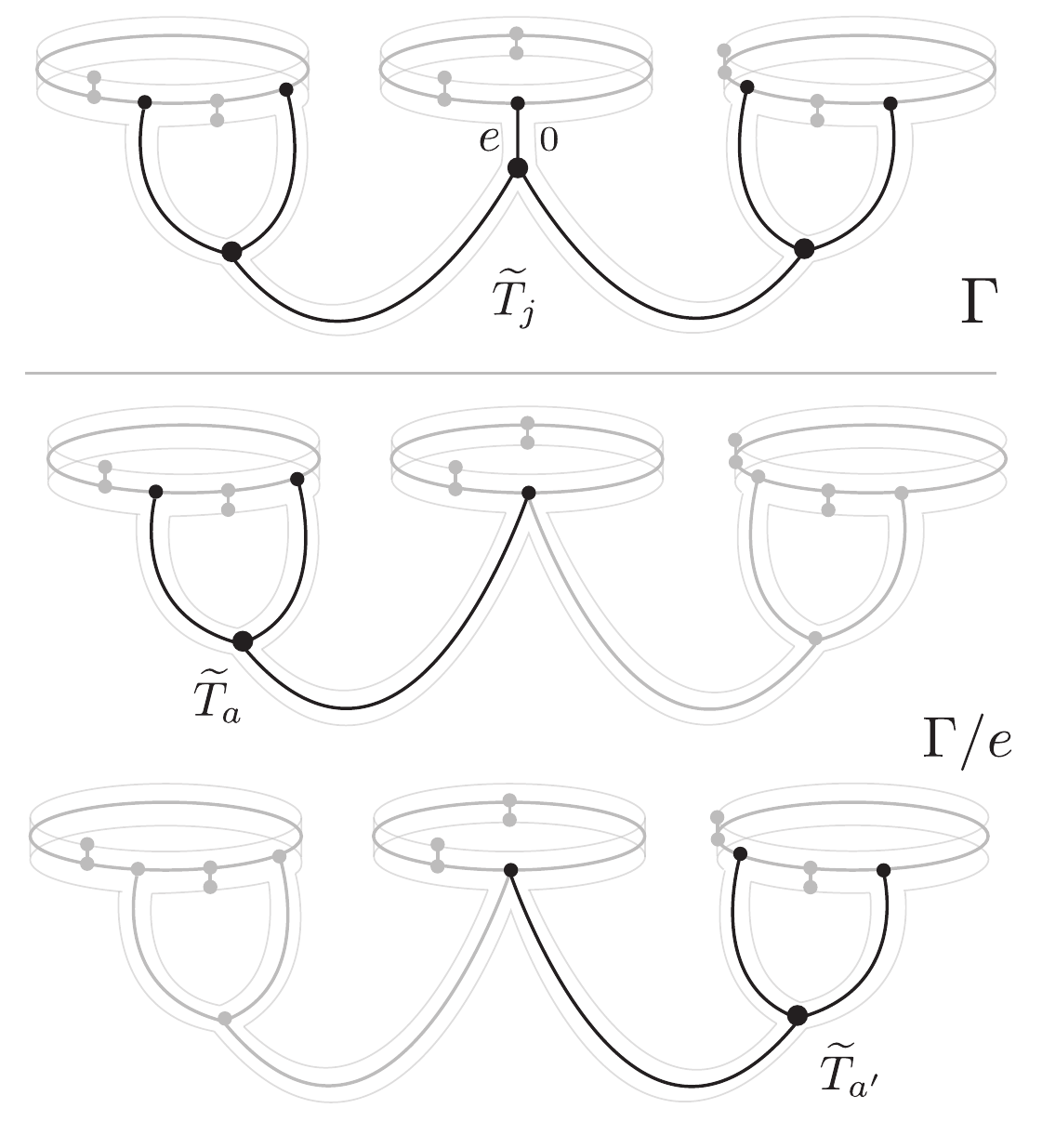}
\caption{Contraction of an external edge of a tree whose internal vertex is 3-valent.}
\label{fig:contract external edge}
\end{figure}
\end{enumerate}
Note that if $e$ is an edge of $T_j$ with two external vertices, then $T_j$ is a segment and thus has leaf length $1$ and length $1$ so $e$ cannot have length $0$.
\end{defi}

\begin{defi}

Let $\Gamma$ be a string diagram and let $h$ be a half-edge in $T_j$ whose source is at least trivalent in $T_j$. Suppose further that the branch $T_h$ of $T_j$ is prunable. The {\em pruning} of $\Gamma$ along $h$, denoted $\Gamma\prune h$, has the same underlying pseudometric fatgraph as $\Gamma$ but has its subgraphs changed as follows. The inclusions of $T_h$ and $T^h$ into $T_j$ induce maps from $T_h$ and $T^h$ into $\Gamma$ which specify $\widetilde{T}_h$ and $\widetilde{T}^h$ to replace $\widetilde{T}_j$ in $\Gamma\prune h$. See Figure \ref{fig:prune}.

\begin{figure}[ht]
\includegraphics[scale=.7]{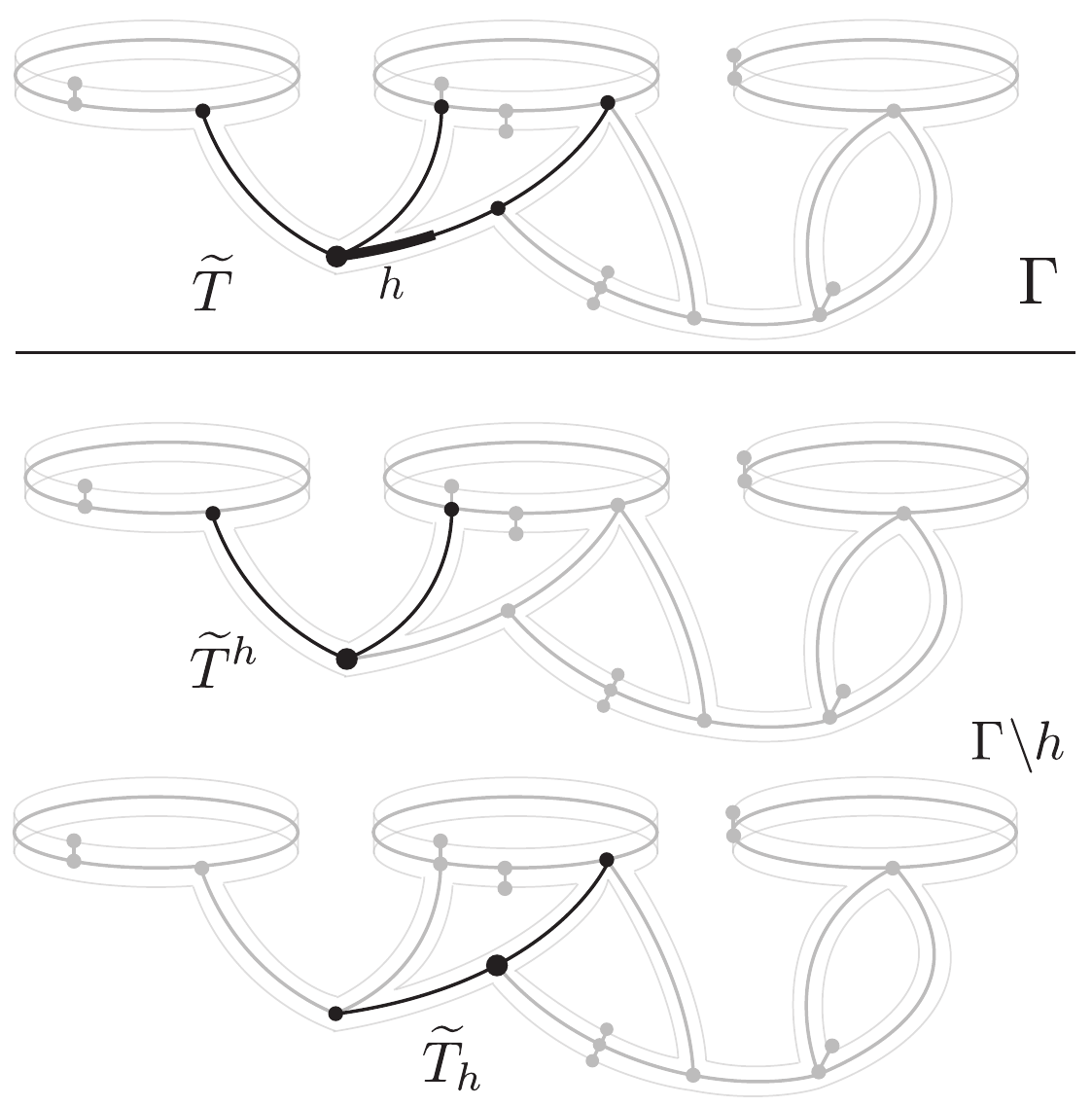}
\caption{The pruning of a tree $T_j$ in a string diagram along $h$.}
\label{fig:prune}
\end{figure}
	\end{defi}
It is an elementary lemma to verify that the induced structures described are in fact string diagrams. That is, it is direct to show that the balance conditions for the fatgraph tree subgraphs and the existence, uniqueness, and cycle-free properties of fundamental vertices are preserved in the edge contraction or the pruning.

\begin{remark}
Let $\Gamma$ be a string diagram.
For an edge $e$ of length zero, there is a canonical isomorphism from $|\Gamma|$ to $|\Gamma / e|$;
similarly, for a half-edge $h$ which determines a prunable branch in some $T_j$, there is a canonical isomorphism from $|\Gamma|$ to $|\Gamma \prune h|$.
\end{remark}

In the case that $\Gamma$ is an oriented string diagram and $e$ is an internal edge of length zero, we induce an orientation on $\Gamma/e$ as follows.
In cases (\ref{case: e in Q_i}) and (\ref{case: e internal of T_j}) in Definition \ref{def: contract edge}, an orientation on $\Gamma$ induces an evident orientation on $\Gamma/e$.
In case (\ref{case: e external of T_j}), where $e$ is an external edge of a tree $T_j$ whose internal vertex $v$ has valence $|v|\ge 3$, we must be more explicit.
Fix a representative ordering of the orientation on $\Gamma$ so that the tree $T_j$ is first in the ordering on $\mathcal{T}$. In $\Gamma/e$, the tree $T_j$ is replaced by the set of trees $\{T_a\}$; choose an arbitrary order $\{T_a^1, \dots, T_a^{|v| - 1}\}$ on this set. 
The corresponding set of new leaves $\{\ell_a\}$ are added to $\mathcal{L}$ such that $\ell_a^2, \dots, \ell_a^{|v|-1}$ appear in order at the beginning, and the leaf $\ell$ of $e$ is replaced by the leaf $\ell_a^1$. 
Notice first that the orientation class of the constructed ordering is invariant of both the representative ordering of $\Gamma$ and the chosen order on $\{T_a\}$ (because any permutation on the set $\{T_a\}$ induces a corresponding permutation with the same sign on the set $\{\ell_a\}$).

In the case that $\Gamma$ is an oriented string diagram and $h$ is a half-edge of $T_j$ corresponding to a prunable branch $T_h$, we induce an orientation on $\Gamma \prune h$ as follows.
Choose a representative ordering of the orientation on $\Gamma$ so that the tree $T_j$ is first in the ordering on $\mathcal{T}$. Then replace the tree $T_j$ with the ordered pair $(T^h, T_h)$.
In the ordering on $\mathcal{L}$, the new leaf $s(h)$ of $T_h$ is added at the beginning.

\begin{remark}
The only automorphism of a combinatorial string diagram (fixing the orders on the sets of subgraphs $\{Q_i\}$ and $\{L_i\}$) is the identity. This implies that isomorphic combinatorial string diagrams can be canonically identified. Because of this we can safely be sloppy about whether we are referring to a particular combinatorial string diagram or its isomorphism class. 
The same is true for oriented combinatorial string diagrams and orientation preserving automorphisms.
Our main application of this fact is the canonical identification of the polytopes $K$ for two oriented combinatorial string diagrams in the same isomorphism class.
\end{remark}

\begin{defi}\label{def:space of string diagrams}
Let $\spaced$ be the cell complex built as follows.

There is a proper class of polytopes $K$, where the underlying oriented combinatorial string diagram ranges over all oriented combinatorial string diagrams. The set of cells of $\spaced$ is this proper class under the canonical identification of polytopes for isomorphic oriented combinatorial string diagrams.

The attaching map for $K$ takes the oriented string diagram ${\Gamma}$ in a codimension one face of $K$ to the oriented string diagram described below.
This is well defined by Proposition \ref{prop:space of string diagrams}.
\begin{enumerate}
\item
If ${\Gamma}$ has an internal edge $e$ of length zero, identify ${\Gamma}$ with $\Gamma/e$ with the induced orientation. Note that this is not a codimension one identification if $e$ is an edge of $\widetilde{T}_j$ between an at least trivalent fundamental vertex of $\widetilde{T}_j$ and a non-fundamental vertex of $\widetilde{T}_j$. 
\item
If a branch ${T}_h$ of ${T}_j$ in ${\Gamma}$ is prunable, identify $\Gamma$ with $\Gamma\prune h$ with the induced orientation.
\end{enumerate}
\end{defi}

Figure \ref{fig:cell attachings} shows a picture of the attaching maps.

\begin{figure}[ht]
\includegraphics[scale=1]{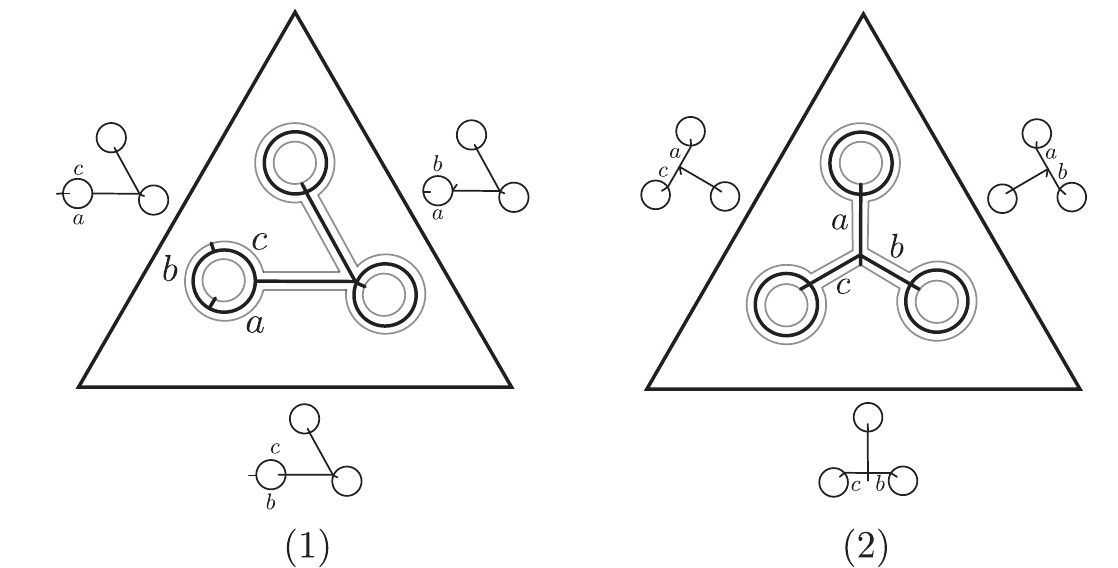}
\caption{
Attachings of the types listed in Definition \ref{def:space of string diagrams}. Compare to Figure \ref{fig:cell degenerations}.
}
\label{fig:cell attachings}
\end{figure}

\begin{prop}\label{prop:space of string diagrams}
The attaching maps in Definition \ref{def:space of string diagrams} are well defined and so $\spaced$ is a cell complex.
\end{prop}

\begin{proof}

In order to ensure that the attaching maps are well defined, we must check that they agree on codimension two faces.
Because each cell $K$ is a polytope, each codimension two face arises in precisely two ways as a codimension one face of a codimension one face.

In almost every case, the two codimension one degenerations evidently commute up to a potential difference in orientation. When both codimension one degenerations are pruning degenerations, there are four cases depending on the combinatorics of the trees involved. In every case, it is easy to see that the orientations of the two sequences of attaching maps agree. See Figure~\ref{fig:prunepruneattachingcommutes}. 

As a representative example, consider case (c) in the figure. Let $T$ be the depicted tree, imagined as some local piece of an oriented string diagram. There is a codimension one degeneration corresponding to pruning $T$ at $h$; the orientation $(T,\ldots);(\ldots)$ induces the orientation $(T^h,T_h,\ldots);(s(h),\ldots)$. Then pruning $T^h$ at $h'$ induces the orientation 
\[((T^h)^{h'},(T^h)_{h'},T_h,\ldots);(s(h'),s(h),\ldots).\] On the other hand, pruning $T$ at $h'$ induces the orientation $(T^{h'},T_{h'},\ldots);(s(h'),\ldots)$. Further pruning $T_{h'}$ at $h$ requires interchanging $T^{h'}$ and $T_{h'}$ in the representative ordering. Then the induced orientation on the doubly pruned string diagram is 
\[((T_{h'})^h,(T_{h'})_h,T^{h'},\ldots);(s(h),s(h'),\ldots),\] where we have swapped $s(h)$ and $s(h')$ in the ordering to compensate for interchanging $T_{h'}$ and $T^{h'}$. Finally, note that $(T^h)^{h'}=T^{h'}$, that $(T^h)_{h'}=(T_{h'})^h$, and that $(T_{h'})_h=T_h$; then the representative orders of the two induced orientations differ by an even permutation on the set of trees and thus agree. 

The other cases in Figure~\ref{fig:prunepruneattachingcommutes} are similar and will be omitted. 

Only one further case requires comment. If the codimension two face arises by contracting an external edge $e$ of length zero of a tree ${T}_j$ whose internal vertex $v$ in $T_j$ is trivalent, then the two degenerations are as follows. 
First recall that branches of ${T}_j$ correspond to internal half-edges of ${T}_j$.
At $v$ there are two branches corresponding to the edges other than $e$.
Because $e$ has length zero, both of these branches are prunable.
After performing an identification of type (2) with either one of these branches, the other degeneration becomes a degeneration of type (1), where $e$ is now a length zero edge.
The two resulting string diagrams are equivalent.  See Figure~\ref{fig:prunecontractattachingcommutes}. Given an ordering on a string diagram $\Gamma$ containing such a tree $T_j$, the induced orderings for these two degenerations always differ by a transposition on both the tree factor and on the leaf factor, so the induced orientations also agree.

\end{proof}
\begin{figure}
\includegraphics[scale=.8]{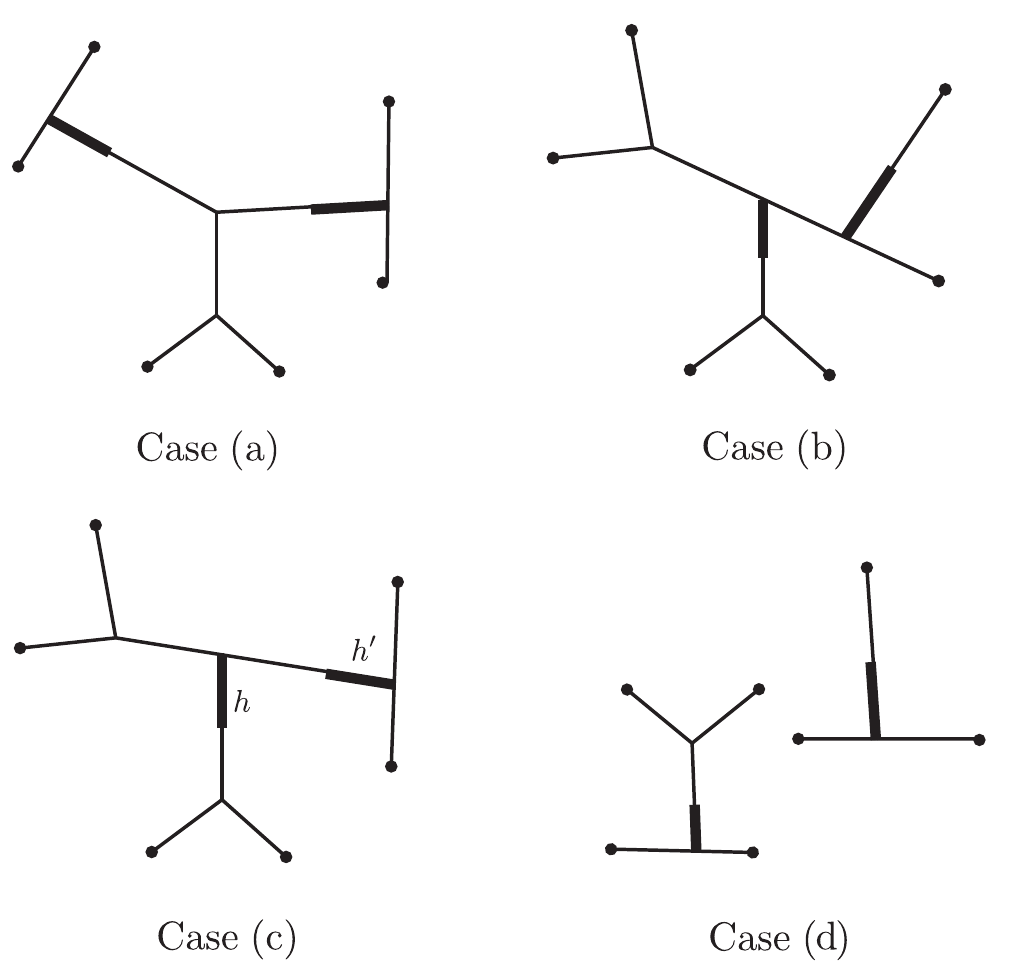}
\caption{The four types of codimension two degenerations where both codimension one degenerations are prunings.}
\label{fig:prunepruneattachingcommutes}
\end{figure}

\begin{figure}
\includegraphics{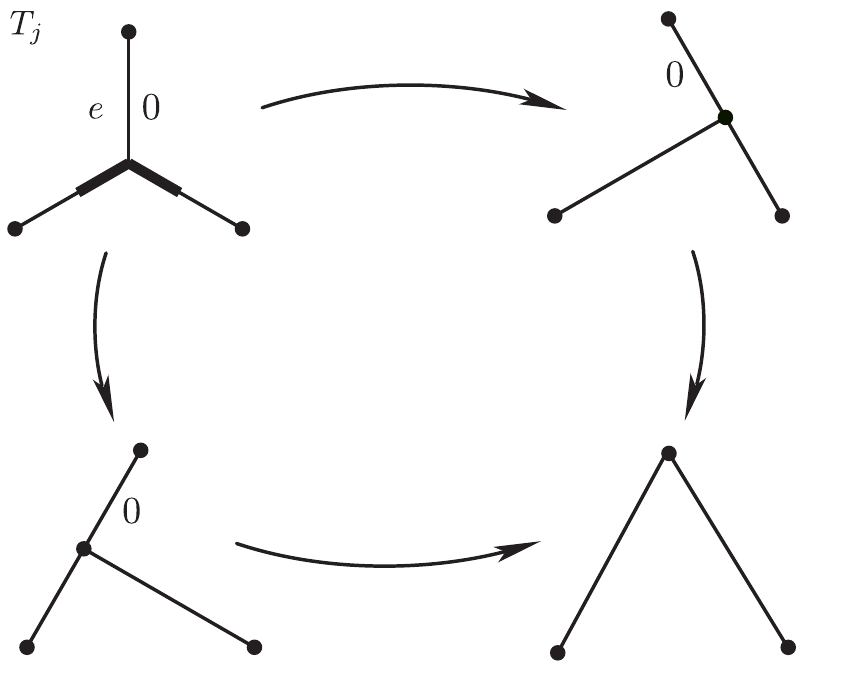}
\caption{The induced orientations for a codimension two degeneration that arises by contracting an external edge of a tree $T_j$ whose internal vertex is trivalent.}\label{fig:prunecontractattachingcommutes}
\end{figure}

\begin{remark}
The attaching maps work just as well without an orientation and we can also define a space $\spaceduo$ of unoriented string diagrams just as in Definition~\ref{def:space of string diagrams}. There is a double cover $\spaced\to \spaceduo$ given by forgetting the orientation.\footnote{In the special case where a string diagram has no $T_j$, then this is \emph{not} a double cover but an isomorphism. See Proposition~\ref{prop:we agree BV} for some discussion of this special case.}
\end{remark}
\begin{conjecture}\label{conj:trivialcover}
The double cover $\spaced\to \spaceduo$ is trivial.
\end{conjecture}
If we knew Conjecture~\ref{conj:trivialcover}, we could treat $\spaceduo$ as our fundamental space of operations instead of $\spaced$; as it is, we will need the extra data of the orientation to make choices to build the diffuse intersection class later. See Section~\ref{section:diffuse intersection class}.

\begin{defi}\label{def:gkl space}
Let $\spaced \gkl$ be the subspace of $\spaced$ where there are $k$ inputs, $\ell$ outputs, and where each string diagram has Euler characteristic $\chi$.

\end{defi}

\begin{prop}\label{prop:gkl space}
The space $\spaced(\chi,k,\ell)$ is a finite cell complex. 
\end{prop}

\begin{proof}
None of the attaching maps change the Euler characteristic of the underlying fatgraph or change the number of inputs or outputs. Therefore $\spaced(\chi,k,\ell)$ is a subcomplex of $\spaced$. 
There are a finite number of connected marked fatgraph isomorphism types with no bivalent vertices and fixed Euler characteristic. 
Given a marked fatgraph, there are only finitely many ways to give it an oriented combinatorial string diagram structure as in Definition \ref{def:string diagram}.

This shows that the space $\spaced \gkl$ is a finite cell complex.

\end{proof}
\begin{remark}
Spaces of fagraphs like $\spaced$ and $\spaceduo$ have been used for some time to study not only string topology, but moduli spaces of Riemann surfaces with boundary as well
\cite{Strebel:QD, Penner:TDTMSOPS, Harer:TCOTMSOC, Kontsevich:ITOTMSOCATAF, Igusa:HFRT, Costello:ADPOVOTRGDOMS, Godin:TUSIHOTMCGOASWB}. We will denote the disjoint union of such moduli spaces over all genera and number of incoming and outgoing boundary components by $\mathcal{M}$.
For example, Cohen and Godin defined {\em Sullivan chord diagrams} and  {\em marked metric chord diagrams}, which they use to define string topology operations \cite{CG} (we will show in Section~\ref{section:previouswork} that our constructions recover those of Cohen--Godin).
Their space of marked metric chord diagrams, which we denote by $\mathcal{CG}$, includes into $\mathcal{M}$. For the subspace of genus-zero surfaces with one outgoing boundary component, this inclusion is a homotopy equivalence.
Cohen and Godin initially thought that their inclusion could be a homotopy equivalence in general, but this turned out not to be the case  \cite{Godin:CGMSMSBRSMSSC}. 
Tradler and Zeinalian later defined a more general version of marked metric chord diagrams to study algebraic string topology operations \cite{TZ}.
Their space,  which we denote by $\mathcal{TZ}$, is a compactificaction of the space $\mathcal{CG}$ and in turn, their space includes into a compactification of $\mathcal{M}$.
More specifically, the space $\mathcal{TZ}$ is a deformation retract of B\"odigheimer's harmonic compactification of $\mathcal{M}$, which we denote by $\mathcal{BM}$ \cite{Bodigheimer, PoirierThesis, EgasKupers:CCMFMSATC}.
Thus there is a commutative square of inclusions, as in the outside square of the diagram below.

\[
\begin{tikzcd}
{\mathcal{CG}} \ar[hookrightarrow]{rrrr}{\not\sim} \ar[hookrightarrow]{dd}
&&&&
\mathcal{M}\ar[hookrightarrow]{dd}{}
\\
&&
\spaceduo \ar{dll}{q} \ar[dashed, swap]{rru}{\sim}
\\
\mathcal{TZ}\ar[hookrightarrow]{rrrr}{\sim}
&&&&
\mathcal{BM}
\end{tikzcd}
\]

The space $\mathcal{TZ}$ receives a quotient map from our space $\spaceduo$.
We have seen that, for some components with low genus and small numbers of boundary components, the composition of this quotient with the inclusion of $\mathcal{TZ}$ into $\mathcal{BM}$ has a lift up to homotopy to a map from $\spaceduo$ to $\mathcal{M}$, as depicted in the digaram.
Further, in these examples, this lift is a homotopy equivalence. This constitutes our evidence for Conjecture~\ref{conj: homotopy type is moduli space} in the introduction.
\end{remark}

\section{Straightening a string diagram}\label{section:straightening}

Next, we define {\em straightening} maps from the intersection graph of a string diagram to a product of simplices, one simplex for each component of the intersection graph.
Once this is done we will use maps of these simplices into $M$ to define the map $\heartsuit$ and its domain $S$.

\begin{prop}\label{prop:degenerate straighten}
Given a short-branched tree $T$ with leaf set $\leaves{T}$, let $\Delta_{\leaves{T}}$ be the simplex spanned by $\leaves{T}$.
There exists a {\em straightening map} $\straighten$ from the pseudometric realization $|T|$ of $T$ to $\Delta_{\leaves{T}}$, which satisfies the following properties.
\begin{enumerate}
\item 
The map $\straighten$ takes  each leaf in $|T|$ to itself in $\Delta_\leaves{T}$.
\item 
Let $e$ be an internal edge of $T$ or an external edge of $T$ whose internal vertex is bivalent,
Assume $e$ has length zero.
Let $|T| \to |T/e|$ be the isomorphism induced the contraction of the edge $e$.
Then the following diagram commutes.
\[
\begin{tikzcd}
{|T|} \ar{rrrr}{\straighten} \ar{dd}
&&&&
\Delta_{\leaves{T}}\ar{dd}{\cong}
\\
\\
{|T/e|} \ar{rrrr}{\straighten}
&&&&
\Delta_{\leaves{T/e}}
\end{tikzcd}
\]

\item
Let $T_h$ be a prunable branch of $T$ with pollard $T^h$.
Since $\leaves{T^h}$ is a subset of $\leaves{T}$, there is a natural inclusion of $\Delta_\leaves{T^h}$ in $\Delta_{\leaves{T}}$.
Since every leaf of $T_h$ except $s(h)$ is also a leaf of $T$, assigning a point in $\Delta_{\leaves{T}}$ to $s(h)$ yields a  linear inclusion of $\Delta_{\leaves{T_h}}$ in $\Delta_{\leaves{T}}$.
Here, since $s(h)$ is also a vertex of $T^h$, it has an image point $|s(h)|$ in the pseudometric realization $|T^h|$. Thus we assign to $s(h)$ the image of $|s(h)|$ under the straightening map of $T^h$.
Then the following diagram commutes.
\[
\begin{tikzcd}
{|T^h|} \sqcup |T_h| \ar{rrrr}{\straighten} \ar{dd}
&&&&
\Delta_{\leaves{T^h}} \sqcup \Delta_{\leaves{T_h}} \ar{dd}
\\\\
{|T|} \ar{rrrr}{\straighten}
&&&&
\Delta_{\leaves{T}}
\end{tikzcd}
\]

\end{enumerate}
\end{prop}

The proof is technical and is deferred to Appendix \ref{appendix:straightening}. See Figure~\ref{fig:straightened tree} for a picture of the straightening map.
\begin{figure}[ht]
\includegraphics[scale=0.8]{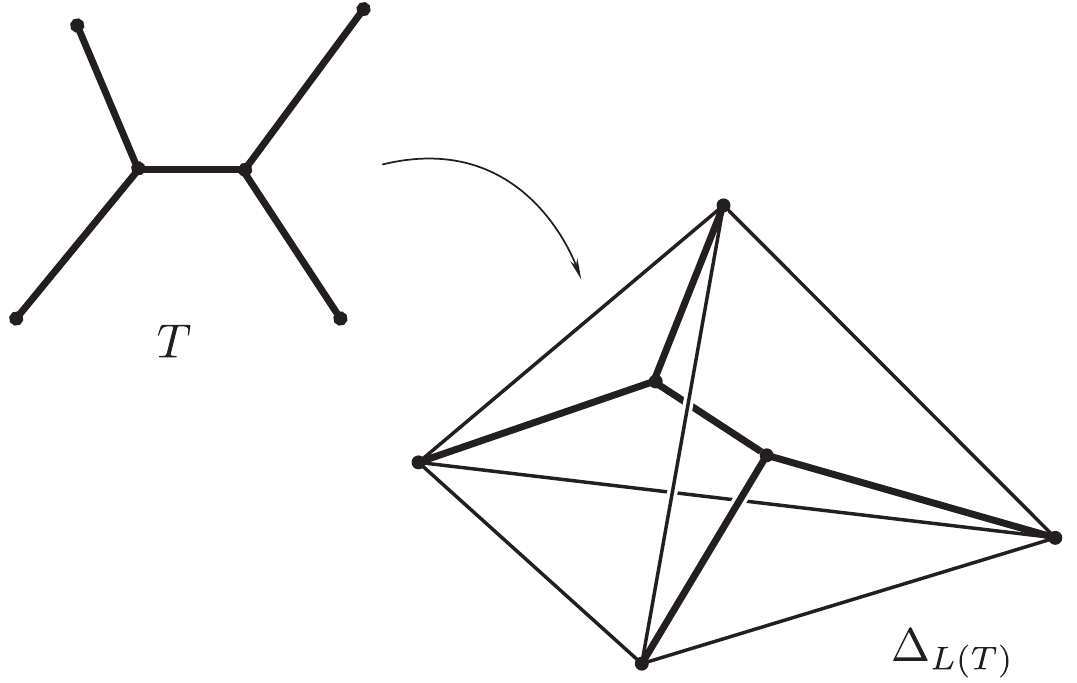}
\caption{A short-branched tree $T$ and its straightening in $\Delta_{\leaves{T}}$}
\label{fig:straightened tree}
\end{figure}

Let $C$ be a component of the intersection graph $\widehat{\Gamma}$ of the  string diagram $\Gamma$ with leaf set $\leaves{C}$.
Let $\Delta_{\leaves{C}}$ be the simplex generated by leaves of $C$.
Here we describe how to use the straightening map for trees above to extend the inclusion of the leaves of $C$ into $\Delta_{\leaves{C}}$ to a map from $|C|$ to $\Delta_{\leaves{C}}$, even if $C$ is not a tree.

The component $C$ comes equipped with canonical maps from each tree in a subset of the set of trees $\{T_j\}$ discussed in the remark following Definition \ref{def:intersectiongraph}.
We will give a well-defined map from $|C|$ to $\Delta_{\leaves{C}}$, defined piecewise by first mapping each  of these trees $|T_{j}|$ to $\Delta_{\leaves{T_j}}$ using straightening maps above, and then mapping each such $\Delta_{\leaves{T_j}}$ to $\Delta_{\leaves{C}}$. 

The maps $\Delta_{\leaves{T_j}}$ to $\Delta_{\leaves{C}}$ are defined inductively and we require an ordered partition $P_1, P_2, \dots, P_m$ of the set of trees $T_{j}$ of $C$.
The ordered partition comes from the canonical  map from $T_{j}$ to $C$ and is defined as follows.
The tree $T_{j}$ is in $P_n$ if it is not in any previous subset of the partition and
the image in $C$ of the leaves of $T_j$ is contained in the union of the leaves of $C$ and the images in $C$ of all trees in $P_1, P_2, \dots, P_{n-1}$.

Assume the maps $\Delta_{\leaves{T_{j}}}$ to $\Delta_{L_C}$ have been defined for trees in $P_1, P_2, \dots, P_{n-1}$.
Let $T_j$ be in $P_n$.
A vertex of $\Delta_{\leaves{T_j}}$ corresponds to a leaf of $T_j$, which already has an image in $\Delta_{L_{C}}$ by the definition of $P_n$.
Extend this linearly to a map from $\Delta_{\leaves{T_j}}$ to $\Delta_C$.

\begin{defi}
The {\em straightening} of a component $C$ of the intersection graph $\widehat{\Gamma}$ of a  string diagram $\Gamma$ is the map described immediately above from $|C|$ to $\Delta_{\leaves{C}}$.
The {\em straightening} of the  string diagram $\Gamma$ is the disjoint union over components $C$ of the intersection graph $\widehat{\Gamma}$ of the straightening maps $|C|$ to $\Delta_{\leaves{C}}$.
We also denote this map by $\straighten$.
\end{defi}

The conditions of Proposition \ref{prop:degenerate straighten} may be extended to the straightening maps for string diagrams, which are used in the  definition of the map $\heartsuit$ in the next section.

Recall that if a string diagram has
 an internal edge of length zero, the attaching map of Definition \ref{def:space of string diagrams} contracts the edge.
If the string diagram has a tree with a prunable branch, the attaching map breaks the tree into two trees.

Roughly, the next lemma shows that contracting an edge commutes with straightening.
For the kind of edges we will consider, there is a canonical bijection between the leaf set of $C$ and the leaf set of $C/e$.

\begin{lemma} \label{lemma: contract straighten component}
Let $|C| \to |C/e|$ be the map which contracts an edge $e$ which is either a external edge with bivalent internal vertex or an internal edge which is not the image of an external edge of a tree $T_j$ whose internal vertex is at least trivalent in $C$.
Then the following diagram commutes.
\[
\begin{tikzcd}
{|C|} \ar{rrrr}{{\straighten}} \ar{dd}[swap]{\cong}
&&&&
\Delta_{L_{C}}\ar{dd}{\cong}
\\
\\
{|C/e|} \ar{rrrr}{{\straighten}}
&&&&
\Delta_{L_{C/e}}
\end{tikzcd}
\]
\end{lemma}

\begin{proof}
The edge $e$ of $C$ is an edge in the image of a unique tree $T_j$. The straightening maps and simplex inclusion maps for any tree earlier than or incomparable with $T_j$ in the ordered partition of trees in $C$ are unaffected by the edge contraction. 

Proposition \ref{prop:degenerate straighten} shows that straightening the tree $T_j$ commutes with contracting $e$.
Because all previous straightening maps and simplex inclusions are unchanged, the map $\Delta_{\leaves{T_j}} \to \Delta_{\leaves{C}}$ is also unchanged. Then, the straightening maps and simplex inclusions are unchanged for trees after $T_j$ in the ordered partition.
\end{proof}

Next, we consider a  string diagram $\Gamma$ where the tree $T_j$ has a prunable branch $T_h$.
Let $C$ be the component of $\widehat{\Gamma}$ which receives the canonical map from $T_j$.
Under the attaching map in Definition \ref{def:space of string diagrams}, $\Gamma$ is identified with the  string diagram $\Gamma \prune h$ where the tree $T_j$ is broken into two trees $T^h$ and $T_h$.
The component $C \prune h$ of  $\widehat{\Gamma \prune h}$ which receives the canonical map from  $T^h$ and $ T_h$ is canonically isomorphic as a pseudometric fatgraph to $C$.

\begin{lemma}\label{lemma: prune straighten component}
Let $C$ and $C \prune h$ be as above and let $|C| \to |C \prune h|$ be the identity map.
Then the following diagram commutes.
\[
\begin{tikzcd}
{|C|} \ar{rrrr}{{\straighten}} \ar{dd}[swap]{\cong}
&&&&
\Delta_{L_{C}}\ar{dd}{\cong}
\\
\\
{|C \prune h|} \ar{rrrr}{{\straighten}}
&&&&
\Delta_{L_{C\prune h}}
\end{tikzcd}
\]

\end{lemma}
\begin{proof}

As in Lemma~\ref{lemma: contract straighten component}, earlier and incomparable trees are not affected a priori.
Proposition~\ref{prop:degenerate straighten} show that $T_j$ and subsequent trees have the same straightening maps and simplex inclusions.  
\end{proof}

\section{The heart of the string topology construction}\label{section:stringtopology}

Let $M$ be a closed Riemannian manifold.
Given an oriented string diagram $\Gamma$, and a map $\gamma$ from its inputs into $M$, we would like to produce a map from its outputs into $M$.
To do this, we would like to extend $\gamma$ to a map from $|\Gamma|$ to $M$, and then pull back to the outputs of $\Gamma$.
The extension of $\gamma$ to $|\Gamma|$ is our version of a wrong-way map and is the heart of the string topology construction.

Our construction requires certain points in the image of $\gamma$  be close together in $M$.
In this section, we define
\begin{enumerate}
\item a space $S$ which realizes this closeness condition,
\item a space $\spaced(M)$ which captures mapping spaces from oriented string diagrams to $M$ as oriented string diagrams vary in $\spaced$, and
\item a map $\heartsuit$ from $S$ to $\spaced(M)$ which realizes the extension of $\gamma$ to $|\Gamma|$.
\end{enumerate}

In what follows, all oriented string diagrams have Euler characteristic $\chi$ and $k$ inputs.
The number of outputs $\ell$ is also fixed and to simplify notation, we use $\spaced$ for $\spaced \gkl$. All of the constructions in this section would work just as well for unoriented string diagrams, but as we will eventually need to use the orientation, we have chosen to work with it from the beginning.

Let $LM$ be the space of continuous maps from the standard circle into $M$.
Let $LM^k$ be the $k$-fold Cartesian product of $LM$ with itself.
A point in $LM^k$ may be represented as a map $\gamma$ from $k$ copies of the standard circle into $M$.
In the definition, we will use the canonical identification of the domain of  $\gamma$ with the subspace of the pseudometric realization of an oriented string diagram $\Gamma$ determined by the subgraph $\sqcup Q_i$.
We also specify the notation $\iota$ for the map from the leaves of $\widehat{\Gamma}$ to $\sqcup |Q_i|$.

The domain of the map $\heartsuit$ will be fibered over $\spaced$ and we begin by describing the fiber over a particular string diagram $\Gamma$.
In fact, we will describe a fiber that depends on a parameter $\varepsilon$ and later we will fix the appropriate $\varepsilon$ which will make the construction of $\heartsuit$ possible.

\begin{defi}\
Let $\Gamma$ be an oriented string diagram and let $\gamma$ be in $LM^k$. 
Let $a$ and $b$ be two leaves in the same component of the intersection graph $\widehat{\Gamma}$. 
We say that $\gamma$ is {\em$\varepsilon$-Lipschitz with respect to $a$ and $b$} if the distance in $M$ between $\gamma(\iota(a))$ and $\gamma(\iota(b))$ is less than $\varepsilon$ times the distance in $|\widehat{\Gamma}|$ between $a$ and $b$.
We say that $\gamma$ is  {\em$\varepsilon$-Lipschitz with respect to $\Gamma$} if
it is $\varepsilon$-Lipschitz with respect to all such pairs of leaves.
\end{defi}

We call such maps $\varepsilon$-Lipschitz with respect to $\Gamma$ because we would like to extend $\gamma$ to a  function from $|\Gamma|$ to $M$ whose restriction to $|\widehat{\Gamma}|$ is Lipschitz with Lipschitz constant $\varepsilon$.

This definition depends on the structure of $\Gamma$. However, it turns out that contracting edges of length zero and pruning prunable branches do not change the $\varepsilon$-Lipschitz condition.

\begin{lemma}\label{lemma:epsilon Lipschitz}
Let $\Gamma$ and $\Gamma'$ be two oriented string diagrams identified by the attaching map of Definition~\ref{def:space of string diagrams}, and let $\gamma$ be in $LM^k$. 
Then $\gamma$ is $\varepsilon$-Lipschitz with respect to $\Gamma$ if and only if $\gamma$ is 
$\varepsilon$-Lipschitz with respect to $\Gamma'$.
\end{lemma}

\begin{proof}
It suffices to consider the case when $\Gamma'$ is obtained from $\Gamma$ by a codimension one degeneration. In fact, there is only one type of degeneration which affects either the leaf set or components of the intersection graph $\widehat{\Gamma}$ or distances in its pseudometric realization $|\widehat{\Gamma}|$. Namely, it is necessary to prove the lemma for the case $\Gamma'=\Gamma/e$ where $e$ is an external edge of $\widehat{\Gamma}$ of length zero. The contraction has the potential to break the component of $\widehat{\Gamma}$ containing the edge $e$ into multiple components.

Let $c$ be the leaf of $e$ and let $v$ be its internal vertex.
Let $a$ and $b$ be leaves of $\widehat{\Gamma}$ which are in the same component as $c$.

Let $d$ denote the distance functions in  $M$, in $|\widehat{\Gamma}|$, and in $|\widehat{\Gamma'}|$.
First assume $\gamma$ is $\varepsilon$-Lipschitz with respect to $\Gamma$.
Thus for $\widehat{\Gamma}$ we have
\begin{align*}
d(\gamma(\iota(a)), \gamma( \iota(b))) &< \varepsilon d(a,b), \\
d(\gamma(\iota(b)), \gamma( \iota(c))) &< \varepsilon d(b,c), \text{ and}\\ 
d(\gamma(\iota(a)), \gamma( \iota(c))) &< \varepsilon d(a,c).
\end{align*}

The contraction of $e$ identifies the leaves $a$ and $b$ of $\widehat{\Gamma}$ with leaves $a'$ and $b'$ of 
$\widehat{\Gamma'}$
The leaf $c$ may be identified with multiple leaves of $\widehat{\Gamma'}$; exactly one such $c'$ is in the same component as $a'$ and another such $c''$ is in the same component as $b'$.

Recall that the distances between vertices in the $|\widehat{\Gamma}|$ are induced by lengths of edges in $\widehat{\Gamma}$ and likewise for $\widehat{\Gamma'}$.
Because lengths in  $\Gamma'$ are induced by those in $\Gamma$, 
for $\widehat{\Gamma'}$, we have
\begin{align*}
d(\gamma(\iota(b')), \gamma( \iota(c''))) &< \varepsilon d(b',c'') \text{ and} \\
d(\gamma(\iota(a')), \gamma( \iota(c'))) &< \varepsilon d(a',c').
\end{align*}
This shows that if  $\gamma$ is $\varepsilon$-Lipschitz with respect to $\Gamma$ then $\gamma$ is also $\varepsilon$-Lipschitz with respect to $\Gamma'$.

To see the converse, we assume
the previous two inequalities, which yield
\begin{align*}
d(\gamma(\iota(b)), \gamma( \iota(c))) &< \varepsilon d(b,c) \text{ and} \\
d(\gamma(\iota(a)), \gamma( \iota(c))) &< \varepsilon d(a,c).
\end{align*}
We must verify
$d(\gamma(\iota(a)), \gamma( \iota(b))) < \varepsilon d(a,b)$ as well.

If the distance-minimizing path between $a$ and $b$ in $|\widehat{\Gamma}|$ does not pass through the vertex $v$, then contracting $e$ does not affect the path and because $\gamma$ is $\varepsilon$-Lipschitz with respect to $\Gamma'$, we have $d(\gamma(\iota(a)), \gamma( \iota(b))) < \varepsilon d(a,b)$.

If the distance-minimizing path between $a$ and $b$ in $|\widehat{\Gamma}|$ does pass through the vertex $v$, then $d(a,b)$ is equal to $d(a,c)+d(b,c)$ since the length of $e$ is zero.
In this case, we have that 
\begin{align*}
d(\gamma(\iota(a)), \gamma( \iota(b))) &\leq d(\gamma(\iota(a)), \gamma( \iota(c))) + d(\gamma(\iota(b)), \gamma( \iota(c))) \\ &< \varepsilon d(a,c) + \varepsilon d(b,c) \\ &= \varepsilon d(a,b).
\end{align*}
This concludes the proof.
\end{proof}

The construction of the map $\heartsuit$ will require that $\varepsilon$ be small enough relative to the geometry of our manifold $M$. 
In particular, the map $\heartsuit$ is defined using a composition of the straightening maps from Section \ref{section:straightening} and maps from the standard simplex into $M$, which are built using the technique of Riemannian centers of mass\footnote{Also called Karcher means. 
See \cite{Karcher:RCMSCKM}.} pioneered by Grove and Karcher~\cite{GroveKarcher:HCC1CGA}. 
These techniques require strongly convex balls in $M$.
In order to guarantee strongly convex balls, we use the bounds in Definition \ref{def: bound for r}.
The presentation of this material mainly follows~\cite{Sander}; see~\cite{Afsari:RLPCMEUC} for a concise historical review.

\begin{defi}\label{def: bound for r}
Let $M$ be a Riemannian manifold. The {\em pre-convexity radius} of $M$, denoted $\rcx$, is
\[
\rcx \coloneqq  \frac{1}{2}\min\left\{\inj(M),\frac{\pi}{\sqrt{\Delta}}\right\}
\]
where $\inj(M)$ is the injectivity radius of $M$ and $\Delta$ is the supremum of sectional curvatures of $M$. If $\Delta\le 0$, we interpret $\frac{\pi}{\sqrt{\Delta}}$ as $\infty$. 
\end{defi}

\begin{theorem}\label{thm:Karcher}
Let $M$ be a compact Riemannian manifold with pre-convexity radius $r$.
Fix an $r$-ball $B$ in $M$.
Let $F$ be a finite set and let $\Delta_F$ be the simplex spanned by $F$.
Then there exists a smooth map $\Upsilon_F: B^F \times \Delta_F \to B$ called the {\em simplicial geodesic interpolation} such that
\begin{enumerate}
\item $\Upsilon_F$, restricted to $B^F \times F$, is evaluation: $(f, x) \mapsto f(x)$,
\item simplicial geodesic interpolation, restricted to a face of the simplex, is simplicial geodesic interpolation for that face, and
 \item fixing a configuration $f$ in $M^F$, the map $\Upsilon(f,\quad)$ viewed as a map from $\Delta_F$ to $M$ does not depend on the choice of $B$ containing $f(F)$.
\end{enumerate}
\end{theorem}

\begin{remark}
This theorem summarizes Definition 2.2,  Theorem 2.1, and Corollary 2.2 of \cite{Sander}. There the reliance on injectivity radius is erroneously omitted. The correct choice of $r$ and independence of $B$ are given in~\cite[Section 1, Theorem 2.1]{Afsari:RLPCMEUC} with some illuminating discussion.
\end{remark}

The values of $\varepsilon$ that make the construction of $\heartsuit$ possible are $\varepsilon=\frac{\rcx}{|\chi|}$ and $\varepsilon=\frac{\rcx}{2|\chi|}$.
We use the larger value of $\varepsilon$ to define the space $S$, the domain of $\heartsuit$, and we use the smaller value to identify a subspace $s$ of $S$ which we will use later for excision.

\begin{defi}\label{def of S}
The {\em diffuse intersection locus} is the subspace $S$ of $\spaced\times LM^k$ consisting of pairs $(\Gamma,\gamma)$ where $\gamma$ is $\frac{r}{|\chi|}$-Lipschitz with respect to $\Gamma$. 

We also identify the subspace $s$ of $S$ consisting of pairs $(\Gamma,\gamma)$ where $\gamma$ is $\frac{r}{2|\chi|}$-Lipschitz with respect to $\Gamma$.
\end{defi}

\begin{remark}
The closure of $s$ is contained in $S$.
\end{remark}

The map $\heartsuit$ will have codomain $\spaced(M)$, which is a universal space over $\spaced$.

\begin{defi}
The space $\spaced(M)$ consists of pairs $(\Gamma, \Theta)$ where $\Gamma$ is an oriented string diagram and  $ \Theta$ is a map from $|\Gamma|$ to $M$.
\end{defi}

\begin{remark}
The {\em set} $\spaced(M)$ comes equipped with a forgetful map  to the space $\spaced$, $(\Gamma, \Theta) \mapsto \Gamma$.
The fiber over $\Gamma$ is $\Maps(|\Gamma|, M)$; the fiber over a point in a small neighborhood of $\Gamma$ comes with a canonical map to the fiber over $\Gamma$.
We use standard point-set techniques to build a basis generating a topology on $\spaced(M)$ from the topologies on $\spaced$ and on fibers.
The notation $\spaced(M)$ to refers to this topological space.
\end{remark}

We are ready to define the map  $\heartsuit: S \to \spaced(M)$. 
For a pair $(\Gamma,\gamma)$, the map $\heartsuit$ will give us the pair $(\Gamma, \Theta(\Gamma, \gamma))$ where $\Theta(\Gamma, \gamma)$ is a map from $|\Gamma|$ to $M$ extending $\gamma$. 
To define $\Theta(\Gamma, \gamma)$, we need only specify its behavior
on the pseudometric realization of the intersection graph $|\widehat{\Gamma}|$ and ensure it  agrees with $\gamma$ on its leaves.

\begin{defi}\label{defi:theta}
The restriction of  $\Theta(\Gamma, \gamma)$ to $|\widehat{\Gamma}|$ is defined as follows.  
Let $C$ be a component  of $\widehat{\Gamma}$ with leaf set $\leaves{C}$.
First we straighten $C$ to get a map from $|C|$ to $\Delta_\leaves{C}$.
Next, we use $\gamma$ to specify a map from the vertices of $\Delta_\leaves{C}$ into $M$ and apply $\Upsilon_C$ to map $\Delta_{\leaves{C}}$ into $M$.

Formally, for $x$ a point  in $|C|$, we have
\[
\Theta(\Gamma,\gamma)(x) = \Upsilon_{\leaves{C}}(\gamma\circ \iota ,\straighten(x)).
\]
\end{defi}

The attaching map of Definition \ref{def:space of string diagrams} which identifies oriented string diagrams $\Gamma$ and $\Gamma'$ induces a canonical identification of their pseudometric realizations $|\Gamma|$ and $|\Gamma'|$.
We would like to show that if $\gamma$ is $\varepsilon$-Lipschitz with respect to $\Gamma$ and $\Gamma'$, then
 $\Theta(\Gamma, \gamma)$ and $\Theta(\Gamma', \gamma)$ are equal.
We will show this for particular pairs of oriented string diagrams $\Gamma$ and $\Gamma'$ identified under the attaching map; transitivity will imply that it is true for all such pairs.

More specifically, let $\Gamma$ be an oriented string diagram in the boundary of a cell $K$ which is identified with the oriented string diagram ${\Gamma}'$ in the interior of a cell ${K}'$.
The choice of a sequence of codimension one attaching maps identifying $\Gamma$ with ${\Gamma'}$ yields a canonical isomorphism from $|\Gamma|$ to $|{\Gamma}'|$.
This isomorphism is independent of the choice of the sequence.

\begin{prop}\label{prop:heartbreakcommutes}
Let ${\Gamma}$ be in the boundary of  $K$ and let ${\Gamma}'$ be in ${K}'$ as above.
Assume that $\gamma$ is $\varepsilon$-Lipschitz with respect to $\Gamma$ and $\Gamma'$.
Then the maps $\Theta(\Gamma, \gamma)$ and $\Theta(\Gamma', \gamma)$, from $|\Gamma| \cong |\Gamma'|$ to $M$, are equal.
\end{prop}

\begin{proof}

The oriented string diagram $\Gamma' $ is obtained from the oriented string diagram $\Gamma$ by contracting internal edges of length zero and pruning prunable branches.
There are three different types of internal edges of length zero, so in total there are four types of degenerations to consider.
\begin{enumerate}
\item \label{case: contract internal}
The intersection graph  $\widehat{\Gamma}$ has an edge of length zero which is either an internal edge or a external edge whose internal vertex is bivalent.
\item \label{case: gamma 0}
The subgraph $\sqcup Q_i$ of $\Gamma$ has an internal edge of length zero. 
\item \label{case: contract external}
The intersection graph $\widehat{\Gamma}$ has an external edge of length zero whose internal vertex is at least trivalent.
\item \label{case: prunable}
The oriented string diagram ${\Gamma}$ has a tree with a prunable branch.
\end{enumerate}

Lemmas \ref{lemma: contract straighten component}  and \ref{lemma: prune straighten component} show that the straightening of $\Gamma$ commutes with contracting length zero edges of 
 $\widehat{\Gamma}$ of types listed in case (\ref{case: contract internal})
 and pruning branches as in case (\ref{case: prunable}).
 In these cases, the leaf set of each component $C$ of $\widehat{\Gamma}$ is unchanged, so the map $\iota: \leaves{\widehat{\Gamma}} \to \sqcup |Q_i|$ is also unchanged.
This shows that the maps $\Theta(\Gamma, \gamma)$ and $\Theta(\Gamma', \gamma)$ are equal for cases  (\ref{case: contract internal}) and (\ref{case: prunable}).

There is a special case of case (\ref{case: contract internal}) that appears to be missing: where the zero-length edge is a external edge of some tree $T_j$ whose internal vertex is at least trivalent in $T_j$.
In fact, in this case there is always a prunable branch $T_h$ of $T_j$ such that the source of $h$ is this vertex.
Thus, after repeated applications of Lemma \ref{lemma: prune straighten component}, this case is then accounted for by Lemma  \ref{lemma: contract straighten component}.

In case (\ref{case: gamma 0}), the straightening maps of $\Gamma$ and $\Gamma'$, from $|\widehat{\Gamma}| \cong |\widehat{\Gamma'}|$ to the disjoint unions $\sqcup \Delta_{\leaves{C}} \cong \sqcup \Delta_{\leaves{C'}}$, coincide.
This shows that the maps $\Theta(\Gamma, \gamma)$ and $\Theta(\Gamma', \gamma)$ are equal for  case (\ref{case: gamma 0}).

Case (\ref{case: contract external}) will take the most care because contracting an external edge $e$ of the intersection graph of $\Gamma$ whose internal vertex $v$ is at least trivalent changes the leaf set of the intersection graph and may even break $\widehat{\Gamma}$ into more components.

In this case, contraction of the edge $e$ induces a natural map $\varphi$ from $\widehat{\Gamma'}$ to $\widehat{\Gamma}$. The map $\varphi$ identifies corresponding edges of $\widehat{\Gamma'}$ with edges of  $\widehat{\Gamma}$ and identifies corresponding  leaves of $\widehat{\Gamma'}$ with leaves of $\widehat{\Gamma}$ or the vertex $v$. The number of leaves of $\widehat{\Gamma'}$ sent to $v$ by $\varphi$ is equal to one less than the valence of $v$; otherwise $\varphi$ is injective.
In Figure \ref{fig:contract external edge}, $\widehat{\Gamma}$ is equal to $\widetilde{T}_j$ and $\widehat{\Gamma'}$ is equal to the union of $\widetilde{T}_a$ and $\widetilde{T}_{a'}$.
In this example, the map $\varphi$ sends one leaf from each of these subgraphs to the vertex $v$.

To complete the proof, we show that contracting the edge commutes first with straightening and then with the Riemannian center of mass map in the appropriate sense.

Let $\{C\}$ and $\{C'\}$ be the sets of components of $\widehat{\Gamma}$ and $\widehat{\Gamma'}$ respectively. 

There is an induced natural map $\Phi$ from $\sqcup \Delta_{\leaves{C'}}$ to $\sqcup \Delta_{\leaves{C}}$ so that the diagram commutes.
\[
\begin{tikzcd}
{|\widehat{\Gamma'}|}  \ar{dd}[swap]{{|\varphi|}} \ar{rrr}{{\straighten}}
&&&
\sqcup \Delta_{\leaves{C'}} \ar{dd}{\Phi}
\\
\\
{|\widehat{\Gamma}|} \ar{rrr}{{\straighten}}
&&&
 \sqcup \Delta_{\leaves{C}}
\end{tikzcd}
\]
We will define $\Phi$ on the vertices of $\sqcup\Delta_{\leaves{C'}}$, that is, the leaves of $\widehat{\Gamma'}$, and extend linearly in each simplex factor. Any leaf of $\widehat{\Gamma'}$ is taken by $\varphi$ either to a leaf of $\widehat{\Gamma}$ or to the internal vertex of $e$.  The map $\Phi$ takes vertices in the first case to their image under $\varphi$, and takes vertices in the second case to the leaf vertex of $e$.

Since $e$ has length zero, the straightening map for $\Gamma$ takes both vertices of $e$ to the same point in $\sqcup \Delta_{\leaves{C}}$, namely, to the leaf vertex of $e$. This shows that contracting $e$ commutes with straightening.

Finally, since $\Phi$ takes vertices to vertices we have the following commutative diagram.
\[
\begin{tikzcd}
\sqcup \leaves{C'} \ar{rrd}{\sqcup \iota'} \ar{dd}{\Phi}
\\
&& \sqcup S^1 \ar{rr}{\gamma} && M
\\
\sqcup \leaves{C} \ar{rru}[swap]{\sqcup \iota} 
\end{tikzcd}
\]
This shows the maps $\Theta(\Gamma, \gamma)$ and $\Theta(\Gamma', \gamma)$ are equal for case (\ref{case: contract external}).

\end{proof}

\begin{defi}\label{defi:heart}
The function $\heartsuit:S \to \spaced(M)$ is defined as $\heartsuit(\Gamma, \gamma) = (\Gamma, \Theta(\Gamma, \gamma))$. 
\end{defi}

\begin{remark}
The function $\heartsuit$ is well-defined by Proposition~\ref{prop:heartbreakcommutes}. 
The map $\heartsuit$ is continuous because of the following.
First, consider the the image of $K \times LM^k$ in $\spaced \times LM^k$ under the characteristic map.
The restriction of $\heartsuit$ to the intersection of $S$ with this subspace depends continuously on parameters in $K$.
Then, because this intersection is closed in $S$, the map $\heartsuit$ is continuous by the gluing lemma.
\end{remark}

\section{The push-pull map and the string topology construction}\label{section:push-pull map}

All but one of the ingredients for the string topology construction are now in place.
In this section we fix an arbitrary singular cochain $W$ in $C^*(S, S-s)$ and define a map called the {\em push-pull map for $W$}.
We will eventually restrict to the subset of cocycles representing a particular homology class in $H^{\chid}(S, S-s)$, called the {\em diffuse intersection class}.
The push-pull map for a cocycle representing the diffuse intersection class will be called a {\em string topology construction}.
The definition of the diffuse intersection class class is rather involved, and its definition is the goal of
Sections \ref{section:thom} and \ref{section:diffuse intersection class}. Again, everything in this section would work fine if we replaced $\spaced$ with its unoriented version $\spaceduo$; we will only use the orientation to build the diffuse intersection class.
We work with coefficients in $R$, an arbitrary commutative ring with identity; we suppress the notation throughout.

Let $M^d$ be a closed, $R$-oriented $d$-dimensional Riemannian manifold.
Then for each $\chi$, $k$, and $\ell$, 
the push-pull map for a degree $|W|$ cochain $W$  in $C^*(S, S-s)$  is a degree $-|W|$ map
\[\mathcal{ST}_W: \chains(\spaced ) \otimes \chains(LM)^{\otimes k} \to \chains(LM)^{\otimes \ell},
\]
which we express as a composition of two maps, given in Definitions \ref{def:in} and \ref{def:out} below.

The first map in this composition is our version of a wrong-way map given by a Pontryagin--Thom construction.
\begin{defi}\label{def:in}
Let \[(\rho_{in})_{!}: \chains(\spaced ) \otimes \chains(LM)^{\otimes k}  \to C_{*-|W|}(\mathcal{SD}(M))\] be the following composition of maps:

\begin{enumerate}

\item
first,
$\chains (\spaced ) \otimes \chains(LM)^{\otimes k} \to \chains(\spaced  \times LM^k)$, an Eilenberg--Zilber map for singular chains,
\item
next,
$\chains(\spaced  \times LM^k) \to \chains(\spaced  \times LM^k, \spaced  \times LM^k - 
s)$, the usual quotient map from absolute to relative chains, 
\item
next,
$\chains(\spaced  \times LM^k, \spaced  \times LM^k - s) \to \chains(S, S- s)
$ a chain homotopy inverse to the map induced on chains by the inclusion of spaces
$ {(S, S- 
s)
 \hookrightarrow (\spaced  \times LM^k, \spaced  \times LM^k - 
s)}
 $
\item
next,
$\chains(S, S- s) \to C_{*-|W|}(S)$, given by the cap product with the relative cochain $W$, and
\item 
finally,
$\heartsuit_*: C_{*-|W|}(S) \to C_{*-|W|}(\mathcal{SD}(M))$, the map induced on singular chains by the map of spaces $\heartsuit: S \to \mathcal{SD}(M)$.
\end{enumerate}
For concreteness, we choose the so-called shuffle map~\cite[VI.12.26.2]{Dold:LAT} as our Eilenberg--Zilber map and use the explicit formula implementing iterated barycentric subdivision in~\cite[Theorem~2.20]{HATCHER} as our chain homotopy inverse to the map induced on chains by inclusion. 
\end{defi}
\begin{remark}
All of the maps involved in this definition are chain maps except potentially the cap product with $W$, which is only a chain map if $W$ is a relative cocycle. 
\end{remark}

The second map in the composition defining $\mathcal{ST}_W$ is also itself a composition, albeit a simpler one. First recall that for a string diagram $\Gamma$ with marked output boundary cycle $C$ we have the map $\bar{\partial}_C$ from the standard circle to $|\Gamma|$ which transverses the oriented edges of $C$ in ``reverse order.''

\begin{defi}\label{def:out}
Let
\[(\rho_{out})_*: \chains(\mathcal{SD}(M)) \to \chains(LM)^{\otimes \ell}\]
be the following composition of chain maps:
\begin{enumerate}
\item
first, 
$\chains(\mathcal{SD}(M)) \to \chains(LM^\ell)$, induced by the map of spaces $\rho_{out}$ from $\mathcal{SD}(M)$ to $LM^\ell$ which takes $(\Gamma, \Theta: |\Gamma| \to M)$ to the pullback of $\Theta$ to its outputs along the disjoint union of $\bar{\partial}_C$ over all output boundary cycles of $\Gamma$, and
\item
next,
$\chains(LM^\ell) \to \chains(LM)^{\otimes \ell}$, an Eilenberg--Zilber map.
\end{enumerate}
For concreteness we choose the Alexander--Whitney map as our Eilenberg--Zilber map in this direction.
\end{defi}

We come to the main definition of the paper.

\begin{defi}
Let $W$ be a cochain in $C^*(S, S-s)$.
The {\em push-pull map for $W$} $\mathcal{ST}_W$ is the
degree $-|W|$ map given by the 
 composition of the above maps:
\[ \mathcal{ST}_W \coloneqq (\rho_{out})_* \circ (\rho_{in})_!: 
 \chains(\spaced ) \otimes \chains(LM)^{\otimes k} \to C_{*-|W|}(LM)^{\otimes \ell}.
\]
\end{defi}

\begin{prop}
If $W$ is a cocycle in $C^*(S, S-s)$ then $\mathcal{ST}_{W}$ is a chain map. If $W$ and $W'$ are cohomologous cocycles in $C^*(S, S-s)$ then $\mathcal{ST}_{W'}$ and $\mathcal{ST}_W$ are chain homotopic chain maps.
\end{prop}

\begin{proof}
If $W$ is a cocycle then $\mathcal{ST}_W$ is a composition of chain maps.

Let $\widetilde{W}$ be a cochain with coboundary $W-W'$. Then $\mathcal{ST}_{\widetilde{W}}$ is a chain homotopy between $\mathcal{ST}_W$ and $\mathcal{ST}_{W'}$.
\end{proof}

\begin{remark}
This proof shows that the string topology construction $\mathcal{ST}_{(\ )}$ is a chain map from $C^{-*}(S,S-s)$ to $Hom(\chains(\spaced ) \otimes \chains(LM)^{\otimes k}, \chains(LM)^{\otimes \ell})$, the homomorphism complex.
\end{remark}

Sections \ref{section:thom} and \ref{section:diffuse intersection class} are devoted to picking out a particular cohomology class $\Omega$ in $H^{\chid}(S, S-s)$ called the {\em diffuse intersection class}.

\begin{defi}\label{defi:string topology construction}
The {\em string topology construction for $W$} is the push-pull map $\mathcal{ST}_W$ where $W$ is a cocycle representing the diffuse intersection class $\Omega$.
\end{defi}

\section{Patching cohomology classes over the space of oriented string diagrams}\label{section:thom}

Our next goal is to define a cohomology class in $H^{|\chi| d}(S, S - s)$.
This relative cohomology class on $S$ will be defined as the pullback of a cohomology class on a {\em stratified pair} $\Mspace{\spaced}$ under an evaluation map from $S$ to $\Mspace{\spaced}$.

\subsection{The stratified pair}
In this section we define the pair $\Mspace{\spaced}$ as a colimit of a diagram in pairs of spaces. Just as for the space $S$, these pairs of spaces will involve some closeness condition in $M$. We will again employ the preconvexity radius $r$ (see Definition~\ref{def: bound for r}).

Up until this point, it has not been important for us to distinguish between the cells of $\spaced$, viewed as abstract polytopes, and the cells of $\spaced$, viewed as subspaces of $\spaced$.
In this section, by the {\em cell} $K$, we mean the abstract polytope, which comes equipped with a  characteristic map $\characteristic_K: K \to \spaced$.
Additionally, by a {\em face} of $K$, we mean a face $\widehat{K'}$ of the abstract polytope $K$.
The face $\widehat{K'}$ comes equipped with a canonical identification $\at : \widehat{K'} {\xrightarrow{\cong}} K'$,
with a particular cell $K'$ of $\spaced$.
By a {\em degeneration} of $K$, we mean such a cell $K'$, equipped with the inclusion $K'{\xrightarrow{\cong}}\widehat{K'}\hookrightarrow K$, whose image is $\widehat{K'}$.

\begin{defi}
Fix an oriented combinatorial string diagram and the corresponding cell $K$ of $\spaced$.
Let $\mathcal{T}$ denote the set of trees $\{T_j\}$ of the combinatorial string diagram and let $\mathcal{L}$ denote the disjoint union of the leaf sets $\leaves{T_j}$. Both of these sets depend on $K$.
Let $N_K$ be the subspace of $M^{\mathcal{L}}$ consisting of maps $f$ from $\mathcal{L}$ to $M$ such that for all trees $T_j$ in $\mathcal{T}$, the image of the restriction of $f$ to $\leaves{T_j}$ lies in an $r$-ball in $M$.
\end{defi}

Next we want to discuss the combinatorics of $N_K$ and how it varies as we move around in $\spaced$. To this end, we consider codimension one degenerations $K'$ of a cell $K$. Recall that there are two types of codimension one degenerations, one given by contracting edges and one given by pruning branches. We call these {\em contraction} and {\em pruning} degenerations respectively.

For a codimension one degeneration $K'$ of $K$, there is canonical map $\mathcal{T}'\to\mathcal{T}$ between the sets of trees of the corresponding oriented combinatorial string diagrams.
This map is bijective for a contraction degeneration.
In a pruning degeneration, the map is surjective and is one-to-one outside of $T^h$ and $T_h$; it sends both $T^h$ and $T_h$ in $\mathcal{T}'$ to the tree $T_j$ in $ \mathcal{T}$.

This canonical map of sets of trees does not, in general, induce a map of sets of leaves.
However, it does give rise to a map in the other direction $\xi: \mathcal{L} \to \mathcal{L}'$ where a leaf of a tree in $\mathcal{T}$ is sent to the corresponding leaf of a tree in $\mathcal{T}'$.
This map is also bijective for a contraction degeneration.
In a pruning degeneration, $\xi$ is injective but misses exactly one leaf $w$ of one tree $T_h$  in $\mathcal{T}'$, namely, the source of $h$ in $T_h$.

Note that if $K'_1$ and $K'_2$ are codimension one degenerations of the cell $K$, and $K''$ is a common codimension one degeneration of $K_1'$ and $K_2'$,
then the compositions $\mathcal{L} \to \mathcal{L}_1' \to \mathcal{L}''$ and $\mathcal{L} \to \mathcal{L}'_2 \to \mathcal{L}''$ need not agree; the following diagram need not commute.
See Figure \ref{fig:NonCommutativeLeafMap}.
This possible discrepancy will be shown not to matter for our purposes.

\begin{center}
\begin{tikzcd}
\mathcal{L}
\ar{rr}
\ar{d}
&&
\mathcal{L}'_1
\ar{d}\\
\mathcal{L}'_2
\ar{rr}
&&
\mathcal{L''}
\end{tikzcd}
\end{center}

\begin{figure}
\includegraphics[scale=0.8]{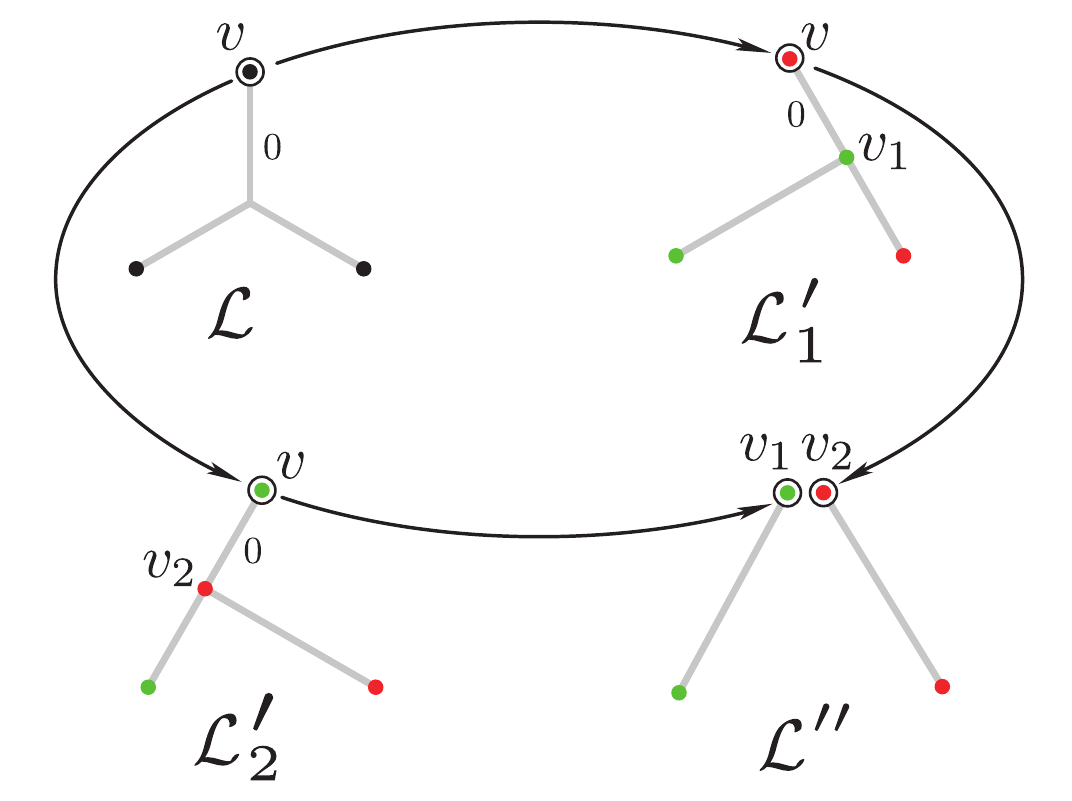}
\caption{This figure shows the noncommutativity of composition of  $\xi$ maps for two codimension two degenerations. Length-zero edges are labeled. The leaf for which the values of the two compositions are different is highlighted.}
\label{fig:NonCommutativeLeafMap}
\end{figure}

We will use $\xi$ and the Riemmanian center of mass map $\Upsilon$ to build a kind of attaching map for building blocks of the form $K\times N_K$. This attaching map, roughly, extends a map $\mathcal{L}\to M$ to a map $\mathcal{L}'\to M$. We begin by defining a map $\bigstarr{\widehat{K'}}{}$ from $\widehat{K'}\times N_K$ to $K'\times M^{\mathcal{L}'}$ and in Lemma \ref{lemma:nabla N_K'} show that its image lies in $N_{K'}$.

For a pair $(\Gamma,f)$ in $\widehat{K'}\times N_K$, the map $\bigstarr{{\widehat{K'}}}{}$ will give us the pair $(\at(\Gamma), \Psi(\Gamma,f))$ where $\Psi(\Gamma,f)$ is a map from $\mathcal{L}_{K'}$ to $M$ extending $f\xi^{-1}$. For a contraction degeneration, this suffices to define $\Psi(\Gamma,f)$. For a pruning degeneration, we need only specify the value of $\Psi(\Gamma,f)$ on the leaf $w$ that is not in the image of $\xi$.

\begin{defi}
Let $(\Gamma,f)$ be in $\widehat{K'}\times N_K$, where $K'$ is a pruning degeneration of $K$. The value of $\Psi(\Gamma,f)$ on $w$ is defined as follows.
Since $w$ is a leaf of $T_h$, it is also a vertex of $T_j$. In the following, $\straighten$ denotes the straightening map $|T_j|\to \Delta_{\leaves{T_j}}$ of Proposition~\ref{prop:degenerate straighten}.

Let $\iota$ be the inclusion of $\leaves{T_j}$ into $\mathcal{L}$. Then formally,

\[
\Psi(\Gamma,f)(w) = \Upsilon_{\leaves{T_j}}(f\circ \iota ,\straighten(w)).
\]

\end{defi}
\begin{defi}
The map $\bigstarr{\widehat{K'}}{}: \widehat{K'} \times N_K \to K' \times N_{K'}$ is defined as 
\[\bigstarr{\widehat{K'}}{}(\Gamma,f)=(\at(\Gamma),\Psi(\Gamma,f)).\]
\end{defi}
Compare these definitions to Definitions~\ref{defi:theta} and~\ref{defi:heart}. In both cases, we use the composition of straightening and the Riemmannian center of mass maps. As a result, the maps $\heartsuit$ and $\nabla$ have good compatibility, which will be exploited later.

When $K$ and $K'$ are clear from context we will suppress them and use $\bigstarr{}{}$ to refer to $\bigstarr{\widehat{K'}}{}$. This is the beginning of a relentless campaign of abuse of notation where any map with a passing resemblance to any $\bigstarr{\widehat{K'}}{}$ will be referred to with the notation $\bigstarr{}{}$.

\begin{lemma}\label{lemma:nabla N_K'}
The image of $\bigstarr{}{}$ is contained in $K' \times N_{K'}$.
\end{lemma}
\begin{proof}
For a contraction degeneration, the map $\bigstarr{}{}$ is a homeomorphism to $K' \times N_{K'}$.
This is because the condition of being in $N_K$ and in $N_{K'}$ are the same.
For a pruning degeneration, the proof is also straightforward.
By assumption, the leaves of $T_j$ are sent into an $r$-ball in $M$.
Therefore, the leaves of $T^h$ are sent into the same $r$-ball and $\Upsilon_{\leaves{T_j}}(f\circ \iota ,\straighten(w))$ lies in the same $r$-ball, so the leaves of $T_h$ also lie into the same $r$-ball.
\end{proof}

\begin{remark}
The map $\bigstarr{}{}: \widehat{K'} \times N_K \to K' \times N_{K'}$ is injective.
\end{remark}

We use the notation $\length{T}$ to refer to the total length of the tree $T$.

\begin{defi}
Let $n_K$ be the subspace of $M^{\mathcal{L}}$ consisting of maps $f$ from $\mathcal{L}$ to $M$ such that the image of the restriction of $f$ to the set of leaves $\leaves{T_j}$ of any tree $T_j$ lies in a $\frac{\length{T_j}}{|\chi|^2} \frac{r}{4}$
 ball in $M$.
\end{defi}

\begin{remark}
Since $\length{T_j}$ is  less than or equal to $|\chi|$, the closure of $n_K$ is a subspace of $N_K$.

\end{remark}

\begin{lemma}\label{lemma:pre-is-good-enough}
The map $\bigstarr{}{}$ sends $\widehat{K'} \times (N_K - n_K)$ into $K' \times (N_{K'} -n_{K'})$.

\end{lemma}

\begin{proof} 
We show that the intersection of $(K'  \times n_{K'})$ with the image of $\bigstarr{}{}$ is contained in $\bigstarr{}{}(\widehat{K'} \times n_K)$. For a contraction degeneration, the proof is trivial so we consider only pruning degenerations. For notational convenience, let $\delta = \frac{r}{4|\chi|^2}$.

Assume $(\Gamma' ,f')$ is in the intersection of $(K'  \times n_{K'})$ with the image of $\bigstarr{}{}$.
This means that $f'$ sends the leaves $\leaves{T^h}$ into a $\delta\length{T^h}$-ball $B^h$ in $M$ and
$f'$ sends the leaves $\leaves{T_h}$ into a $\delta\length{T_h}$-ball $B_h$ in $M$.
Since $(\Gamma', f')$ is in the image of $\bigstarr{}{}$, the map $f'$ sends $w$ to $\Upsilon_{\leaves{T_j}}(f\circ \iota ,\straighten(w))$.
This means that $f'$ sends $w$ into $B^h$.

Therefore the union of $B^h$ and $B_h$ is contained in a $\delta(\length{T^h} + \length{T_{h}})$-ball in $M$.
See Figure \ref{fig:EpsilonScales}.

\begin{figure}
\includegraphics[scale=0.7]{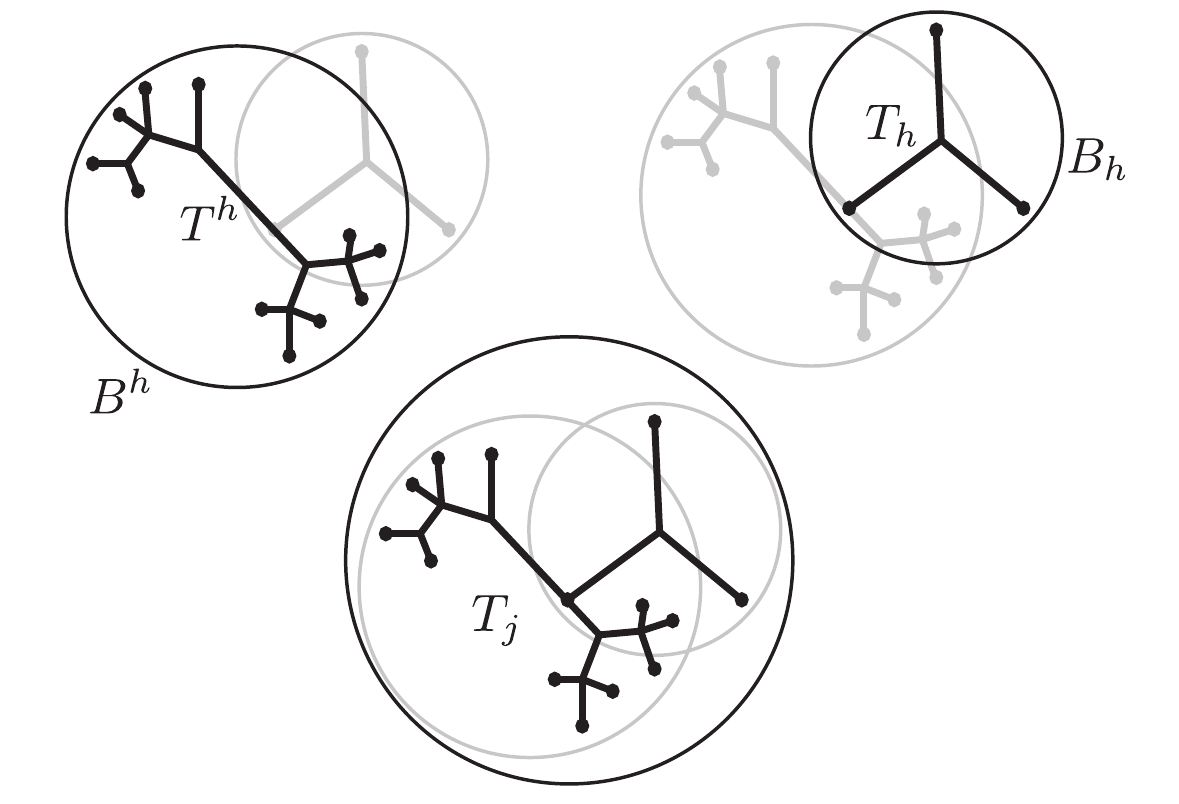}
\caption{The map $f'$ sends leaves of $T_j$  into a $\delta\length{T_j}$-ball.}
\label{fig:EpsilonScales}
\end{figure}

But $\leaves{T_j}$ is contained in the union of $\leaves{T^h}$ and $\leaves{T_h}$ and $\length{T_j} = \length{T^h} + \length{T_h}$.
Therefore $f'$ sends the leaves of $T_j$ into a $\delta\length{T_j}$-ball in $M$.
This shows that $(\Gamma', f')$ is in $\bigstarr{}{}(\widehat{K'} \times n_K)$.
\end{proof}

\begin{cor}
The map $\bigstarr{}{}$ is a map of pairs of spaces 
\[(\widehat{K'} \times N_K, \widehat{K'} \times (N_K - n_K)) \to (K' \times N_{K'}, K'  \times (N_{K'} - n_{K'})).\]
\end{cor}

We will need to refer to these types of pairs of spaces a number of times below.
To simplify notation, for $J$ a subspace of a cell $K$ of $\spaced$, we denote the pair
$(J \times N_K, J \times (N_K - n_K))$ by $\Mover{J}{K}$ and refer to the pair as {\em the $K$-product over $J$}. In this notation, $\bigstarr{}{}$ is a map of pairs
\[
\Mover{\widehat{K'}}{K}\to \Mover{K'}{K'}.
\]

We arrive at the first main definition of this section, that of our stratified space $\Mspace{\spaced}$, which is assembled from all such $K$-products over the cells $J$ and the $\bigstarr{}{}$ maps.

\begin{defi}\label{defi:coequalizer}
Let $\Mspace{\spaced}$ be the coequalizer of the following diagram
\[
\coprod_{{K'}, K}\Mover{\widehat{K'}}{K}\rightrightarrows \coprod_{K}\Mover{K}{K}
\]
Here the first disjoint union is over pairs of cells of $\spaced$ where $K'$ is a codimension one degeneration of $K$, and the second disjoint union is over cells of $\spaced$. The two maps are 
\begin{itemize}
\item 
the maps $\Mover{\widehat{K'}}{K} \hookrightarrow \Mover{K}{K}$ induced by the inclusions of faces ${\widehat{K'} \hookrightarrow K}$ and 
\item
 the maps $\bigstarr{}{}:\Mover{\widehat{K'}}{K}\to \Mover{K'}{K'}$.
 \end{itemize}
\end{defi}
The projection maps $K \times N_K \to K$ and $K \times (N_K - n_k) \to K$ induce a well-defined projection map of pairs of spaces from $\Mspace{\spaced} \to (\spaced, \spaced)$.
Given an inclusion $V \hookrightarrow \spaced$, the {\em space over $V$}, $\Mspace{V}$ is the preimage of $(V,V)$ in $\Mspace{\spaced}$ under the projection map.
(Because the characteristic map $\characteristic_K: K \to \spaced$ is typically not injective on $\partial K$, we avoid the notation $\Mspace{K}$ and use $\Mspace{\characteristic_K(K)}$ instead.)

Informally, we can think of $\Mspace{\spaced}$ as being built inductively. If $K$ is an $m$-cell and we have already built the space over the $(m-1)$-skeleton of $\spaced$, then we glue $\Mover{K}{K}$ in using the various $\bigstarr{}{}$ maps as some sort of attaching maps.

A priori, this metaphor is not justified because the attaching could be poorly behaved. This is because of the that fact if $K'_1$ and $K'_2$ are codimension one degenerations of $K$ which share a common codimension one degeneration $K''$, the compositions of maps $\xi$ of leaves $\mathcal{L} \to \mathcal{L}'_{1} \to \mathcal{L}''$ and $\mathcal{L} \to \mathcal{L}'_{2} \to \mathcal{L}''$ need not agree. 

As a consequence, it might be possible that gluing in via the ``attaching maps'' could disrupt the topology of the space over the $(m-1)$-skeleton of $\spaced$.

Fortunately, this does not occur. In the next lemma, we show that the compositions of the corresponding $\bigstarr{}{}$ maps do agree.

\begin{lemma}\label{lemma:starscommute}
Let $K'_1$ and $K'_2$ be codimension one degenerations of $K$ which share a common codimension one degeneration $K''$.
Then the two compositions of maps of pairs from the $K$-product over $\widehat{K''}$ to the $K''$-product over $K''$---one factoring through the $K_1'$-product over $\widehat{K''}$ and the other factoring through the $K_2'$-product over $\widehat{K''}$---agree.
Specifically, the following diagram commutes, where maps are understood to refer to their appropriate restrictions:

\begin{center}
\begin{tikzcd}
\Mover{\widehat{K''}}{K}
\ar{rr}{\bigstarr{\widehat{K'_1}}{}}
\ar{dd}[swap]{\bigstarr{\widehat{K'_2}}{}}
&&
\Mover{\widehat{K''}}{K'_1}
\ar{dd}{\bigstarr{\widehat{K''}}{}}\\\\
\Mover{\widehat{K''}}{K'_2}
\ar{rr}[swap]{\bigstarr{\widehat{K''}}{}}
&&
\Mover{K''}{K''}
\end{tikzcd}
\end{center}

\end{lemma}
\begin{proof}

Contracting a zero length edge of $\Gamma$ may be a degeneration of codimension greater than one. This phenomenon occurs for an external edge of a tree $T_j$ whose internal vertex is at least trivalent. For the purposes of this proof we call such edges {\em non-contractible in $\Gamma$}.

Consider the codimension two degeneration $K\to K''$. Pruning degenerations increase the total number of trees in a representative oriented  string diagram and all other codimension one degenerations preserve the number of trees. Therefore the number of pruning degenerations in any decomposition $K\to K'\to K''$ is constant. We divide into cases.

First, if there are no pruning degenerations, then the statement is trivially true because everything commutes and every map is an isomorphism.

Next, consider a codimension two degeneration sequence $K\to K'_1\to K''$ such that exactly one of the two codimension one degenerations is a pruning degeneration. The other codimension one degeneration must be a contraction degeneration along an edge $e$ of $\Gamma$ or $\Gamma'_1$.

If $e$ is contractible in $\Gamma$,
then the unique other degeneration sequence ${K\to K'_2\to K''}$ giving rise to the same codimension two degeneration consists of the same two codimension one degenerations applied in the opposite order. 
Contracting zero-length edges does not change the combinatorics of the map $\xi$ on leaf sets, nor do these actions affect the straightening or Riemannian center of mass maps, which give the missing component of $M^{\mathcal{L}''}$.

There is a  case involving contracting an edge which is non-contractible in $\Gamma$ (this is the most involved case). That is, after performing a pruning degeneration on $\Gamma$ to yield $\Gamma'$, there may be an edge which is contractible in $\Gamma'$ but not in $\Gamma$. In this case, the other degeneration sequence involves performing a different pruning degeneration followed by a contraction degeneration. Overall, this codimension two degeneration is realized by the contraction of an external edge of a tree $T_j$ whose non-leaf vertex is trivalent. Recall Figure~\ref{fig:NonCommutativeLeafMap}. 
 
In this case, the maps on leaves $\mathcal{L}\to \mathcal{L}''$ do not commute. There are two leaves $v_1$ and $v_2$ of two trees in $\Gamma''$ which are identified in $\Gamma''$. The leaves $v_1$ and $v_2$ correspond to a vertex $v$ in $\at(\Gamma'')$, the oriented string diagram in $K$.

Examining the compositions  $\bigstarr{\widehat{K''}}{}\circ\bigstarr{\widehat{K'_2}}{}$ and $\bigstarr{\widehat{K''}}{}\circ\bigstarr{\widehat{K'_1}}{}$ reveals that the two maps differ only in how they deal with $v_1$ and $v_2$. The first composition identifies $v_1$ with $v$ and uses the straightening map to decide what to do with $v_2$, while the second composition does the reverse. In both cases the second degeneration corresponds to contracting the edge between $v_1$ and $v_2$. By Lemma~\ref{lemma: contract straighten} this means that the straightening map identifies $v_1$ and $v_2$ in the appropriate simplex so that the two maps coincide after all.

Finally, in the case that both codimension one degenerations are pruning degenerations, then there is no confusion about leaf sets, which are canonically identified with one another, but we must ensure that the various straightening maps agree. There are three subcases, but all are dealt with by repeated applications of Lemma~\ref{lemma: graft straighten}. See Figure \ref{fig:Prune-prune}.
\end{proof}
\begin{figure}
\includegraphics[scale=0.75]{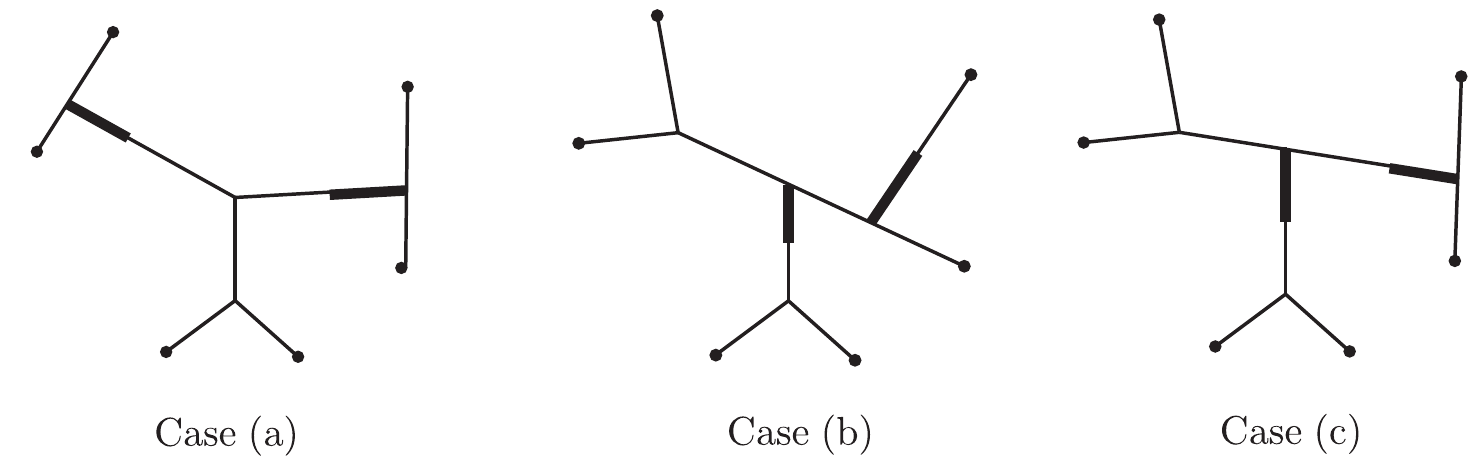}
\caption{These are examples of the three types of codimension two degenerations arising from distinct pruning degenerations. The pruning edges may be part of the same tree in $\Gamma''$ (case a) or may be part of distinct trees. If they are distinct trees, they may share a common pollard (case b) or one may be (contained in) the pollard of the other (case c).}
\label{fig:Prune-prune}
\end{figure}

\begin{cor}
For every cell $K$, the natural map $\Mover{K}{K}\to \Mspace{\characteristic_K(K)}$ is an inclusion of pairs; restricted to the interior $\mathring{K}$ of $K$, the map $\Mover{\mathring{K}}{K}\to \Mspace{\mathring{K}}$ is an isomorphism.
\end{cor}

Notice that Lemma \ref{lemma:starscommute} allows us to make the following definition.

\begin{defi}\label{def: starr}
Let $K'$ be a degeneration of the cell $K$ of arbitrary codimension. 
Then the map $\bigstarr{}{}: \Mover{\widehat{K'}}{K} \to \Mover{K'}{K'}$ is defined to be the composition ${\bigstarr{\widehat{K'}}{}\circ\bigstarr{\widehat{K_n}}{} \circ \cdots \circ \bigstarr{\widehat{K_1}}{}}$ where $K' \hookrightarrow K_n \hookrightarrow K_{n-1} \hookrightarrow \cdots \hookrightarrow K_1 \to K$ is a sequence of codimension one inclusions of cells.
\end{defi}

By Lemma \ref{lemma:starscommute}, the definition of  $\bigstarr{}{}$ is independent of the choice of this sequence.

\subsection{Patching cohomology classes}

In this section we describe how to assemble coherent collections of relative cohomology classes on $K$-products to give global relative cohomology classes on $\Mspace{\spaced}$.
We begin with a few lemmas about $K$-products and we return to the space $\Mspace{\spaced}$ in Theorem \ref{thm:patching}.

\begin{lemma}\label{lemma: disk bundle over diagonal}
The space $N_K$ is homeomorphic to the total space of a $\chid$-disk bundle over the space $M^\mathcal{T}$; the space $N_K - n_K$ is homotopy equivalent to the total space of the $(\chid - 1)$-sphere bundle of the disk bundle.
\end{lemma}

\begin{proof}
The inclusion of leaves in a tree defines a map from $\mathcal{L}$ to $\mathcal{T}$, which induces a diagonal embedding  of $M^{\mathcal{T}}$ into $M^{\mathcal{L}}$.
Since we are working inside the injectivity radius of $M$, we can project from $N_K$ to $M^\mathcal{T}$, for example, by projecting to a single leaf coordinate in each tree.
Again, because we are inside the injectivity radius, this gives $N_K$ the structure of a disk bundle.
A similar argument shows that $n_K$ is also a disk bundle whose closure is contained in the interior of $N_K$.
The complement $N_K - n_K$ is then homotopy equivalent to the sphere bundle of the disk bundle.

The calculation of the codimension of the embedding $M^\mathcal{T} \hookrightarrow M^\mathcal{L}$ requires an intermediate Euler characteristic argument.
Recall that an oriented string diagram $\Gamma$ can be written as the union of the $Q_i$, the $T_j$, and the $L_i$, modulo identification of vertices.
The  identifications occur at all leaves of $T_j$ and at one leaf of each $L_i$.
The overall Euler characteristic is the total number of $T_j$, plus the total number of $L_i$, minus the number of such identifications.
Thus the Euler characteristic of $\Gamma$ is equal to the cardinality of $\mathcal{T}$, minus the cardinality of $\mathcal{L}$.

The diagonal embedding $M^\mathcal{T} \hookrightarrow M^\mathcal{L}$ is a product over all trees $T_j$ in $\mathcal{T}$ of diagonal embeddings $M \hookrightarrow M^{\leaves{T_j}}$. 
For each $T_j$ the diagonal embedding has codimension $(\leaves{T_j} - 1)d$.
Therefore, the codimension of the diagonal embedding $M^\mathcal{T} \hookrightarrow M^\mathcal{L}$ is equal to $\sum_{T_j} (\leaves{T_j}- 1)d$.
By the Euler characteristic calculation above, the codimension of the diagonal embedding has codimension $\chid$.

\end{proof}

\begin{remark}
The space $N_K$ is homeomorphic to the disk bundle of the normal bundle of the diagonal embedding.
The space $N_K - n_K$ is homotopy equivalent to the sphere bundle of this disk bundle.
This conceptually explains the connection between our construction and others \cite{CJ, CG, godin, HingstonWahl, Kupers:CHSOMURSC}.
Our only use of these facts in this paper will be to connect our work to theirs. See Section~\ref{section:previouswork}.
\end{remark}

\begin{lemma}\label{lemma:Leray--Serre}
Let $K$ be a cell of $\spaced \gkl$.
Then the restriction map \[H^\ast(N_K)\to H^\ast(N_K-n_K)\] is an isomorphism for $\ast<\chid-1$ and injective for $\ast=\chid-1$.
\end{lemma}
\begin{proof}

The statement is vacuous for $\chid<1$; for $\chid=1$ it is a straightforward statement about the normal bundle of the diagonal map from the circle to the torus which is left to the reader's imagination. So for the purpose of this proof assume $\chid>1$.

The map at the level of cohomology between these two bundles is induced by a map of the respective Leray--Serre spectral sequences. 

The entries of the $E_2$ page of the spectral sequence for $N_K$ are $H^p(M^{\mathcal{T}}, H^q(\mathbb{R}^{\chid}))$. The only nonzero groups occur when $q=0$ so the sequence collapses at $E_2$.
On the other hand, the entries of the $E_2$ page of the spectral sequence for $N_K - n_K$ are $H^p(M^{\mathcal{T}}, H^q(S^{\chid - 1}))$.
Therefore the nonzero entries on the $E_2$ page must have $q=0$ or $q=\chid-1$. 
The map of spectral sequences on the $E_2$ page is an isomorphism onto the bottom row.

While the spectral sequence for the sphere bundle need not collapse at page $E_2$, most differentials beginning with $\partial_2$ are zero for degree reasons. The only possibly nonzero differential is the component of $\partial_{\chid}$ which goes from 
the entry in position $(0,\chid-1)$ to the entry in position $(\chid,0)$ on the $E_{\chid}$ page.

Thus the $q=0$ row of the spectral sequence for the sphere bundle survives to $E_\infty$ below degree $\chid$. This gives us injectivity in the desired range. Further, the lowest degree present in the $q=\chid-1$ row of the spectral sequence for the sphere bundle is $\chid-1$. Thus we also get surjectivity in the desired range.

If $M$ is not simply connected, the cohomology groups giving the entries of the $E_2$ pages of both spectral sequences should be interpreted as having coefficients in local systems, but this does not affect the validity of the argument.
\end{proof}
\begin{lemma}\label{lemma:rank-zero}
For any cell $K$ of $\spaced \gkl$ and any degree $m<\chid$, the cohomology group $H^{m}(N_K, N_K - n_K)$ vanishes.
\end{lemma}

\begin{proof}
We consider the long exact sequence of the pair $(N_K, N_K - n_K)$:
\[
\xymatrixrowsep{20pt}
\xymatrixcolsep{20pt}
\xymatrix{
\save[]+<-55pt,+10pt>*{}="b"\restore \qquad\qquad\qquad\qquad\cdots\ar[r]
& H^{m-1}(N_K)\ar[r]^-{\psi}
& H^{m-1}(N_K-n_K)
\save[]+<+0pt,-16pt>*{}="a"\restore
\ar@{}[d]|(.45){}="a"
\ar `r/5pt[d]`"a" `^d"b" `[dll][dll]
\\
H^{m }(N_K, N_K - n_K)\ar[r]^-{\varphi}
&H^{m}(N_K)\ar[r]^-{\psi}
& H^{m}(N_K - n_K)\ar[r]
&
}\]


In degree $m$ the restriction map $\psi$ is injective by Lemma~\ref{lemma:Leray--Serre}, so $H^{m}(N_K, N_K - n_K)$ is the image of the connecting homomorphism.
In degree $m-1$ the map $\psi$ is surjective by Lemma \ref{lemma:Leray--Serre}, so the connecting homomorphism is  the zero map.
The result then follows by exactness.
\end{proof}
\begin{cor}\label{cor:nolowcohooversphere}
For any cell $K$, the cohomology group $H^{\chid - 1}(\Mover{\partial K}{K})$ vanishes.
\end{cor}
\begin{proof}
The $K$-product over $\partial K$ is a bundle of pairs of spaces over $\partial K$, which is trivial as it is the restriction of the $K$-product over the cell $K$.
The corollary follows by the K\"unneth formula. There is no torsion contribution because the cohomology of $\partial K$ is flat.
\end{proof}

\begin{cor}\label{cor:partial K}
Let $\widehat{K'}$ be a face of the cell $K$ of $\spaced$.
The restriction map $H^{\chid}(\Mover{\partial K}{K}) \to H^{\chid}(\Mover{\widehat{K'}}{K})$ is injective.
\end{cor}
\begin{proof}
Again, by the K\"unneth formula and Lemma \ref{lemma:rank-zero},  this restriction map is equivalent to 
\[
H^0(\partial K) \otimes H^{\chid}(N_K, N_K - n_K) \xrightarrow{\iota^* \otimes \identity} H^0(K') \otimes H^{\chid}(N_K, N_K - n_K)
\]
where $\iota$ is the inclusion of $\widehat{K'}$ into $\partial K$.
In degree 0, $\iota^*$ is an isomorphism.
\end{proof}

\begin{defi}
Let $\{ \tau_K \}$ be a collection of classes in $H^{\chid}(\Mover{K}{K})$, as $K$ varies over all cells of $\spaced$.
Assume that for each $K$ and each codimension one degeneration $K'$ of $K$ that
 $\tau_K$ and $\tau_{K'}$ pull back to the same class in $H^{\chid}(\Mover{\widehat{K'}}{K})$ under the maps induced by inclusion and $\bigstarr{}{}$ respectively:
\[
\iota^\ast(\tau_K) = {\bigstarr{}{}}^\ast(\tau_{K'}).
\]
Then we call the collection of classes $\{ \tau_K \}$ {\em coherent}.
\end{defi}

\begin{theorem}\label{thm:patching}
Let $\{ \tau_K\}$ be a coherent collection of cohomology classes.
There exists a unique cohomology class $\tau$ in $H^{\chid}(\Mspace{\spaced})$ that pulls back to $\tau_K$ under the map ${\Mover{K}{K} \to \Mspace{\spaced}}$.
\end{theorem}

\begin{proof}

Let $U$ be the interior of a cell $K$ of $\spaced$.
Let $V$ be the union in $\spaced$ of 
\begin{itemize}
\item a subcomplex $V_0$ of $\spaced$ containing the image of the boundary of $K$ under $\characteristic_K$ but not the image of $K$ itself, and 
\item the image under $\characteristic_K$ of the complement of a point in the interior of $K$.
\end{itemize}
Then a homotopy equivalence from $U$ to $K$ lifts to a homotopy equivalence from $\Mover{U}{K}=\Mspace{U}$ to $\Mover{K}{K}$.
Similarly, a deformation retraction from $V$ to $V_0$ lifts to a deformation retraction from $\Mspace{V}$ to $\Mspace{V_0}$.
Note that $U \cap V$ is homotopy equivalent to $\partial K$.
This homotopy equivalence lifts to a homotopy equivalence between $\Mover{U \cap V}{K}=\Mspace{U \cap V} = \Mspace{U} \cap \Mspace{V}$ and $\Mover{\partial K}{K}$.
Also note that $\Mspace{U \cup V} = \Mspace{U} \cup \Mspace{V}$.

Assume that $\tau_{V_0}$ is an element of $H^{\chid}(\Mspace{V_0})$ which pulls back to $\tau_{K'}$
under the map $\Mover{K'}{K'} \to \Mspace{V_0} \subset \Mspace{\spaced}$,
for each cell $K'$ of $V_0$.
Since $\Mspace{V}$ deformation retracts to $\Mspace{V_0}$, the pullback of $\tau_{V_0}$ gives an element $\tau_V$ in $H^{\chid}(\Mspace{V})$ which restricts to $\tau_{V_0}$ on $\Mspace{V_0}$.
The restriction of $\tau_K$ to $\Mspace{U}$ gives a class $\tau_U$ in $H^{\chid}(\Mspace{U})$.
We use the Mayer--Vietoris sequence to show that $\tau_U$ and $\tau_V$ are, respectively, restrictions of a unique element $\tau_{U \cup V}$ in $H^{\chid}(\Mspace{U \cup V})$ to $\Mspace{U}$ and $\Mspace{V}$.

The relevant portion of the Mayer--Vietoris sequence is:
\[
\xymatrixrowsep{20pt}
\xymatrixcolsep{20pt}
\xymatrix{
\save[]+<-50pt,+10pt>*{}="b"\restore& \qquad\qquad\cdots\qquad\qquad\ar[r]^-\psi
& H^{\chid -1}(\Mspace{U\cap V})\ar@{}[d]|(.35){}="a"
\ar `r/5pt[d]`"a" `^d"b"_(.51)\delta`[dll][dll]
\\
H^{\chid}(\Mspace{U\cup V})\ar[r]^-{\varphi}
&{\begin{array}{c}
H^{\chid}(\Mspace{U})\\ \oplus\\ H^{\chid}(\Mspace{V})
\end{array}}\ar[r]^-{\psi}
& H^{\chid}(\Mspace{U\cap V})\ar[r]
&
}\]


First we show the existence of the class $\tau_{U \cup V}$ by showing that $\psi(\tau_U, \tau_V) = 0$, that is, that $\tau_U$ and $\tau_V$ restrict to the same class in $H^{\chid}(\Mspace{U \cap V})$.
By Corollary \ref{cor:partial K} it is enough to show that $\tau_U$ and $\tau_V$ restrict to the same class in $H^{\chid}(\Mover{\widehat{K'}}{K})$ for a face $\widehat{K'}$ in the boundary of $K$.

By Definition~\ref{defi:coequalizer}, the diagram
\[\begin{tikzcd}
\Mover{\widehat{K'}}{K}\ar{dr}[swap]{\bigstarr{\widehat{K'}}{}}\ar{rr}
&&
\Mspace{V} 
\\
&\Mover{K'}{K'}\ar{ur}
\end{tikzcd}\]
commutes, so the pullback of $\tau_V$ to $H^{\chid}(\Mover{\widehat{K'}}{K})$ is the pullback of $\tau_{K'}$ in $H^{\chid}(\Mover{{K'}}{K'})$ under the map induced by $\bigstarr{\widehat{K'}}{}$. On the other hand, the pullback of $\tau_U$ to $H^{\chid}(\Mover{\widehat{K'}}{K})$ is given by the restriction of $\tau_K$ to $\Mover{\widehat{K'}}{K}$, by the definition of $\tau_U$.
But these are equal in $H^{\chid}(\Mover{\widehat{K'}}{K})$ by coherence.

Therefore $\psi(\tau_U, \tau_V) = 0$ and there exists a class $\tau_{U \cup V}$ in $H^{\chid}(\Mspace{U \cup V})$ such that $\varphi(\tau_{U \cup V}) = (\tau_U, \tau_V)$.

Now we show uniqueness of $\tau_{U \cup V}$.
Any two lifts $\tau_{U \cup V}$ in $H^{\chid}(\Mspace{U \cup V})$  of $(\tau_U, \tau_V)$ differ by an element in the image of the connecting homomorphism $\delta$.
The group $H^{\chid - 1}(\Mspace{U \cap V})$ vanishes by Corollary \ref{cor:nolowcohooversphere}.
Therefore, $\delta = 0$ and $\tau_{U \cup V}$ is unique.

Choose a total ordering of all the cells $K_1, K_2, \dots K_N$ of $\spaced$ such that the dimension of $K_i$ is less than or equal to the dimension of $K_{i+1}$. We will refer to the union $\bigcup_{1}^i \characteristic_{K_i}(K_i)$ as $\spaced_i$.

Now assume we have defined a class $\tau_i$ on $\Mspace{\spaced_i}$ which pulls back to $\tau_{K_j}$ on $\Mover{K_j}{K_j}$ for $j\le i$. Let $p$ be the image under the characteristic map $\characteristic_{K_{i+1}}$ of an interior point of $K_{i+1}$. Then the space $\Mspace{\spaced_i}$ is homotopy equivalent to $\Mspace{\spaced_{i+1} - p}$.

By the Mayer--Vietoris argument above, we can then extend $\tau_i$ to $\tau_{i+1}$. Define the class $\tau\in H^{\chid}\Mspace{\spaced}$ as $\tau_N$. 

This shows the existence part of Theorem~\ref{thm:patching}.

Now for uniqueness, let $\tau'$ be any class in $H^{\chid}(\Mspace{\spaced})$ that pulls back  to $\tau_K$ for each cell $K$. 
In particular, both $\tau'$ and $\tau$ restrict to the same class on $\spaced_1$. Now assume that $\tau'\ne \tau$. Then there exists a first $\spaced_i$ in the sequence of subcomplexes defining $\tau$ for which the restriction of $\tau'$ is not equal to the restriction of $\tau$.
However, the class $\tau_i$ above was unique.
Therefore, $\tau$ and $\tau'$ must be equal on $\spaced_i$ and hence on any subcomplex in the sequence.
In particular $\tau$ and $\tau'$ must be equal on the entire space $\Mspace{\spaced}$.

\end{proof}

\section{The diffuse intersection class}\label{section:diffuse intersection class}

In this section, we finally use the orientation of the $d$-dimensional manifold $M$ and orientations of string diagrams.
Our goal is to define a relative cohomology class in $H^{\chid}(S, S-s)$ called the {\em diffuse intersection class}.
In this section, we define a coherent collection of Thom classes for each cell $K$ of $\spaced$, assemble them to build a global Thom class on $\Mspace{\spaced}$ using Theorem \ref{thm:patching}, and build an evaluation map of pairs from $(S, s-s)$ to $\Mspace{\spaced}$.
The diffuse intersection class will be the pullback of the global Thom class under this evaluation map.

\subsection{The global Thom class}\label{subsection: global thom class}

Recall from the proof of Lemma \ref{lemma: disk bundle over diagonal} that the canonical map of leaves into trees $\mathcal{L} \to \mathcal{T}$ induces a diagonal embedding $M^\mathcal{T} \hookrightarrow  M^\mathcal{L}$ whose image lies in the subspace $N_K$, which can be given the structure of a tubular neighborhood of $M^\mathcal{T}$ in $M^\mathcal{L}$.

Fix a string diagram $\Gamma$ such that each $T_j$ is a segment and fix an order $o$ on $\Gamma$.
There is a canonical identification of $\mathcal{T}$ with the ordered set $\{1, \dots, |\chi|\}$.
Precomposing the map induced on homology by the canonical identification of $M^{|\chi|}$  with $M^\mathcal{T}$ with the Eilenberg--Zilber map  gives a map $H_d(M)^{\otimes |\chi|} \to H_{\chid}(M^{|\chi|}) \to H_{\chid}(M^\mathcal{T})$.
We define $[M^{\mathcal{T}}]$ to be the image under this composition of the $|\chi|$-fold tensor power of the fundamental class $[M]$ of $M$ determined by the $R$-orientation of $M$.
We define $[M^{\mathcal{L}}]$ to be the analogous element of $H_{2|\chi|d}(M^{\mathcal{L}})$ and $[N_K]$ to be its image under the composition  
\[H_{2\chid}(M^\mathcal{L}) \to H_{2\chid}(M^\mathcal{L}, M^\mathcal{L}-n_K) \stackrel{\sim}{\to}  H_{2\chid}(N_K, N_K-n_K).\]
The cap product with $[N_K]$ is a Poincar\'e--Lefschetz duality isomorphism between ${H^*(N_K, N_K-n_K)}$ and $H_{2|\chi|d - *}(N_K)$.
The Poincar\'e--Lefschetz dual to the image of $[M^\mathcal{T}]$ under the map induced by the diagonal embedding $M^\mathcal{T} \to N_K$ is an element of $H^{\chid}(N_K, N_K-n_K)$.
In particular, this dual element corresponds to the Thom class of the normal bundle of the diagonal embedding $M^\mathcal{T} \hookrightarrow M^\mathcal{L}$; see the remark after Lemma \ref{lemma: disk bundle over diagonal}.
If $o$ and $o'$ induce the same orientation on $\Gamma$, then the two dual elements in $H^{\chid}(N_K, N_K-n_K)$ are equal. 

\begin{remark}
In the case that $\spaced$ is a trivial double cover of $\spaceduo$, then we could work with $\spaceduo$ instead of $\spaced$.
Additionally, if $M$ is even-dimensional or if $R$ has characteristic $2$,  then we could work with $\spaceduo$ instead of $\spaced$.
\end{remark}

\begin{defi}\label{defi:K-Thom}
For $K$ a cell of $\spaced$, we will define  {\em the $K$-Thom class} $\omega_K$, a cohomology class in $H^{\chid}(\Mover{K}{K})$, as follows.

If $K$ is a $0$-cell, then we denote it by $p$. The $0$-cell $p$ represents an oriented string diagram $\Gamma$ where each $T_j$ is a segment.
Additionally, $\Mover{p}{p}$ is canonically isomorphic to the pair $(N_p, N_p - n_p)$.
In this case, the {\em $p$-Thom class} $\omega_p$ is the Poincar\'e--Lefschetz dual element in $H^{\chid}(N_p, N_p - n_p)$ described above, which is independent of the choice of order $o$ in the orientation class of $\Gamma$.

For $K$ a higher dimensional cell, let $\widehat{p}$ be a $0$-cell of $K$ corresponding to the $0$-cell $p$ of $\spaced$
. There are maps
\[\Mover{K}{K}\to \Mover{\widehat{p}}{K}\to \Mover{p}{p}.\]
The former map is induced by the projection of $K$ onto $\widehat{p}$; 
the latter map is $\bigstarr{\widehat{p}}{}$.
In this case, we define the {\em $K$-Thom class at $\widehat{p}$} as the pullback of the $p$-Thom class $\omega_p$ along this composition. As it will turn out that this is independent of the choice of 0-cell $\widehat{p}$, we will refer to it simply as the {\em $K$-Thom class} $\omega_K$.
\end{defi}
A priori, however, the $K$-Thom class $\omega_K$ depends on the choice of the vertex $\widehat{p}$. 

\begin{lemma}
The $K$-Thom class at $\widehat{p}$ is independent of the choice of the $0$-cell $\widehat{p}$.
\end{lemma}

\begin{proof}
Any two $0$-cells in $K$ can be connected by a chain of one-cells in $K$.
Therefore, it is enough to prove that the $K$-Thom class at $\widehat{p}$ is equal to the $K$-Thom class at $\widehat{q}$, where $\widehat{p}$ and $\widehat{q}$ are vertices of $K$ that are connected by a one-cell $\widehat{\pq}$.

The $1$-cell $\pq$ of $\spaced$ induces an identification of the set of trees of any string diagram in the cell $\pq$ with the set of trees of any string diagram in the cell $p$, and likewise for $q$. These identifications also respect the sets of leaves of all trees. This identification, in turn, induces an identification of $N_{p}$ with $N_{q}$ and of $n_{p}$ with $n_{q}$.
Let 
\[\theta_\pq: \Mover{p}{p} \cong (N_{p}, N_p - n_p) \to \Mover{q}{q} = (N_{q}, N_{q} - n_{q})\]
be induced by these identifications.
Let $\pi_p$ be the composition as defined above $\Mover{K}{K}\to  \Mover{p}{p}$ and let $\pi_q: \Mover{K}{K}\to  \Mover{q}{q}$ be the analogous composition.

To prove that $\pi_p^*(\omega_p)$ and $ \pi_q^*(\omega_q)$ are equal, we first show that the the map $\theta_\pq \circ \pi_p$ is homotopic to the map $\pi_q$ and then that $\theta_{\pq}^*(\omega_q)=\omega_p$.

First, to construct a homotopy, parametrize the $1$-cell $\widehat{\pq}$ of $K$ by a fixed homeomorphism  $[0,1]\to \widehat{\pq}$ and consider the following map:
\[H:[0,1] \times \Mover{K}{K}  \to \Mover{\widehat{pq}}{K}\to \Mover{\pq}{\pq}\to (N_\pq, N_\pq-n_\pq)\to  \Mover{q}{q},\]
where
\begin{itemize}
\item
the first map is induced by the parametrization of $\widehat{\pq}$ and projection from $\Mover{K}{K}$ to $(N_K, N_K - n_K)$,
\item 
the second map is $\bigstarr{\widehat{\pq}}{}$,
\item 
the third map is projection on the second factor in both coordinates (recall that $\Mover{\pq}{\pq}$ is ${(\pq\times N_\pq, \pq\times (N_\pq-n_\pq))}$), and
\item 
the final map uses the canonical bijection between the set of trees of $\pq$ and the set of trees of $q$.
\end{itemize}
It is clear that $H$ is a homotopy between $\theta_{\pq}\circ\pi_p$ and $\pi_q$.

Now, because the map $\theta_\pq \circ \pi_p$ is homotopic to the map $\pi_q$, it suffices to show that $\theta_{\pq}^*(\omega_q)=\omega_p$. Since both $p$ and $q$ are degenerations of the one-cell $\pq$, there is a canonical orientation-preserving bijection between their ordered sets of trees and leaves. Then by the definition of $\omega_p$ and $\omega_q$ for $0$-cells of $\spaced,$ they agree.
\end{proof}

\begin{lemma}\label{lemma:Thom classes pull back}
The collection $\{\omega_K\}$ of $K$-Thom classes is coherent.
\end{lemma}
\begin{proof}
By Definition~\ref{defi:K-Thom}, the result is implied by the commutativity of the following diagram of pairs of spaces.
\[
\begin{tikzcd}
&\Mover{K}{K}\ar{dr}{\pi}\\
\Mover{\widehat{K'}}{K}\ar{ur}{\iota}\ar{rr}{\pi}\ar{d}[swap]{\bigstarr{\widehat{K'}}{}}
&&
\Mover{\widehat{p}}{K}\ar{d}[swap]{\left.\bigstarr{\widehat{K'}}{}\right|_{\widehat{p}}}\ar{dr}{\bigstarr{\widehat{p}}{}}
\\
\Mover{K'}{K'}\ar{rr}[swap]{\pi}&&\Mover{\widehat{p}}{K'}\ar{r}[swap]{\bigstarr{\widehat{p}}{}}&\Mover{p}{p}.
\end{tikzcd}
\]
In the diagram, $\iota$ is induced by the inclusion of $\widehat{K'}$ into $K$ and $\pi$ is induced by projection onto $p$. In particular, both are the identity on the second factor.

The upper triangle commutes because projection to $\widehat{p}$ is insensitive to the domain space. The rectangle commutes because the right vertical map is just the restriction of the left vertical map. The bottom right triangle commutes using Lemma~\ref{lemma:starscommute} by the argument following Definition \ref{def: starr}.

\end{proof}

\begin{defi}
The {\em global Thom class} $\omega$ is the unique class in $H^{\chid}(\Mspace{\spaced})$ pulling back to $\omega_K$ in $H^{\chid}(\Mover{K}{K})$ guaranteed by Theorem \ref{thm:patching}.
\end{defi}

\subsection{The evaluation map}
Recall from Definition \ref{def of S} that
the {\em diffuse intersection locus} is the subspace $S$ of $\spaced\times LM^k$ consisting of pairs $(\Gamma,\gamma)$ where $\gamma$ is $\frac{r}{|\chi|}$-Lipschitz with respect to $\Gamma$ and
the subspace $s$ of $S$ contsists of pairs $(\Gamma,\gamma)$ where $\gamma$ is $\frac{r}{2|\chi|}$-Lipschitz with respect to $\Gamma$.

\begin{defi}
The subspaces $S_K$ and $s_k$ of $K \times LM^k$ are the preimages of the subspaces $S$ and $s$ of  $\spaced \times LM^k$ under $\characteristic_K \times \identity$.
\end{defi}

Let $\Gamma$ be an oriented string diagram in the cell $K$ of $\spaced$.
Recall that the set of leaves $\mathcal{L}$ has a canonical map $\iota$ into $\Gamma$.
Given a map $\Theta: |\Gamma| \to M$, the pullback $\iota^*(\Theta)$ is the restriction of $\Theta$ to the image of $\iota$ in $|\Gamma|$.
Recall also that $\spaced(M)$ consists of pairs $(\Gamma, \Theta: |\Gamma| \to M)$ and the projection $\spaced(M) \to \spaced$ is the forgetful map $(\Gamma, \Theta) \mapsto \Gamma$.
Let $K(M)$ denote the preimage in $\spaced(M)$ of $\characteristic_K(K)$ under this projection map.
Notice that the restriction of the map 
$\heartsuit$
to $S_K$ has image contained in $K(M)$.

\begin{defi}
The  {\em naive $K$-evaluation map} is a map $ev_K$ from $S_K$ to $K \times M^{\mathcal{L}}$.  It is given by:
\[
S_K\xrightarrow{\heartsuit}K(M)
\xrightarrow{(\id, \iota^*)}K\times M^{\mathcal{L}}.
\]
That is, on $(\Gamma, \gamma)$ it fixes $\Gamma$ and then uses $\heartsuit$ to decide where to send $\mathcal{L}$.
\end{defi}

\begin{lemma}
The map $ev_K$ takes $S_K$ into $K \times N_K$.
\end{lemma}
\begin{proof}
Let $(\Gamma, \gamma)$  be in $S_K$. 
The length conditions on trees in $\Gamma$ imply that
 for each component $C$ of $\widehat{\Gamma}$, $\gamma \circ \iota$ sends the leaves of $C$ into an ${r}$-ball in $M$.
The straightening of $C$ takes $C$ to $\Delta_{\leaves{C}}$ and the Riemannian centers of mass map $\Upsilon(\gamma \circ \iota, \quad)$ takes $\Delta_{\leaves{C}}$ into the $r$-ball containing $\gamma(\iota(\leaves{C}))$.
So for $\heartsuit(\Gamma, \gamma) = (\Gamma, \Theta)$, $\Theta: |\Gamma| \to M$ sends $C$ into the $r$-ball containing $\gamma(\iota(\leaves{C}))$.
In particular, for any tree $T$ in $C$, $\Theta$ sends $\leaves{T}$ into this $r$-ball in $M$.
Hence, $ev_K(\Gamma, \gamma)$ lies in $K \times N_K$.
\end{proof}

In order to proceed further, we use the following elementary consequence of convexity, already implicitly used in Lemma \ref{lemma:pre-is-good-enough} and Figure \ref{fig:EpsilonScales}.

\begin{lemma}\label{lemma:convexballs}
Let $U$ be a convex subset of a Riemannian manifold and let $\{P_j\}$ be a finite set of finite sets of points in $U$ such that 
\begin{enumerate}
\item each set of points $P_j$ lies in an $\varepsilon_j$-ball in $U$,
\item for any pair $(i,k)$ there is a sequence $P_i=P_{j_1},P_{j_2},\ldots, P_{j_n}=P_k$ such that 
the finite point sets $P_{j_r}$ and  $P_{j_{r+1}}$ have at least one point in in common.
\end{enumerate}
Then the union of all $P_j$ lies in a ball of radius $\sum \varepsilon_i$.
\end{lemma}
\begin{proof}
It suffices to prove the statement for $|\{P_j\}|=2$. Let $x$ be in $P_1\cap P_2$ and let $c_j$ be the center of an $\varepsilon_j$-ball containing $P_j$. Let $\ell_j$ be the distance from $x$ to $c_j$. By convexity, there is a segment from $c_1$ to $c_2$ with length at most $\ell_1+\ell_2<\varepsilon_1+\varepsilon_2$. Pick a point $y$ on the segment with $d(c_1,y)<\varepsilon_2$ and $d(c_2,y)<\varepsilon_1$. Then $P_1\cup P_2$ is all within $\varepsilon_1+\varepsilon_2$ of $y$.
\end{proof}

\begin{lemma}
The  map $ev_K$ takes $S_{K}-s_{K}$ into $K \times (N_K - n_K)$.
\end{lemma}

\begin{proof}
We show that  the intersection of $K \times n_K$ with the image of $ev_K$  is contained in $ev_K(s_{K})$. 

Let $ev_K(\Gamma,\gamma)$ lie in $K\times n_K$ and write $\heartsuit(\Gamma,\gamma)=(\Gamma,\Theta)$. Fix a component $C$ of the intersection graph $\widehat{\Gamma}$ made up of the trees $\{T_j\}$. 
A vertex of $\widehat{\Gamma}$ is in  $\widetilde{P}_j$ if
\begin{enumerate}
\item it is the image of a vertex of $T_j$, and
\item it is the image of a leaf of $T_{j'}$ for any $j'$ (possibly including $j$).
\end{enumerate}
Let $P_{j}$ be  $\Theta(\widetilde{P}_j) \subset M$.
Since $ev_K(\Gamma,\gamma)$ is in $n_K$ and $\heartsuit$ relies on the straightening map and the Riemannian center of mass map, the collection $P_j$ lies in a $|T_{j}|\frac{r}{4|\chi|^2}$-ball in $M$. 
Furthermore, since the component $C$ is connected, the second condition of Lemma~\ref{lemma:convexballs} holds and we can conclude that the union $\bigcup P_{j}$ lies in a $\sum |T_{j}|\frac{r}{4|\chi|^2}$-ball. This union contains the image of all leaves of $C$ and the radius is at most $\frac{r}{4|\chi|}$. Thus the distance between the images of any two leaves of $C$ is at most $\frac{r}{2|\chi|}$. This shows that $(\Gamma,\gamma)$ is in $s_{K}$.
\end{proof}

\begin{cor}
The naive evaluation map $ev_K$ induces a map of pairs 
\[(S_K,S_K-s_K)
\to
\Mover{K}{K}\]
\end{cor}

\begin{lemma}

Let $K'$ be a codimension one degeneration of a cell $K$ of $\spaced$.
The following diagram commutes:

\[
\begin{tikzcd}
S_K \cap (\widehat{K'} \times LM^k) \ar{rr}{ev_K} \ar[rightarrow]{d}
&&
\widehat{K'} \times N_K \ar{d}{\bigstarr{}{}}
\\
S_{K'}\ar{rr}[swap]{ev_{K'}} && K' \times N_{K'}
\end{tikzcd}
\]
where the existence of the left vertical map is guaranteed by Lemma \ref{lemma:epsilon Lipschitz}.
\end{lemma}
\begin{proof}

If $K'$ is a contraction degeneration of $K$, the statement is trivially true, so we consider only when $K'$ is a pruning degeneration of $K$.

We break the diagram up into the following diagram:

\[
\begin{tikzcd}
S_K \cap (\widehat{K'} \times LM^k) 
	\ar{rr}{ev_K} 
	\ar[rightarrow]{dd}
	\ar{dr}[swap]{\heartsuit}
&&
\widehat{K'} \times N_K 
	\ar{dd}{\bigstarr{}{}}
\\
&
K'(M)
	\ar{dr}[swap]{(\id,\iota^*)}
\\
S_{K'}
	\ar{rr}[swap]{ev_{K'}} 
	\ar{ur}{\heartsuit}
&& 
K' \times N_{K'}
\end{tikzcd}
\]

The bottom triangle commutes by definition of $ev_{K'}$.
By Proposition \ref{prop:heartbreakcommutes}, the left triangle commutes.
The upper right triangle commutes because $ev_K$, $\heartsuit$ and $\bigstarr{}{}$ are all defined by straightening composed with the Riemannian centers of mass map.

\end{proof}

We will use a relative version of this result to define a global evaluation map.
\begin{cor}
Let $K'$ be a codimension one degeneration of a cell $K$ of $\spaced$.
The following diagram commutes:

\[
\begin{tikzcd}
(S_K,S_K-s_K)\cap (\widehat{K'} \times LM^k, \widehat{K'} \times LM^k) \ar{rr}{ev_K} \ar[rightarrow]{d}
&&
\Mover{\widehat{K'}}{K} \ar{d}{\bigstarr{}{}}
\\
{(S_{K'}, S_{K'}-s_{K'})}
\ar{rr}[swap]{ev_{K'}} && \Mover{K'}{K'}
\end{tikzcd}
\]

\end{cor}

This finally allows us to make the following definition

\begin{defi}
The  {\em evaluation map} $ev: (S, S - s) \to \Mspace{\spaced}$ is defined on $(\Gamma, \gamma)$ as $ev_K(\Gamma,\gamma)$, where $\Gamma$ is in the cell $K$.
\end{defi}

\begin{defi}
The {\em diffuse intersection class} $\Omega$
is the pullback to $(S, S-s)$ of the global Thom class $\omega$ on $\Mspace{\spaced}$  under the evaluation map:
\[ \Omega = ev^*(\omega) \in H^{\chid}(S, S-s).\]

\end{defi}

\section{Connections to previous work}\label{section:previouswork}

Recall that the string topology construction  for a cocycle $W$ in $C^{\chid}(S, S-s)$ representing the diffuse intersection class $\Omega$ is a chain map $\mathcal{ST}_{W}$ from the tensor product ${C_*(\spaced) \otimes C_*(LM)^{\otimes k}}$ to $C_{*- \chid}(LM)^{\otimes \ell}$.
By fixing a cycle $\alpha$ in $C_n(\spaced)$, we obtain a chain map
\[
\mu_{(\alpha,W)}:C_*(LM)^{\otimes k} \to C_{*+n- \chid}(LM)^{\otimes \ell}
\]
which induces a map on homology
\[
\bar{\mu}_{(\alpha,W)}:H_*(LM)^{\otimes k} \to H_{*+n- \chid}(LM)^{\otimes \ell}
\]
which depends only on the homology class of $\alpha$ in $\spaced.$

In this section, we review the work of Cohen--Godin~\cite{CG}, who constructed a family of string topology operations
\[\mu_c: H_*(LM)^{\otimes{k}} \to H_{*-\chid}(LM)^{\otimes \ell}\]
parameterized by a space of decorated graphs of a certain type called {\em marked metric chord diagrams} closely related to our string diagrams. These marked metric chord diagrams are a version of what Cohen and Godin call \emph{Sullivan chord diagrams} with some extra data attached to them.

We isolate when marked metric chord diagrams are in fact string diagrams. Then, by choosing an arbitrary orientation, we can treat such a marked metric chord diagram as a $0$-cycle $\Gamma$ in $\spaced$. Then we prove Proposition~\ref{prop:we agree}, which says that our induced map on homology $\bar\mu_{(\Gamma,W)}$ coincides with $\mu_c$.

Tamanoi has shown that many (but not all) of these operations are trivial~\cite{Tamanoi:LCSTTHGTQFTO}. His methods also imply that \emph{some} operations induced by higher degree cycles in $\spaced$ are trivial. One interesting question for further research is to determine precisely which higher homology classes in $\spaced$ induce trivial string topology operations on the homology of the loop space.

At the end of the section we also show that we recover the Batalin--Vilkovisky operator described by Chas--Sullivan in~\cite{CS}.
\subsection{Homology operations induced by \texorpdfstring{$H_0(\spaced)$}{H0(SD)}}
\begin{prop}
\label{prop:we agree}
For a marked metric chord diagram $c$ which is the underlying string diagram of the oriented string diagram $\Gamma$,
the chain map $\mu_{(\Gamma, W)}$ induces the map $\mu_c$ on homology, up to a possible sign $(-1)^d$.
\end{prop}

The proposition will be proved by Lemmas \ref{lemma: first square commutes}, \ref{lemma: second square commutes}, and \ref{lemma: third diagram commutes} below.
The reason for the ambiguity in the sign is as follows.
Recall that our construction uses the orientation of string diagrams to define the $K$-Thom classes $\omega_K$ for each cell $K$.
If the cells $K$ and $K'$ of $\spaced$ project to the same cell in $\spaceduo$, then there is a canonical identification of the spaces $\Mover{K}{K}$ and $\Mover{K'}{K'}$ which identifies the cohomology class $\omega_K$ with $(-1)^d \omega_{K'}$.
Cohen--Godin do not clearly state the ordering or orientation convention they use to define their Thom classes. If $M$ is even dimensional, then the choice of ordering or orientation is irrelevant. Thus in this case, $\mu_{(\Gamma, W)}$ induces a corresponding Cohen--Godin operation on homology with no problem. If $M$ is odd-dimensional, $\mu_{(\Gamma, W)}$ induces the corresponding operation up to the choice of sign. Godin uses local coefficient systems to keep track of signs  \cite{godin}; we expect that we recover her operations without ambiguity.

Note that if the string diagram $\Gamma$ were in a component of $\spaced$ where the two-sheeted covering ${\spaced\to\spaceduo}$ was nontrivial, then $\bar\mu_{(\Gamma, W)}$ and $-\bar\mu_{(\Gamma, W)}$ would agree on all odd dimensional manifolds. In this case, $\Gamma$ would induce a zero operation on the homology of the loop space of an odd dimensional manifold with coefficients in a ring where $2$ was invertible. We know of no such component.

There is another potential ambiguity because it is not clearly stated in~\cite{CG} which convention is used for identifying output boundary cycles with the standard circle. For a given boundary cycle $C$, one might use either the canonical map $\partial_C$ or its reverse $\bar\partial_C$. We choose to use the canonical map for input boundary cycles and its reverse for output boundary cycles as yields operations that agree with~\cite{CS}. In the absence of other evidence, we ascribe some choice along these lines to Cohen and Godin because without it, their gluing theorem (Theorem 6) fails.

Cohen and Godin construct string topology operations 
\[h_*(LM)^{\otimes k}\to h_{* - \chid}(LM)^{\otimes \ell}\] for any homology theory $h_*$ supporting an orientation of the manifold $M$, with coefficients in a field \cite{CG}. 
For our purposes, we consider only singular homology $H_*$ with coefficients in a field.
Their version of the space of string diagrams is the space of marked metric chord diagrams. We adapt their definition to our notation; their original definition is in~\cite{CG} as Definition 1 and the subsequent discussion.

\begin{defi}\label{defi: marked metric chord diagram}
A {\em marked metric chord diagram} is the equivalence class of a combinatorial string diagram equipped with a pseudometric structure such that
\begin{enumerate}
\item an edge has length zero if and only if it is a marking,
\item the union of the $\widetilde{T}_j$ contains no cycle subgraph, and
\item any vertex that belongs to more than one $\widetilde{T}_j$ or $L_i$ also belongs to some $Q_{i'}$.
\end{enumerate}
$\Gamma_1$ and $\Gamma_2$ are equivalent if 
\begin{enumerate}
\item 
by pruning both $\Gamma_1$ and $\Gamma_2$
along every internal half-edge  which is in some $L_i$
one obtains 
isomorphic partially marked pseudometric fatgraphs with specified subfatgraphs and choices of fundamental vertices, and
\item 
under this isomorphism, corresponding output boundary cycles of $\Gamma_1$ and $\Gamma_2$ induce the same cyclic order on the subset of half-edges in that boundary cycle not in any $\widetilde{T}_j$.
\end{enumerate}
\end{defi}
In particular, a combinatorial string diagram satisfying the appropriate properties is alone in its equivalence class unless some $L_i$ and some $\widetilde{T}_j$ intersect, necessarily at a vertex that is also in some $Q_{i'}$. See Figure \ref{fig:CG Equivalence}.

\begin{figure}[ht]
\includegraphics[scale=1]{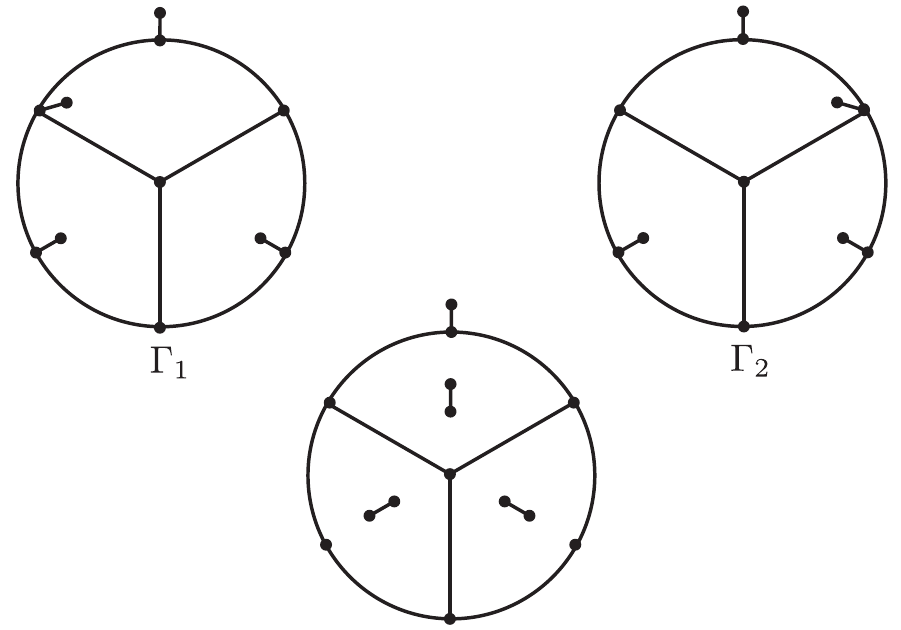}
\caption{$\Gamma_1$, $\Gamma_2$, and their common pruning.}
\label{fig:CG Equivalence}
\end{figure}

Cohen and Godin regard the input circles of a marked metric chord diagram $c$ as copies of the standard circle (of variable length). 
They do not distinguish between the graph and its pseudometric realization explicitly; we will do so. They also introduce a metric fatgraph $S(c)$ where all length-zero edges of $c$ have been contracted.

Cohen and Godin's primary tool for constructing a string topology operation $\mu_c$ for a marked metric chord diagram $c$ is the Pontryagin--Thom construction for the finite codimension subspace $\Maps(|S(c)|,M)$ of $LM^k$.
Evaluation of an element of $LM^k$ at the leaves $\mathcal{L}$ of the trees $\mathcal{T}$ of $c$ gives a map $e_c$ from $LM^k$ to $M^\mathcal{L}$.
The subspace $\Maps(|S(c)|,M)$ fits into the pullback square:
\[
\xymatrix{
\Maps(|S(c)|,M) \ar[r] \ar[d]\save[]+<6pt,-10pt>*{\lrcorner}\restore
&
LM^k \ar[d]^{e_c}
 \\
M^\mathcal{T}  \ar[r]
&
M^\mathcal{L} 
}
\]
where $M^\mathcal{T} \hookrightarrow M^\mathcal{L}$ is the diagonal embedding determined by $c$.
Let $\nu$ be the normal bundle of the diagonal embedding. Then $\nu$ can be given the structure of a tubular neighborhood of the diagonal inside $M^\mathcal{L}$ by choosing an appropriate diffeomorphism. This gives $e_c^*(\nu)$ the structure of a tubular neighborhood of $\Maps(|S(c)|,M)$ in $LM^k$; see~\cite{godin} for details.

For a marked metric chord diagram $c$ with $k$ inputs, $\ell$ outputs, and Euler characteristic $\chi$, Cohen and Godin define an operation 
\[
\mu_c: H_*(LM)^{\otimes k} \to H_{*- \chid}(LM)^{\otimes \ell}.
\]
as the following composition of maps:

\begin{enumerate}[label=(CG\arabic*),ref=(CG\arabic*)]
\item\label{CG: kunneth1}
$H_*(LM)^{\otimes k} \to H_*(LM^k)$, the K\"unneth isomorphism,
\item\label{CG: thom collapse}
next,
$H_*(  LM^k) \to H_*(  \Maps(|S(c)|,M)^\nu)$, the Thom collapse map from $LM^k$ to the Thom space $ \Maps(|S(c)|,M)^\nu$ of the bundle $e_c^*(\nu)$,
\item
\label{CG: thom iso}
next,
$H_*(\Maps(|S(c)|,M)^\nu) \to H_{*-\chid}(\Maps(|S(c)|,M))
$, the Thom isomorphism,
\item
\label{CG: outputs}
next, 
$H_{*-\chid}( \Maps(|S(c)|,M)) \to  H_{*-\chid}(LM^\ell)$, induced by the pullback to outputs, and
\item
\label{CG: kunneth2}
finally, 
$H_{*-\chid}(LM^\ell) \to H_{*-\chid}(LM)^{\otimes \ell}$, the K\"unneth isomorphism.
\end{enumerate}  

Cohen and Godin go on to show that the operation $\mu_c$ depends only on the type $\gkl$ of the marked metric chord diagram $c$.
Therefore, it is enough for us to consider only marked metric chord diagrams of a particularly manageable form. The definitions of marked metric chord diagrams and string diagrams give us the following lemma.

\begin{lemma}\label{lemma: companions}
Let $c$ be a marked metric chord diagram such that:
\begin{enumerate}
\item each $L_i$ is disjoint from all $\widetilde{T}_j$,
\item every $\widetilde{T}_j$ subgraph is a length $1$ segment and the images in $c$ of all of these segments' leaves are distinct, and
\item every $Q_i$ has total length $1$.
\end{enumerate}
Then $c$ satisfies the correct metric properties and is alone in its equivalence class and is thus a string diagram. 
\end{lemma}

For the remainder of this section, assume that $c$ is a marked metric chord diagram satisfying the conditions of Lemma~\ref{lemma: companions} (and thus also a string diagram). Fix an arbitrary oriented string diagram $\Gamma$ with $c$ as its underlying string diagram.

It will be convenient for us to analyze the chain maps defining $\mu_{(\Gamma,W)}$ more closely.
Consider the inclusion of $LM^k$ into $\spaced \times LM^k$ as the fiber over the point $\Gamma$.
Denote the intersection of the fiber with $S$ by $S_\Gamma$ and the intersection of the fiber with $s$ by $s_\Gamma$; the inclusion of the fiber induces inclusions of $S_\Gamma$ into $S$ and of $S_\Gamma - s_\Gamma$ into $S - s$.

There is a unique cell $K$ of $\spaced$ such that $\Gamma$ is in $K$ but not in the boundary of $K$. Then the fiber over the point  $\Gamma$ in the space  $\Mspace{\spaced}$ is the pair $(N_K, N_K - n_K)$.
Because $\Gamma$ is an oriented string diagram where each $T_j$ is a segment, the restriction of the global Thom class $\omega$ on $\Mspace{\spaced}$ to the fiber $(N_K, N_K - n_K)$ is equal to the Thom class of the normal bundle of the embedding $M^\mathcal{T} \hookrightarrow M^\mathcal{L}$.
The evaluation map sends the fiber  $(S_\Gamma, S_\Gamma - s_\Gamma)$  over $\Gamma$ to the fiber $(N_K, N_K - n_K)$ over $\Gamma$ so the restriction of the diffuse intersection class $\Omega$ to  $(S_\Gamma, S_\Gamma - s_\Gamma)$ is the pullback of the Thom class of the normal bundle under the evaluation map.

Additionally, the fiber over the point $\Gamma$ in the space $\spaced(M)$ is the  mapping space $\Maps(|\Gamma|, M)$ and the  map 
$\heartsuit$ sends the fiber over $\Gamma$ to the fiber over $\Gamma$.
Therefore, the map $\mu_{(\Gamma,W)}$ is given by the following composition of maps:

\begin{enumerate}[label=(ST\arabic*),ref=(ST\arabic*)]
\item
\label{us: EZ}
first,
$\chains(LM)^{\otimes k} \to \chains(LM^k)$, the Eilenberg--Zilber shuffle map,
\item
\label{us: quotient}
next,
$\chains(  LM^k) \to \chains(   LM^k,    LM^k - s_\Gamma)$, the usual quotient map, 
\item
\label{us: excision}
next,
$\chains(   LM^k, LM^k - s_\Gamma) \to \chains(S_\Gamma, S_\Gamma- s_\Gamma)
$ a chain homotopy inverse to the map induced by inclusion 
$ (S_\Gamma, S_\Gamma- s_\Gamma)
 \hookrightarrow (LM^k,  LM^k - 
s_\Gamma)
 $
\item
\label{us: cap}
next,
$\chains(S_\Gamma, S_\Gamma- s_\Gamma) \to C_{*-\chid}(S_\Gamma)$,  the cap product with a representative of the pulled-back Thom class, 
\item 
\label{us: heart}
next,
$\heartsuit_*: C_{*-\chid}(S_\Gamma) \to C_{*-\chid}(\Maps(|\Gamma|, M))$, induced by the restriction of the map  $\heartsuit: S \to \mathcal{SD}(M)$ to $S_\Gamma$,
\item
\label{us: outputs}
next, 
$C_{*-\chid}(\Maps(|\Gamma|, M)) \to  C_{*-\chid}(LM^\ell)$, induced by the pullback to outputs, and
\item
\label{us: AW}
finally, 
$C_{*-\chid}(LM^\ell) \to C_{*-\chid}(LM)^{\otimes \ell}$, the Alexander--Whitney map.
\end{enumerate}

Since the shuffle and Alexander--Whitney maps induce the K\"unneth isomorphism over a field, to prove Proposition \ref{prop:we agree}, we need only to check that the map from $H_*(LM^k)$ to $H_{*-\chid}(LM^\ell)$ induced by our composition of maps \ref{us: quotient} to \ref{us: outputs} agrees with Cohen--Godin's composition of maps \ref{CG: thom collapse} to \ref{CG: outputs}, up to a possible sign.
We do this step-by-step in Lemmas  \ref{lemma: first square commutes}, \ref{lemma: second square commutes}, and \ref{lemma: third diagram commutes}, which focus on Cohen--Godin's maps \ref{CG: thom collapse}, \ref{CG: thom iso}, and \ref{CG: outputs}, respectively.

We wish to realize $S_\Gamma$ and $s_\Gamma$ as tubular neighborhoods of $Maps(|S(c)|,M)$ in $LM^k$.
We begin by discussing tubular neighborhoods of $M^\mathcal{T}$ in $M^\mathcal{L}$. In particular, the following definition generalizes the construction of $N_K$ and $n_K$.
\begin{defi}
Let $\epsilon>0$. Then $N_K^\epsilon$ is the subspace of $M^\mathcal{L}$ consisting of maps which, for each tree $T_j$,
 take the leaves of  $T_j$ into an ball of radius $\epsilon$ in $M$.
 \end{defi}
Then $N_K=N_K^r$. Furthermore, since each $T_j$ has length $1$,  $n_K=N_K^{\frac{r}{4|\chi|^2}}$.

\begin{lemma}
For any $\epsilon\le r$, the space $N_K^\epsilon$ has the structure of a tubular neighborhood of the diagonal map $M^\mathcal{T}\to M^{\mathcal{L}}$. Furthermore, if $\epsilon_1<\epsilon_2$ then the closure of $N_K^{\epsilon_1}$ is contained in the interior of $N_K^{\epsilon_2}$.
\end{lemma}
\begin{proof}
The only part that does not follow directly is that $N_K^\epsilon$ is a tubular neighborhood. 

The diagonal map $M^\mathcal{T} \hookrightarrow M^\mathcal{L}$ is the product over $\{T_j\}$ of the diagonal maps ${M \hookrightarrow M^\leaves{T_j}}$.
The normal bundle $\nu$ of the diagonal $M^\mathcal{T} \hookrightarrow M^\mathcal{L}$ is a product over $\{T_j\}$ of normal bundles $\nu_j$ of diagonals $M \to M^\leaves{T_j}$ and we identify each $\nu_j$ with the tangent bundle $TM$.

The space $N_K^\epsilon$ is a product over $\{T_j\}$ of neighborhoods $N_{K,j}^\epsilon$ of the diagonals $M \to  M^\leaves{T_j}$ where $N_{K,j}^\epsilon$ consists of maps $\varphi_j$ from $\leaves{T_j}$ to $M$ with images lying in $\epsilon$ balls in $M$.

In particular, the map $\varphi_j$ sends the two leaves of the segment $T_j$ to points $x_j$ and $y_j$ in the $\epsilon$ ball centered at the midpoint $z_j$ of the unique geodesic segment joining $x_j$ and $y_j$.
(Note that any two of the points $x_j$, $y_j$, and $z_j$ determine the third.)
Let $exp_{z_j}$ be the exponential map for the tangent space to $M$ at the point $z_j$ and define $\phi_j: N_{K,j}^\epsilon \to \nu_j$ as $\phi_j(x_j, y_j) = exp^{-1}_{z_j}(x_j)$, identifying $\nu_j$ with $TM$, the tangent bundle of $M$.
Then $\phi_j$ maps the image of the diagonal to the zero section of $\nu_j$, and is a diffeomorphism onto its image which is the $\epsilon$-disk bundle of $\nu_j$.

Finally, define $\phi: N_K^\epsilon \to \nu$ as the product over $\{T_j\}$ of diffeomorphisms $\phi_j$.
Then $\phi$  maps the image of the diagonal to the zero section of $\nu$ and is a diffeomorphism onto its image which is a convex neighborhood of the zero section of $\nu$.

\end{proof}
\begin{cor}\label{cor: iso pairs downstairs}
Let $0<\epsilon_1<\epsilon_2\le r$; let $0<s_1<s_2$. Let $\nu$ be the normal bundle of the embedding $M^\mathcal{T}\to M^\mathcal{L}$; let $B_s\nu$ be the radius $s$ disk bundle of $\nu$. 

Then there is an isomorphism of fiber bundle pairs between $(B_{s_2}\nu,B_{s_2}\nu-B_{s_1}\nu)$ and $(N_K^{\epsilon_2},N_K^{\epsilon_2}-N_K^{\epsilon_1})$.
\end{cor}
\begin{lemma}
The subspace $S_\Gamma$  (respectively $s_\Gamma$) of $LM^k$ is the pullback of $N_K^{\frac{r}{|\chi|}}$ (respectively $N_K^{\frac{r}{2|\chi|}}$) under the evaluation map $e_c$.
\end{lemma}
\begin{proof}
This is true roughly by definition. The conditions from Lemma~\ref{lemma: companions} imply that each tree $T_j$ is a separate component of the intersection graph of $\Gamma$.
\end{proof}
\begin{cor}\label{cor: S is a tubular neighborhood}
Let $0<s_1<s_2$. The isomorphism of fiber bundle pairs from Corollary~\ref{cor: iso pairs downstairs} pulls back to an isomorphism fiber bundle pairs between 
\[(e_c^*(B_{s_2}\nu),e_c^*(B_{s_2}\nu-B_{s_1}\nu))=
(e_c^*(B_{s_2}\nu),e_c^*(B_{s_2}\nu)-e_c^*(B_{s_1}\nu))
\]
and 
\[(S_\Gamma,S_\Gamma-s_\Gamma).\]
\end{cor}
We will use the notation $f$ for the isomorphism from Corollary~\ref{cor: S is a tubular neighborhood}.

Now to prove that our operations and those of Cohen--Godin agree, we show that each of a sequence of three diagrams commute. 

\begin{lemma}
\label{lemma: first square commutes}
The following diagram commutes:
\[
\begin{tikzcd}
H_*(LM^k) \ar{rr}{\ref{CG: thom collapse}}  \ar{d}[swap]{\ref{us: quotient}}
\ar[dotted]{ddrr}
&&
H_*(\Maps(|S(c)|,M)^\nu)
 \\
H_*(LM^k, LM^k - s_\Gamma) \ar{d}{\sim}[swap]{\ref{us: excision}}
\\
H_*(S_\Gamma, S_\Gamma - s_\Gamma)\ar{rr}[swap]{f_*}
&&
H_*(e_c^*(B_{s_2}\nu),e_c^*(B_{s_2}\nu)-e_c^*(B_{s_1}\nu))\ar{uu}
\end{tikzcd}
\]
where the lower right horizontal map is induced by the diffeomorphism of Corollary \ref{cor: S is a tubular neighborhood} and the right vertical map is the isomorphism on homology that exists because the fiber bundle pair is a cofibration.
\end{lemma}
\begin{proof}
This is true by definition of the Thom collapse map. There is a quotient to relative homology followed by excision from the upper left corner to the bottom right corner which commutes by construction with the lower left maps and by definition with the upper right maps. 
\end{proof}

Now note that since $\Gamma$ is merely an oriented version of $c$, there is a map at the level of pseudometric realizations $|\Gamma|=|c|\to |S(c)|$ given by contracting images of trees.
This map induces an inclusion $i$ of $\Maps(|S(c)|,M)$ into $\Maps(|\Gamma|,M)$ as maps which are constant on images of trees.

\begin{lemma}
\label{lemma: second square commutes}
The following diagram commutes:
\[
\begin{tikzcd}
H_*(\Maps(|S(c)|,M)^\nu) \ar{rr}{\ref{CG: thom iso}}
&&
H_{* - \chid}(\Maps(|S(c)|,M)) \ar{d}{i_*}
 \\
H_*(S_\Gamma, S_\Gamma - s_\Gamma) \ar{u}
\ar{r}[swap]{\ref{us: cap}}
&
H_{*-\chid}(S_\Gamma) \ar{r}[swap]{\ref{us: heart}}
&
H_{*-\chid}(\Maps(|\Gamma|, M)) 
\end{tikzcd}
\]
where the left vertical map is induced by the diffeomorphism of Corollary~\ref{cor: S is a tubular neighborhood} and the right vertical map is induced by the inclusion $i$.
\end{lemma}

\begin{proof}
The Thom isomorphism $H_*(\Maps(|S(c)|,M)^\nu)\xrightarrow{\ref{CG: thom iso}} H_*(\Maps(|S(c)|,M))$ splits as a composition of two isomorphisms:
\begin{enumerate}
\item[\ref{CG: thom iso}$'$]
 first, the isomorphism from $H_*(\Maps(|S(c)|,M)^\nu)$ to $H_{*-\chid}(e_c^*(\nu))$ given by capping with the Thom class of the bundle $e_c^*(\nu)$, and
\item[\ref{CG: thom iso}$''$]
 then, the isomorphism from $H_{*-\chid}(e_c^*(\nu))$ to $H_{*-\chid}(\Maps(|S(c)|,M))$ induced by projection $p$ from the total space to the base space of the pulled back normal bundle.
\end{enumerate}

Therefore, the diagram above may be rewritten as two adjacent squares. The square on the left commutes, possibly up to sign $(-1)^d$, again by Corollary~\ref{cor: S is a tubular neighborhood}. A priori, the Thom class on the bottom and the Thom class on the top are Thom classes of pulled back fiber bundle pairs $(N_K^{\epsilon_2},N_K^{\epsilon_2}-N_K^{\epsilon_1})$ for two different choices of $(\epsilon_1,\epsilon_2)$, but there is a map of pairs between these two choices which makes everything commute.
\[
\begin{tikzcd}
H_*(\Maps(|S(c)|,M)^\nu) \ar{r}{\ref{CG: thom iso}'} 
&
H_{* - \chid}(e_c^*(\nu)) \ar{r}{\ref{CG: thom iso}''} 
&
H_{* - \chid}(\Maps(|S(c)|,M)) \ar{d}{i_*}
 \\
H_*(S_\Gamma, S_\Gamma - s_\Gamma)
\ar{r}[swap]{\ref{us: cap}} \ar{u}{f_*}
&
H_{*-\chid}(S_\Gamma) \ar{r}[swap]{\ref{us: heart}} \ar{u}{f_*}
&
H_{*-\chid}(\Maps(|\Gamma|, M))
\end{tikzcd}
\]

The square on the right is induced by a diagram in spaces:
\[
\begin{tikzcd}
e_c^*(\nu) \ar{r}{p} 
&
\Maps(|S(c)|,M) \ar{d}{i}
 \\
S_\Gamma \ar{r}{\heartsuit} \ar{u}{f}
&
\Maps(|\Gamma|, M)
\end{tikzcd}
\]
This diagram does not commute on the nose, but it does commute up to homotopy.
Roughly, for a fixed element $\gamma$ of $S_\Gamma$, $\heartsuit(\Gamma)$ extends $\gamma$ to a map from $|\Gamma|$ to $M$ by mapping trees to short geodesic segments.
Following the same logic as in the proof of Lemma~\ref{lemma: disk bundle over diagonal}, there is a homotopy from such a map to one which is constant on each tree.
A consistent choice of such a homotopy for all points $\gamma$ in $S_\Gamma$ gives rise to a homotopy from $\heartsuit$ to the composition $i \circ p \circ f$.
Therefore, the induced diagram on the homology of these spaces commutes.
 \end{proof}

\begin{lemma}
\label{lemma: third diagram commutes}
The following diagram commutes:

\[
\begin{tikzcd}
H_{*-\chid}(\Maps(|S(c)|,M)) 
\ar{dd}{i_*}
 \ar{dr}{\ref{CG: outputs} }
&
\\
& H_{*-\chid}(LM^\ell)
\\
H_{*-\chid}(\Maps(|\Gamma|,M)) 
\ar{ur}[swap]{\ref{us: outputs}}&
\end{tikzcd}
\]
\end{lemma}

\begin{proof}
We show that the following diagram in spaces commutes up to homotopy:
\[
\begin{tikzcd}
\Maps(|S(c)|,M) \ar{dd}{i} \ar{dr}{\rho_{out} }&
\\
& LM^\ell
\\
\Maps(|\Gamma|,M) \ar{ur}[swap]{\rho_{out} }&
\end{tikzcd}
\]
where the two maps $\rho_{out}$ to $LM^\ell$ are given by pulling back maps from $|S(c)|$ and $|\Gamma|$ respectively to output circles.

Consider a single output circle of $\Gamma$, with $S^1 \to |\Gamma|$.
The composition of maps ${S^1 \to |\Gamma| \to |S(c)|}$ contracts distinct subintervals of $S^1$ which are preimages of trees of $\Gamma$.
The corresponding output of $S(c)$, $S^1 \to |S(c)|$ does not have such contracted subintervals.
Therefore, the following diagram commutes up to a homotopy parametrizing the contraction of the subintervals and concommitant rescaling:
\[
\begin{tikzcd}
& 
{|S(c)|} 
\\
 S^1 \ar{ur} \ar{dr} &&
 \\
&
{|\Gamma|}.  \ar{uu}
\end{tikzcd}
\]
Postcomposing this homotopy with a map $\theta$ in $\Maps(|S(c)|,M)$ gives a homotopy of the two pullbacks to this output, that is, a homotopy of the two maps from $S^1$ to $M$.
For this fixed $\theta$, the disjoint union over all $\ell$ outputs of these homotopies gives a map from $\sqcup_\ell S^1\times I \to M$ which gives rise to a homotopy of maps in the diagram.
Therefore, the diagram of mapping spaces commutes up to homotopy.

\end{proof}

By placing these three commutative diagrams side by side, we see that both constructions induce the same operations from $H_*(LM^k)$ to $H_{*-\chid}(LM^\ell)$ and the proposition follows.

\subsection{The Batalin--Vilkovisky operator}
In this final section we shall show that we recover the Batalin--Vilkovisky operator defined by Chas--Sullivan~\cite{CS}. 

Consider the map $\reparam:S^1\times LM\to LM$ given by ${\reparam(s,\gamma)(t)=\gamma(s+t\pmod 1)}$. This map induces a map on homology $\Delta:H_*(LM)\to H_{*+1}(LM)$ given by
\[\Delta(\alpha)=\reparam_*(\eta
\times
 \alpha)\]
where $\eta$ is the class of the definining map $[0,1]\to S^1$, viewed as a $1$-chain.

At the chain level, we may make the following definition:
\begin{defi}
The \emph{chain-level Batalin--Vilkovisky operator} $\Delta$ is given by the composition
\[
C_i(LM)\xrightarrow{\eta\otimes} C_1(S^1)\otimes C_i(LM)\xrightarrow{\text{EZ}}C_{i+1}(S^1\times LM)\xrightarrow{\reparam_*}C_{i+1}(LM).\]
\end{defi}

Now consider our space $\spaced(0,1,1)$ of string diagrams of Euler characteristic $0$ with one input and one output. A string diagram $\Gamma$ in that space necessarily has a single $Q_1$ and a single $L_1$ with empty $\{\widetilde{T}_j\}$ set. 
For the input $in$ and the output $out$ of such a diagram $\Gamma$, the maps $\partial_{in}$ and $\bar\partial_{out}$ provide explicit identifications of the pseudometric realization $|\Gamma|$ of $\Gamma$ with the standard circle.
On such a string diagram, an ordering or orientation is no data at all. 
There is an explicit cellular identification $p$ of the space $\spaced(0,1,1)$ with the standard circle unit circle $S^1$ given by the position on $Q_1$ of the vertex of $L_1$, relative to the vertex of $Q_1$.
See Figure \ref{fig:SD011}. 
The string diagram for which these vertices coincide corresponds to the $0$-cell; string diagrams for which these vertices are distinct correspond to points in the interior of the $1$-cell. 
An explicit formula for the identification is given by $p(\Gamma) = \partial_{in}^{-1}\bar\partial_{out}(0)$.

\begin{figure}[ht]
\includegraphics[scale=1]{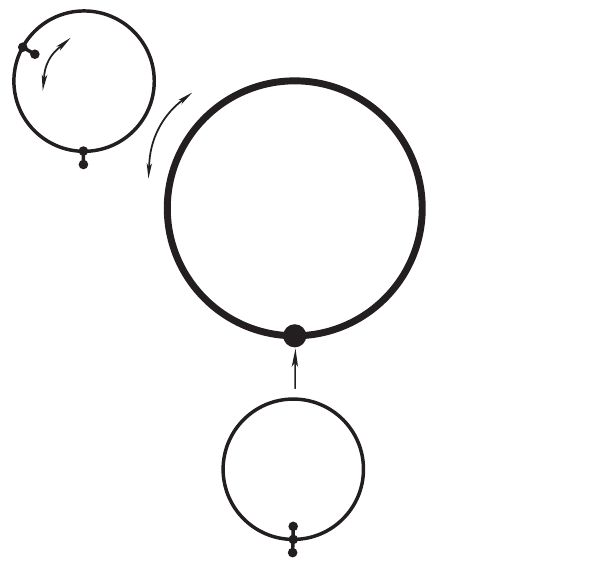}
\caption{The space $\spaced(0,1,1)$ and two string diagrams; one corresponding to  the $0$-cell and one corresponding to a point in the interior of the $1$-cell.}
\label{fig:SD011}
\end{figure}

\begin{prop}\label{prop:we agree BV}
Our string topology operation $\mu_{p^{-1}\eta,W}:C_*(LM)\to C_{*+1}(LM)$ coincides with the chain-level Batalin--Vilkovisky operator $\Delta$.
\end{prop}
\begin{proof}
Because there are no $T_j$ in any string diagram in $\spaced(0,1,1)$, the string topology construction is significantly simpler. The spaces $S$ and $s$ are both equal to $\spaced(0,1,1)\times LM$, so our relative chains are absolute chains and our excision map is the identity. 

The embedding of $M^\mathcal{T}$ in $M^\mathcal{L}$ is the constant map on the point and its Poincar\'e--Lefschetz dual is the degree $0$ cohomology class represented by the constant function $1$. Chasing the definitions, we see that the diffuse intersection class $\Omega$ is also the class in degree $0$ cohomology of the constant function $1$. In degree $0$, cocycle representatives are unique, so $W$ must be the constant function $1$ and capping with it is the identity.

The map $\heartsuit$ uses the canonical map $\partial_{in}$ for the input boundary cycle and takes the pair $(\Gamma,\gamma)$ to the pair $(\Gamma,\gamma\circ \partial_{in}^{-1})$ in $\spaced(0,1,1)(M)$.

Then the entire composition $(\rho_{in})_!$ described in Definition~\ref{def:in} boils down to the Eilenberg--Zilber shuffle map $C_1(\spaced(0,1,1))\otimes C_i(LM)\to C_{i+1}(\spaced(0,1,1)\times LM)$ followed by the map induced by $\heartsuit$.

Next, the map $(\rho_{out})_*$ described in Definition~\ref{def:out} is induced by restriction to 
the single output using $\bar\partial_{out}$
(the Alexander--Whitney map is the identity). 
Then our map is the composition that runs along the left side of the following commutative diagram, whereas the chain level Batalin--Vilkovisky operator is the composition that runs along the right:
\[
\begin{tikzcd}
&C_i(LM)\ar{dl}[swap]{p^{-1}\eta \otimes}
\ar{dr}{\eta\otimes}
\\
C_1(\spaced(0,1,1))\otimes C_i(LM)\ar{rr}{p_*\otimes\id}
\ar{d}[swap]{\text{EZ}} && C_1(S^1)\otimes C_i(LM)\ar{d}{\text{EZ}}
\\
C_{i+1}(\spaced(0,1,1)\times LM)\ar{rr}{(p\times \id)_*}
\ar{dr}[swap]{(\rho_{out})_*\circ\heartsuit_*}&&C_{i+1}(S^1\times LM)\ar{dl}{{\reparam_*}}
\\
&C_{i+1}(LM).
\end{tikzcd}\]
The upper triangle and square clearly commute. To see commutativity of the bottom triangle, observe that it is induced by a diagram of spaces
\[
\begin{tikzcd}
\spaced(0,1,1)\times LM\ar{rr}{p\times \id}
\ar{dr}[swap]{\rho_{out}\circ\heartsuit}&&S^1\times LM\ar{dl}{{\reparam}}
\\
&LM.
\end{tikzcd}\]
The two compositions are as follows:
\begin{align*}
(\rho_{out}\circ\heartsuit)(\Gamma,\gamma)(t)
&=\gamma(\partial_{in}^{-1}\circ\bar\partial_{out}(t))\\
\reparam(p(\Gamma),\gamma)(t)
&=\gamma(t+\partial_{in}^{-1}\bar\partial_{out}(0)).
\end{align*}
But $\partial_{in}^{-1}\bar\partial_{out}:S^1\to S^1$ is always a rotation so these coincide and the diagram commutes.
\end{proof}

\appendix

\section{Straightening of short-branched trees}\label{appendix:straightening}

The purpose of this appendix is to prove Proposition \ref{prop:degenerate straighten}.
We will do this constructively by defining a straightening map $\straighten$ which satisfies the conditions of the proposition.
Namely, given a short-branched tree $T$, the map $\straighten$ from the pseudometric realization $|T|$ of $T$ to the simplex $\Delta_{\leaves{T}}$ spanned by the leaves of $T$ takes each leaf to itself.
Further, if $T$ has an internal edge of length zero or a prunable branch, then the straighten map behaves well with respect to contracting the edge or pruning the branch.

To assign a point in $\Delta_{\leaves{T}}$ to a point $v$ of $|T|$ it suffices to give barycentric coordinates.
This means that for every leaf $w \in \leaves{T}$ we give $v$ a nonnegative coordinate $a(v,w)$ corresponding to $w$  such that the sum over all leaves of the coordinates $a(v,w)$ is $1$.

Our formula for $a(v,w)$ includes expressions involving lengths of branches of trees.
Because we are working in $|T|$ rather than in $T$ itself, we have definitions analogous to those given previously.
In particular, we define versions of {\em vertices}, {\em edges}, {\em length,} and {\em leaf length} in the metric space $|T|$ rather than in $T$ itself.

Fix a point $v$ of $|T|$.
\begin{enumerate}
\item
A {\em $v$-zero cell of $|T|$} is a point in the image of a vertex of $T$ or $v$ itself, if $v$ is not in the image of any vertex of $T$.
\item
A {\em $v$-one cell of $|T|$} is a subspace of $|T|$ homeomorphic to a closed interval whose boundary consists of two $v$-zero cells and which contains no other $v$-zero cells.
\item
The {\em length} of a $v$-one cell is the distance in $|T|$ between its two boundary $v$-zero cells.
\item
Given a subspace of $|T|$ consisting of $v$-one cells, its {\em length} is the sum of the lengths of the constituent $v$-one cells.
\item
Given a pair of distinct $v$-zero cells $v_i$ and $v_j$ of $|T|$ the subspace $T(v_i, v_j)$ is the closure in $|T|$ of the connected component of $|T|- v_i$ containing $v_j$.
\item 
Given such a subspace $T(v_i, v_j)$ of $|T|$, its {\em leaf length} is the number of leaves of $T$ whose image in $|T|$ lies in $T(v_i, v_j) - v_i$.
\item
Given such a subspace $T(v_i, v_j)$ of $|T|$, its {\em deviation} $D(v_i, v_j)$ is the leaf length of $T(v_i, v_j)$ minus the length of $T(v_i, v_j)$.
\end{enumerate}\

Now we prove a sequence of lemmas giving bounds on the values of the deviations.
We will use these lemmas to define barycentric coordinate $a(v,w)$.

\begin{lemma}\label{lemma:positive deviation}
Given a pair distinct $v$-zero cells $v_i$ and $v_j$ of $|T|$, the deviation $D(v_i, v_j)$ of  $ T(v_i,v_j)$ is greater than or equal to zero.
\end{lemma}

\begin{proof}
Consider the subtree of $T$ consisting of edges of $T$ whose images in $|T|$ intersect $T(v_i,v_j) - v_i$.
At most one leaf of this subtree is not a leaf of $T$.
In the case that this subtree is a branch of $T$ or all of $T$ itself, we name this subtree $\widehat{T}(v_i,v_j)$.
Otherwise, the subtree has one leaf which is a bivalent vertex of $T$.
In this case, we add to the subtree the other edge of $T$ adjacent to this bivalent vertex.
We continue adding edges in this way until we reach a non-bivalent vertex of $T$ and name the resulting subtree $\widehat{T}(v_i, v_j)$.
In the case that this last vertex is at least trivalent in $T$, the subtree $\widehat{T}(v_i, v_j)$ is a branch of $T$.
In the case that this last vertex is a leaf of $T$, the subtree $\widehat{T}(v_i, v_j)$ is $T$ itself.

The leaf length of $\widehat{T}(v_i, v_j)$ is equal to the leaf length of ${T}(v_i, v_j)$.
The length of $\widehat{T}(v_i, v_j)$ is greater than or equal to the length of ${T}(v_i, v_j)$.
Because $T$ is a short-branched tree, the length of $\widehat{T}(v_i, v_j)$ is less than or equal to its leaf length.
Thus the  length of ${T}(v_i, v_j)$ is less than or equal to its leaf length. See Figure \ref{fig:branch cover}.

\begin{figure}[ht]
\includegraphics[scale=0.8]{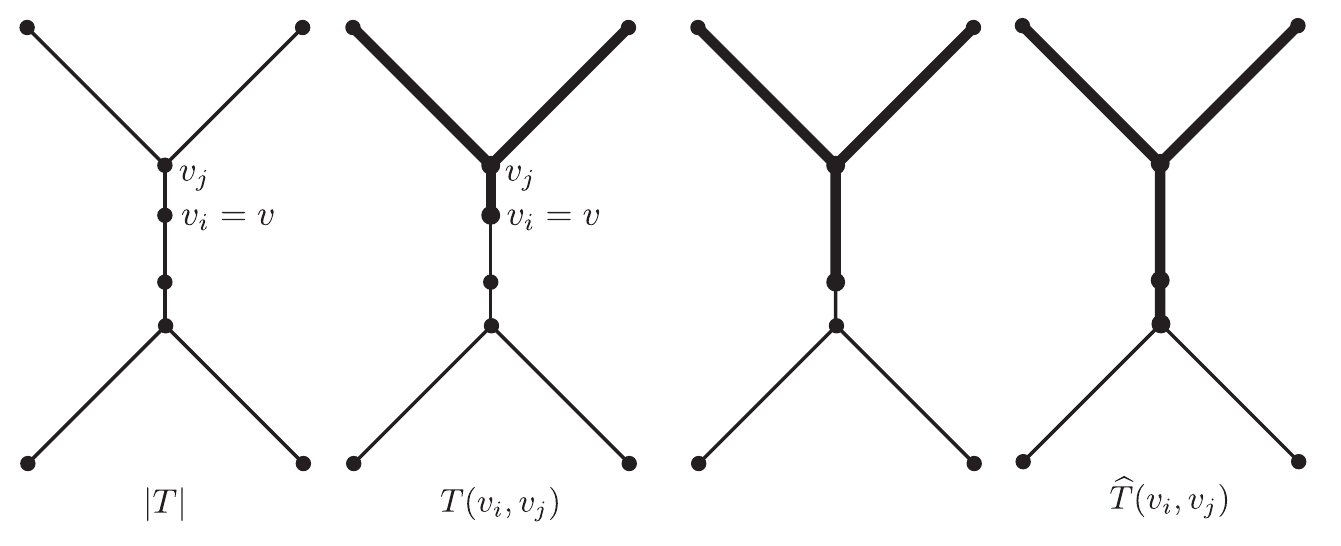}
\caption{Construction of the branch $\widehat{T}(v_i, v_j)$ where $v_i = v$ is not the image of a vertex of $T$. 
The first picture shows the $v$-zero and $v$-one cells of $|T|$ and the second shows the subspace $T(v_i,v_j)$ of $|T$.
The third picture shows the initial subtree of $T$ and the fourth shows its completion to the branch $\widehat{T}(v_i, v_j)$ of $T$.
}
\label{fig:branch cover}
\end{figure}
\end{proof}
\begin{lemma}\label{lemma: fundamental lemma}
Let $T$ be a short-branched tree.
Let $v_i$ be a $v$-zero cell of $|T|$ and let $\{v_j\}$ be a set of points, one from each component of $|T|- v_i$.
Then
\[
\sum_j D(v_i, v_j) =
\left \{
\begin{array}{ll}
0 & \text{ if } v_i \text{  is  the  image  of  a leaf of } T \\
1 & \text{ otherwise.}
\end{array}
\right.
\]
\end{lemma}
Figure \ref{fig:fundamental lemma} shows the space $|T|$, together with the $v$-zero cell $v_i$ and the set $\{v_j, v_j', v_j''\}$, each in a different component of $|T| - v_i$.
The space $|T|$ is the union of subspaces $T(v_i, v_j)$, $T(v_i, v_j')$, and $T(v_i, v_j'')$

\begin{figure}[ht]
\includegraphics[scale=0.9]{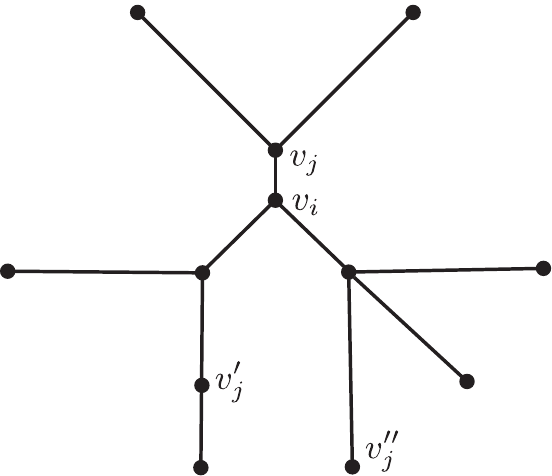}
\caption{The space $|T|$ as the union of subspaces $T(v_i, v_j)$, $T(v_i, v_j')$, and $T(v_i, v_j'')$.
}
\label{fig:fundamental lemma}
\end{figure}

\begin{proof}
By definition, 
\[\sum_j D(v_i, v_j) = \sum_j \text{leaf length of } T(v_i, v_j) - \sum_j \text{length of } T(v_i, v_j).\]
The first sum  on the right is equal to the number of leaves of $T$ whose images lie in $|T| - v_i$.
The second sum on the right is the length of $|T|$, which is the length of $T$.
Because $T$ is a short-branched tree, the length of $T$ is equal to one less than the number of leaves of $T$.
Thus, if $v_i$ is itself the image of a leaf, the two sums on the right are equal; if $v_i$ is not the image of a leaf, the first sum is one greater than the second.
\end{proof}

\begin{lemma}
Given a pair of distinct $v$-zero cells $v_i$ and $v_j$ of $|T|$, the deviation $D(v_i, v_j)$ of  $ T(v_i,v_j)$ is strictly less than one.
\end{lemma}

\begin{proof}
If $v_i$ is the image of a leaf of $T$, then the previous two lemmas show that $D(v_i, v_j)=0$.

If $v_i$ is not the image of a leaf of $T$, then assume $D(v_i, v_j)=1$.
Then for any $v_j'$ in a different component of $|T| - v_i$ than the one containing $v_j$, $D(v_i, v_j')$ must be $0$ by the previous two lemmas.
Let $e$ be the $v$-one cell of $|T|$ adjacent to $v_i$ in $T(v_i, v_j)$ and consider the subspace of $|T|$ given by the union of the $T(v_i, v_j')$, for all such $v_j'$ and $e$.
If  $v_i'$ is the other $v$-zero cell of $e$ then this subspace is equal to $T(v_i', v_i)$.
The leaf length of $T(v_i', v_i)$ is equal to the sums of the leaf lengths of the $T(v_i, v_j')$ and
the length of $T(v_i', v_i)$ is equal to the sums of the lengths of the $T(v_i, v_j')$, plus the length of $e$.
Thus, $D(v_i', v_i)$ is equal to the sum of the  $D(v_i, v_j')$, minus the length of $e$ and $D(v_i', v_i)$ is equal to negative the length of $e$.
But $e$ has positive length and $D(v_i', v_i)$ must be nonnegative by Lemma \ref{lemma:positive deviation}, so we arrive at a contradiction.

\end{proof}

\begin{lemma}\label{lemma: sum of deviations}
Given  a $v$-one cell $e$ of $|T|$ containing the $v$-zero cells $v_i$ and $v_j$, the sum of deviations $D(v_i, v_j) + D(v_j, v_i)$ is equal to one minus the length of $e$.
\end{lemma}
\begin{proof}
The sum of the leaf lengths of $T(v_i, v_j)$ and $T(v_j, v_i)$ is equal to the total number of leaves of $T$.
The sum of the lengths of $T(v_i, v_j)$ and $T(v_j, v_i)$ is equal to the length of $|T|$ plus the length of $e$.
But the length of $|T|$ is equal to the length of $T$, which is equal to one less than the number of leaves of $T$.
Therefore the sum of deviations $D(v_i, v_j) + D(v_j, v_i)$ is equal to one minus the length of $e$.
\end{proof}

\begin{cor}
Given  a $v$-one cell $e$ of $|T|$ containing the $v$-zero cells $v_i$ and $v_j$,
$\frac{D(v_i, v_j)}{1-D(v_j, v_i)}$ is well-defined and is between $0$ and $1$.
\end{cor}

We are now ready to define the barycentric coordinate $a(v,w)$, which uses a sequence of $v$-zero cells in $|T|$.
Given a point $v$ of $|T|$ and a leaf $w$ of $T$, any path in $|T|$ from $v$ to the image $|w|$ of $w$ passes through a finite sequence of $v$-zero cells.
In particular, there is a unique such sequence $v=v_1, v_2, \dots, v_n=|w|$ of {\em distinct} $v$-zero cells.

\begin{figure}[ht]
\includegraphics[scale=0.8]{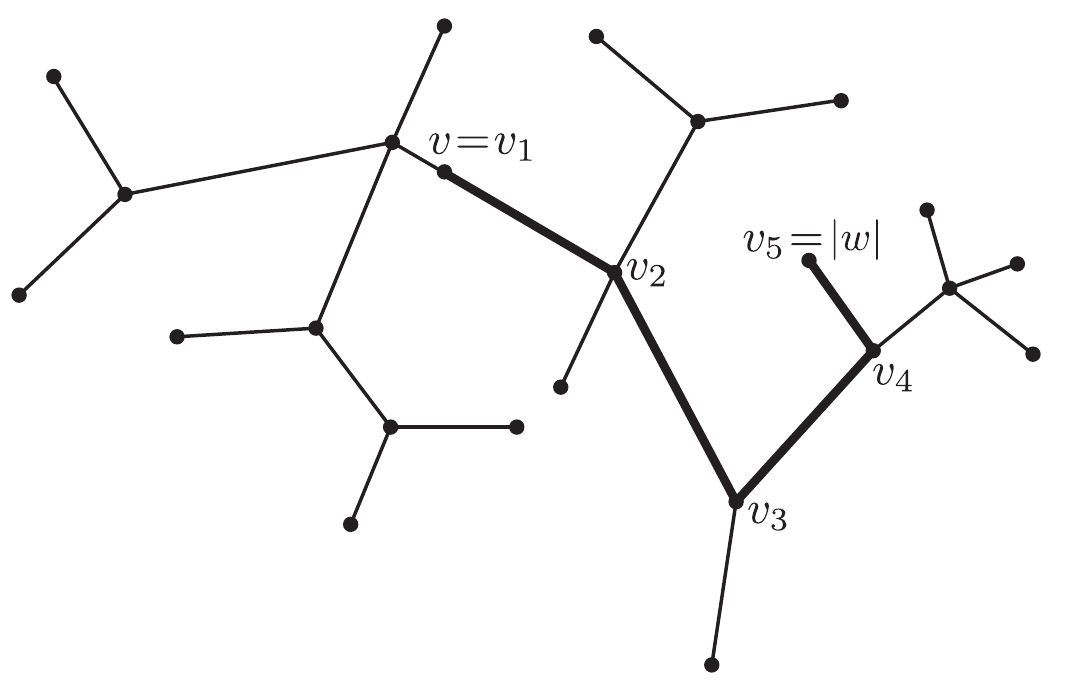}
\caption{A path from a point $v$ to a leaf $w$ in the  short-branched tree $T$}
\label{fig:path in tree}
\end{figure}

\begin{defi}\label{def:barycentriccoordinate}
The {\em barycentric coordinate} $a(v,w)$ is given by the formula
\[
a(v,w) = \prod_{i=1}^{n-1}\frac{D(v_i, v_{i+1})}{1-D(v_{i+1}, v_i).}
\]
\end{defi}
\begin{figure}[ht]
\includegraphics[scale=0.8]{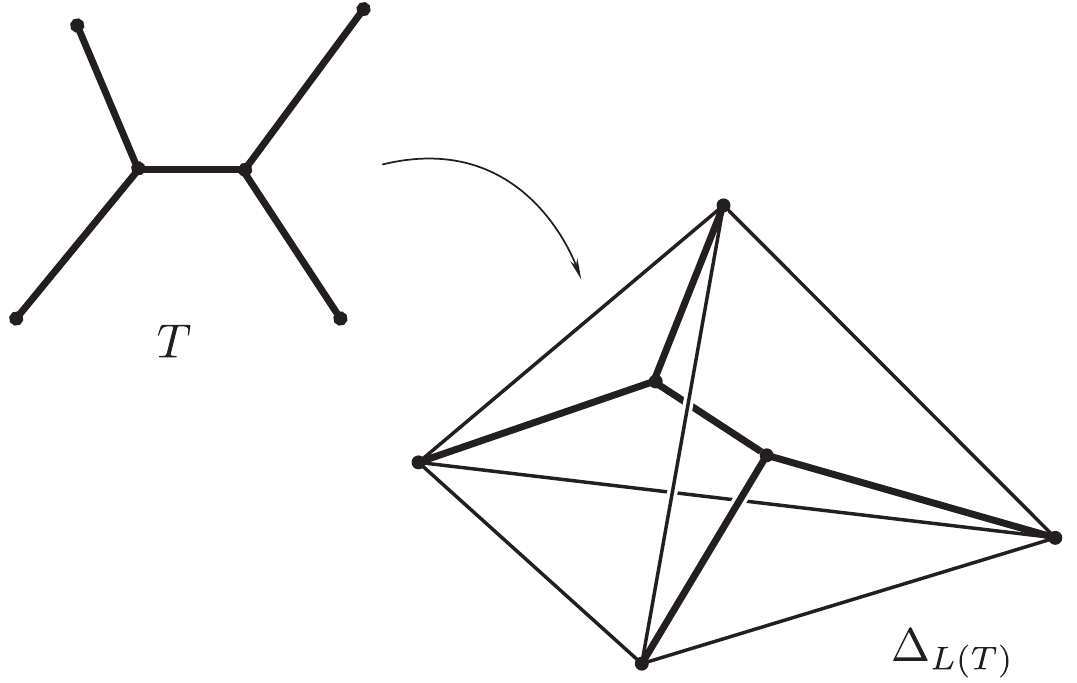}
\caption{A short-branched tree $T$ and its straightening in $\Delta_{\leaves{T}}$}
\label{fig:straightened tree2}
\end{figure}

\begin{lemma}
The barycentric coordinate is continuous with respect to the first coordinate.
\end{lemma}
\begin{proof}

Fix an oriented edge $\vec{e}$ of $T$ with length $\ell(\vec{e})>0$.
In general, the factors $\frac{D(v_i, v_{i+1})}{1-D(v_{i+1}, v_i)}$ defining $a(v,w)$ vary continuously---and in fact, all but the first one of them are constant---as $v$ moves around in the image of $(0, \ell(\vec{e})) \times \{\vec{e}\}$ in $|T|$.
Therefore it suffices to consider a sequence of points $\{v^k\}$ in the image of ${(0, \ell(\vec{e})) \times \{\vec{e}\}}$ in $|T|$ converging to the point $v^\infty$, the image of $\{0\} \times \{\vec{e}\}$ in $|T|$.

Unless $w$ is in $T(v^k, v^\infty)$, the factors vary continuously and therefore $a(v^k, w)$ converges to $a(v^\infty, w)$.

If $w$ is in $T(v^k, v^\infty)$, then in the limit the first factor $\frac{D(v^k, v^\infty)}{1-D(v^\infty, v^k)}$ defining $a(v^k, w)$ disappears in the product defining $a(v^\infty,w)$.
By Lemma \ref{lemma: sum of deviations}, the sum of deviations $D(v^k, v^\infty) + D(v^\infty, v^k)$ is equal to one minus the length of the $v^k$-one cell $e$ joining $v^k$ and $v^\infty$.
When $v^k$ approaches $v^\infty$, the length of $e$ approaches zero.
Therefore the first factor $\frac{D(v^k, v^\infty)}{1-D(v^\infty, v^k)}$ approaches one
and $a(v^k, w)$ converges to $a(v^\infty, w)$.
\end{proof}

\begin{lemma}\label{lemma: sum is 1}
For any point $v$ in $|T|$, the sum over all leaves $w$ in $\leaves{T}$ of the barycentric coordinates $a(v,w)$ is $1$.
\end{lemma}
\begin{proof}

Fix a $v$-zero cell $\hat{v}$ of $|T|$.

If $v$ is the image of a leaf of $T$, then the first factor defining $a(v,w)$ is zero unless $v = |w|$.
If $v = |w|$, the product is empty, so $a(v,w) = 1$.

Therefore, assume that $v$ is not the image of a leaf in $T$.
Consider the set of leaves $\{w_i\}$ such that  every path from $v$ to $|w_i|$ in $|T|$ passes through $\hat{v}$.
Define $a(v,\hat{v})$ to be the sum $\sum_{w\in \{w_i\}}a(v,w)$.
Notice that 
\[\sum_w a(v,w) = \sum_{\hat{v}} a(v,\hat{v}),\]
where the sum on the left is over all leaves $w$ of $T$ and the sum on the right is over all vertices $\hat{v}$ which are adjacent to $v$ in $T$,
so it suffices to prove that $\sum_{\hat{v}} a(v,\hat{v}) = 1$.

In fact, we will prove a more general statement than this.
We will show that for a general $v$-zero cell $\hat{v}$ of $|T|$, 
\[a(v,\hat{v}) = D(v_{n-1},\hat{v}) \prod_{i=1}^{n-2}{\frac{D(v_i, v_{i+1})}{1-D(v_{i+1}, v_i)}}\]
where $v=v_1, v_2, \dots, v_n=\hat{v}$ is the sequence of $v$-zero cells in $|T|$ between $v$ to $\hat{v}$.
This will prove what we want, that $\sum_{\hat{v}} a(v,\hat{v}) = 1$ for $\hat{v}$ adjacent to $v$, since in this case
$a(v,\hat{v}) = D(v, \hat{v})$.
The result then follows by Lemma \ref{lemma: fundamental lemma}.

We now prove the general statement, that
by induction over the vertices $\hat{v}$.

In the base case, $\hat{v}$ is the image of a leaf  of $T$, so the formula holds by our original definition of $a(v,w)$ since all deviations
$D(\hat{v},v')$ are equal to zero.

Now assume that $\hat{v}$ is not the image of a leaf of $T$ and that this formula holds for $a(v,\tilde{v})$, where $\tilde{v}$ is a $v$-zero cell sharing a $v$-one cell with $\hat{v}$ and every path from $v$ to $\tilde{v}$ passes through $\hat{v}$.
For such $\tilde{v}$, let $v=v_1, \dots, \hat{v} = v_n, \tilde{v}=v_{n+1}$ be the sequence of $v$-zero cells between $v$ and $\tilde{v}$.
So by assumption
\[a(v,\tilde{v}) = 
D(\hat{v}, \tilde{v}) \prod_{i=1}^{n-1}\frac{D(v_i, v_{i+1})}{1 - D(v_{i+1}, v_i)}.
\]

We need to sum over all choices of $\tilde{v}$ as above.
Doing, this we get

\[ a(v,\hat{v}) = \sum_{\tilde{v}} a(v,\tilde{v})\\
= 
\left(
\prod_{i=1}^{n-1}\frac{D(v_i, v_{i+1})}{1 - D(v_{i+1}, v_i)}\right)
 \sum_{\tilde{v}} 
D(\hat{v}, \tilde{v}).
\]

By Lemma \ref{lemma: fundamental lemma}, the sum $ \sum_{\tilde{v}} 
D(\hat{v}, \tilde{v}) = 1 - D(\hat{v}, v_{n-1})$.

\end{proof}

\begin{cor}\label{cor:straighten to simplex}
The barycentric coordinate is a map from $|T|$ to $\Delta_{\leaves{T}}$ that sends each leaf to itself.
\end{cor}

\begin{defi}
The {\em straightening} of the  short-branched tree $T$ with leaf set $\leaves{T}$  is the map from $|T|$ to $\Delta_{\leaves{T}}$ 
defined by barycentric coordinates.
\end{defi}

By Corollary \ref{cor:straighten to simplex}, the straightening map $\straighten$ satisfies the first condition of Proposition \ref{prop:degenerate straighten}.
It remains to show that it satisfies the second and third conditions.
Namely,  the straightening map degenerates nicely with respect to contraction degenerations (Lemma \ref{lemma: contract straighten}) and with respect to pruning degenerations (Lemma \ref{lemma: graft straighten}).

\begin{lemma}\label{lemma: contract straighten}
Let $e$ be an internal edge of $T$ or an external edge of $T$ whose internal vertex is bivalent,
Assume $e$ has length zero.
Let $|T| \to |T/e|$ be the isomorphism induced the contraction of the edge $e$.
Then the following diagram commutes.
\[
\begin{tikzcd}
{|T|} \ar{rrrr}{\straighten} \ar{dd}[swap]{\cong}
&&&&
\Delta_{\leaves{T}}\ar{dd}{\cong}
\\
\\
{|T/e|} \ar{rrrr}{\straighten}
&&&&
\Delta_{\leaves{T/e}}
\end{tikzcd}
\]

\end{lemma}
\begin{proof}
For any point $v$ of $|T|$, the isomorphism $|T| \to |T/e|$ preserves $v$-zero cells and lengths of $v$-one cells.
Therefore the barycentric coordinates of Definition \ref{def:barycentriccoordinate} coincide in the two cases.
\end{proof}

\begin{lemma}\label{lemma: graft straighten}
Let $T_h$ be a prunable branch of $T$ with pollard $T^h$.
Since $\leaves{T^h}$ is a subset of $\leaves{T}$, there is a natural inclusion of $\Delta_\leaves{T^h}$ in $\Delta_{\leaves{T}}$.
Since every leaf of $T_h$ except $s(h)$ is also a leaf of $T$, assigning a point in $\Delta_{\leaves{T}}$ to $s(h)$ yields a  linear inclusion of $\Delta_{\leaves{T_h}}$ in $\Delta_{\leaves{T}}$.
Here, since $s(h)$ is also a vertex of $T^h$, it has an image point $|s(h)|$ in the pseudometric realization $|T^h|$. Thus we assign to $s(h)$ the image of $|s(h)|$ under the straightening map of $T^h$.
Then the following diagram commutes.
\[
\begin{tikzcd}
{|T^h|} \sqcup |T_h| \ar{rrrr}{\straighten} \ar{dd}[swap]{\rho}
&&&&
\Delta_{\leaves{T^h}} \sqcup \Delta_{\leaves{T_h}} \ar{dd}
\\\\
{|T|} \ar{rrrr}{\straighten}
&&&&
\Delta_{\leaves{T}}
\end{tikzcd}
\]

\end{lemma}

See Figure \ref{fig:grafted straightening} for a picture of the inclusion of $\Delta_{\leaves{T_h}}$ in $\Delta_{\leaves{T}}$.

\begin{proof}
Let $v$ be a point in ${|T^h|} \sqcup |T_h|$ and let $w$ a be leaf of $T$.
We prove  different cases, which depend on whether the preimage of $|w|$ in ${|T^h|} \sqcup |T_h|$ under $\rho$ lies in the same component as $v$ or not.

\subsubsection*{Case 1}If the point $v$ is in $|T^h|$ and there is a leaf $w^h$ of $T^h$ with $\rho(|w^h|) = |w|$, 
then the choice of leaf $w^h$ is unique. Let $T^h(v_i,v_j)$ be a subspace of $|T^h|$ whose deviation appears in the expression for $a(v,w^h)$ in $\Delta_{\leaves{T^h}}$.
If $|s(h)|$ is not a point in $T^h(v_i,v_j)-v_i$, then the deviation of the isomorphic image $T(v_i,v_j) = \rho(T^h(v_i,v_j))$ appears in the expression for $a(v,w)$ in $\Delta_{\leaves{T}}$.
On the other hand,
if $|s(h)|$ is a point in $T^h(v_i,v_j)-v_i$, then in the expression for $a(v,w)$ in $\Delta_{\leaves{T}}$, the subspace $T(v_i,v_j)$  of $|T|$ corresponding to $T^h(v_i,v_j)$ is instead the image under $\rho$ of the union of $(|T_h|)$ with  $T^h(v_i,v_j)$.
Because $T_h$ is prunable, the deviation $D(v_i,v_j)$ is equal to $D^h(v_i,v_j)$.
Each factor of $a(v,w^h)$ in $\Delta_{\leaves{T^h}}$ is equal to the corresponding factor of $a(v,w)$ in $\Delta_{\leaves{T}}$, so $a(v,w^h) = a(v,w)$.

\subsubsection*{Case 2}If the point $v$ is in $|T^h|$ and there is no leaf $w^h$ of $T^h$ with $\rho(|w^h|) = |w|$, 
then notice by Lemma \ref{lemma: sum is 1}, the sum over leaves of $T^h$ of $a(v,w) = 1$, and likewise for the sum over leaves of $T$.
Then the argument for Case 1 also shows that in this case $a(v,w) =0$.

\subsubsection*{Case 3}If the point $v$ is in $|T_h|$ and there is no leaf $w^h$ of $T^h$ with $\rho(|w^h|) = |w|$, 
then there is a unique leaf $w_h$ of $T_h$ so that $\rho(|w_h|)=|w|$. Now this case is analogous to the first case.

Let $T_h(v_i,v_j)$ be a subspace of $|T_h|$ whose deviation appears in the expression for $a(v,w_h)$ in $\Delta_{\leaves{T_h}}$.
If $|s(h)|$ is not a point in $T_h(v_i,v_j)-v_i$, then the deviation of the isomorphic image $T(v_i,v_j) = \rho(T_h(v_i,v_j))$ appears in the expression for $a(v,w)$ in $\Delta_{\leaves{T}}$.
On the other hand,
if $|s(h)|$ is a point in $T_h(v_i,v_j)-v_i$, then in the expression for $a(v,w)$ in $\Delta_{\leaves{T}}$, the subspace $T(v_i,v_j)$  of $|T|$ corresponding to $T_h(v_i,v_j)$ is instead the image under $\rho$ of the union of $(|T^h|)$ with  $T_h(v_i,v_j)$.
 As in the previous case, prunability of $T_h$ implies that the deviation $D(v_i,v_j)$ is equal to $D_h(v_i,v_j)$.
Each factor of $a(v,w_h)$ in $\Delta_{\leaves{T_h}}$ is equal to the corresponding factor of $a(v,w)$ in $\Delta_{\leaves{T}}$, so $a(v,w_h) = a(v,w)$.

\subsubsection*{Case 4} If the point $v$ is in $|T_h|$ and there is a leaf $w^h$ of $T^h$ with $\rho(|w^h|) = |w|$, then any path from $v$ to $|w|$ in $|T|$ must pass through the $v$-zero cell $|s(h)|$; let this $v$-zero cell be denoted $v_j$.

Then the formula for $a(v,w)$ in $\Delta_{\leaves{T}}$ may be written as
\[
a(v,w) = \left( \prod_{i=1}^{j-1}\frac{D(v_i, v_{i+1})}{1 - D(v_{i+1}, v_i) }\right)  \left(\prod_{i=j}^{n-1}\frac{D(v_i, v_{i+1})}{1-D(v_{i+1}, v_i) }\right)
\]
As in the previous cases, we have obvious equalities
\[\begin{array}{ll}
D(v_i,v_{i+1})=D^h(v_i,v_{i+1}),& i\ge j\\
D(v_{i+1},v_i)=D_h(v_{i+1},v_i),&i<j
\end{array}\]
because of an isomorphism of the corresponding subspaces. 

On the other hand, for $i<j$ the subspace
$T(v_i,v_{i+1})$ is equal to the union of the image under $\rho$ of $|T^h|$ and $T_h(v_i,v_{i+1})$; for $i\ge j$ the subspace $T(v_{i+1},v_i)$ is equal to the union of the image under $\rho$ of $|T_h|$ and $T^h(v_{i+1},v_{i})$. In both cases, prunability of $T_h$ implies that the corresponding deviations are equal to one another. Thus
\[
\prod_{i=1}^{j-1}\frac{D(v_i, v_{i+1})}{1 -D(v_{i+1}, v_i) } = 
\prod_{i=1}^{j-1}\frac{D_h(v_i, v_{i+1})}{1 -D_h(v_{i+1}, v_i) } = 
a(v,s(h)) \text{ in } \Delta_{\leaves{T_h}}
\]
and
\[
\prod_{i=j}^{n-1}\frac{D(v_i, v_{i+1})}{1-D(v_{i+1}, v_i) } =
\prod_{i=j}^{n-1}\frac{D^h(v_i, v_{i+1})}{1-D^h(v_{i+1}, v_i) } 
= a(|s(h)|, w) \text{ in } \Delta_{\leaves{T^h}}.
\]
This guarantees that a point $v$ in $T_h$ has the same barycentric coordinates when $\Delta_{\leaves{T_h}}$ is mapped to $\Delta_{\leaves{T}}$ as when the barycentric coordinates are computed in $\Delta_{\leaves{T}}$ directly.
\end{proof}

\begin{figure}[ht]
\includegraphics{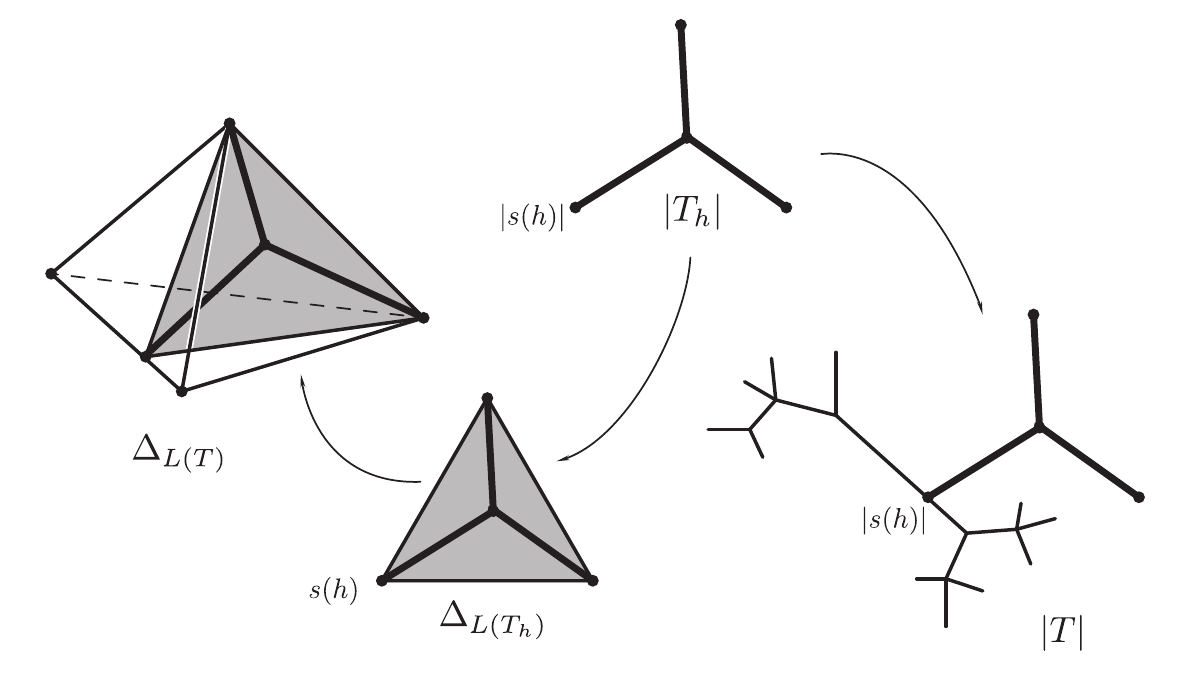}
\caption{The inclusion of $|T_h|$ into $|T|$, the straightening of $T_h$ in  $\Delta_{\leaves{T_h}}$, and the inclusion of  $\Delta_{\leaves{T_h}}$ into $\Delta_{\leaves{T}}$}
\label{fig:grafted straightening}
\end{figure}

\bibliography{dpr}

\providecommand{\bysame}{\leavevmode\hbox to3em{\hrulefill}\thinspace}
\providecommand{\MR}{\relax\ifhmode\unskip\space\fi MR }
\providecommand{\MRhref}[2]{%
  \href{http://www.ams.org/mathscinet-getitem?mr=#1}{#2}
}
\providecommand{\href}[2]{#2}
\begin{thebibliography}{God07b}

\bibitem[Afs11]{Afsari:RLPCMEUC}
B.~Afsari, \emph{Riemannian ${L}^{p}$ center of mass: Existence, uniqueness,
  and convexity}, Proc. Amer. Math. Soc. \textbf{139} (2011), 655--673.

\bibitem[B{\"o}d06]{Bodigheimer}
C.~F. B{\"o}digheimer, \emph{Configuration models for moduli spaces of
  {R}iemann surfaces with boundary}, Abh. Math. Sem. Univ. Hamburg \textbf{76}
  (2006), 191--233.

\bibitem[CG04]{CG}
R.~L. Cohen and V.~Godin, \emph{A polarized view of string topology}, Topology,
  geometry and quantum field theory, London Math. Soc. Lecture Note Ser., vol.
  308, Cambridge Univ. Press, Cambridge, 2004, pp.~127--154. \MR{2079373
  (2005m:55014)}

\bibitem[Cha05]{Chataur:BAST}
D.~Chataur, \emph{A bordism approach to string topology}, Int. Math. Res. Not.
  IMRN \textbf{2005} (2005), no.~46, 2829--2875.

\bibitem[CJ02]{CJ}
R.~L. Cohen and J.~D.~S Jones, \emph{A homotopy theoretic realization of string
  topology}, Math. Ann. \textbf{324} (2002), no.~4, 773--798. \MR{1942249
  (2004c:55019)}

\bibitem[Cos06]{Costello:ADPOVOTRGDOMS}
K.~Costello, \emph{A dual point of view on the ribbon graph decomposition of
  moduli space}, arXiv preprint, 2006, arXiv:math/0601130v1.

\bibitem[CS99]{CS}
M.~Chas and D.~Sullivan, \emph{String topology}, arXiv preprint, 1999,
  math.GT/9911159v1.

\bibitem[CS04]{ChasSullivan:CSOTLLBHSA}
\bysame, \emph{Closed string operators in topology leading to {L}ie bialgebras
  and higher string algebra}, The legacy of Niels Henrik Abel, Springer, 2004.

\bibitem[Dol95]{Dold:LAT}
A.~Dold, \emph{Lectures on algebraic topology}, Classics in Mathematics, vol.
  200, Springer Berlin Heidelberg, 1995.

\bibitem[EK]{EgasKupers:CCMFMSATC}
D.~Egas and A.~Kupers, \emph{Comparing combinatorial models for moduli space
  and their compactifications}, in preparation.

\bibitem[GK73]{GroveKarcher:HCC1CGA}
K.~Grove and H.~Karcher, \emph{How to conjugate ${C}^1$-close group actions},
  Math. Z. \textbf{132} (1973), 11--20.

\bibitem[God04]{Godin:CGMSMSBRSMSSC}
V.~Godin, \emph{Categorical graph models in the study of the moduli space of
  bordered {R}iemann surfaces and the moduli space of smooth curves}, PhD
  Thesis, 2004, Stanford University.

\bibitem[God07a]{godin}
\bysame, \emph{Higher string topology operations}, arXiv preprint, 2007,
  math.AT/0711.4859v2.

\bibitem[God07b]{Godin:TUSIHOTMCGOASWB}
\bysame, \emph{The unstable integral homology of the mapping class groups of a
  surface with boundary}, Math. Ann. \textbf{337} (2007), no.~1, 15--60.

\bibitem[Har88]{Harer:TCOTMSOC}
J.~L. Harer, \emph{The cohomology of the moduli space of curves}, Theory of
  Moduli (Montecatini Terme, 1985), Lecture Notes in Math., vol. 1337,
  Springer, Berlin, 1988, pp.~138--221.

\bibitem[Hat02]{HATCHER}
A.~Hatcher, \emph{Algebraic topology}, Cambridge U., 2002.

\bibitem[HW14]{HingstonWahl}
N.~Hingston and N.~Wahl, \emph{Compactified string topology}, Various public
  lectures, 2014.

\bibitem[Igu02]{Igusa:HFRT}
K.~Igusa, \emph{Higher {F}ranz--{R}eidemeister torsion}, AMS/IP Stud. Adv.
  Math., vol.~31, Amer. Math. Soc./International Press, Providence/Boston,
  2002.

\bibitem[Iri14]{Irie:TPSTDRC}
K.~Irie, \emph{Transversality problems in string topology and de {R}ham
  chains}, arXiv preprint, 2014, http://arxiv.org/abs/1404.0153.

\bibitem[Iri15]{Irie:CLBVSSTDC}
\bysame, \emph{A chain level {B}atalin--{V}ilkovisky structure in string
  topology and decorated cacti}, arXiv preprint, 2015,
  http://arxiv.org/abs/1503.00403.

\bibitem[Kar14]{Karcher:RCMSCKM}
H.~Karcher, \emph{Riemannian center of mass and so called {K}archer mean},
  arxiv:1407.2087, 2014.

\bibitem[Kon92]{Kontsevich:ITOTMSOCATAF}
M.~Kontsevich, \emph{Intersection theory on the moduli space of curves and the
  matrix {A}iry function}, Comm. Math. Phys. (1992), no.~147, 1--23.

\bibitem[Kup13]{Kupers:CHSOMURSC}
A.~Kupers, \emph{Constructing higher string operations for manifolds using
  radial slit configurations},
  \url{http://math.stanford.edu/~kupers/radialslitoperationsnew.pdf}, 2013.

\bibitem[Mal11]{Malm:STBLS}
E.~Malm, \emph{String topology and the based loop space}, arXiv preprint, 2011,
  http://arxiv.org/abs/1103.6198.

\bibitem[Pen87]{Penner:TDTMSOPS}
R.~C. Penner, \emph{The decorated {T}eichm\"uller space of punctured surfaces},
  Comm. Math. Phys. \textbf{113} (1987), no.~2, 299--339.

\bibitem[Poi10]{PoirierThesis}
K.~Poirier, \emph{String topology and compactified moduli spaces}, PhD Thesis,
  2010, City University of New York.

\bibitem[PR11]{PR1}
K.~Poirier and N.~Rounds, \emph{Compactifying string topology}, arXiv preprint,
  2011, arXiv:1111.3635v1.

\bibitem[San12]{Sander}
O.~Sander, \emph{Geodesic finite elements on simplicial grids}, Int. J. Numer.
  Meth. Engng. \textbf{92} (2012), no.~12, 999--1025.

\bibitem[Str84]{Strebel:QD}
K.~Strebel, \emph{Quadratic differentials}, vol.~5, Springer-Verlag, 1984.

\bibitem[Sul05]{Sullivan:STBAPS}
D.~Sullivan, \emph{String topology background and present state}, Current
  developments in mathematics (2005), 41--88. \MR{2459297 (2010c:55007)}

\bibitem[Tam09]{Tamanoi:SSOAT}
Hirotaka Tamanoi, \emph{Stable string operations are trivial}, Int. Math. Res.
  Not. IMRN (2009), no.~24, 4642--4685. \MR{2564371 (2010k:55015)}

\bibitem[Tam10]{Tamanoi:LCSTTHGTQFTO}
H.~Tamanoi, \emph{Loop coproducts in string topology and triviality of higher
  genus {TQFT} operations}, J. Pure Appl. Algebra \textbf{214} (2010), no.~5,
  605--615.

\bibitem[TZ06]{TZ}
T.~Tradler and M.~Zeinalian, \emph{Algebraic string operations}, arXiv
  preprint, 2006, math.QA/0605770v1.

\end{thebibliography}
\bibliographystyle{amsalpha}
\end{document}